\documentclass[11pt]{amsart}

\pdfoutput=1

\usepackage[text={420pt,660pt},centering]{geometry}

\usepackage{esint,amssymb} 
\usepackage{graphicx}
\usepackage{MnSymbol}
\usepackage{mathtools} 
\usepackage[colorlinks=true, pdfstartview=FitV, linkcolor=blue, citecolor=blue, urlcolor=blue,pagebackref=false]{hyperref}
\usepackage{microtype}

\usepackage{bm}
\usepackage{dsfont}

\newtheorem{proposition}{Proposition}
\newtheorem{theorem}[proposition]{Theorem}
\newtheorem{lemma}[proposition]{Lemma}

\theoremstyle{remark}
\newtheorem{remark}[proposition]{Remark}

\theoremstyle{definition}
\newtheorem{definition}[proposition]{Definition}

\numberwithin{equation}{section}
\numberwithin{proposition}{section}
\numberwithin{figure}{section}
\numberwithin{table}{section}
\newcommand{\Z}{\mathbb{Z}}
\newcommand{\N}{\mathbb{N}}

\newcommand{\R}{\mathbb{R}}

\newcommand{\E}{\mathbb{E}}
\renewcommand{\P}{\mathbb{P}}

\newcommand{\ep}{\varepsilon}
\newcommand{\eps}{\varepsilon}

\renewcommand{\le}{\leqslant}
\renewcommand{\ge}{\geqslant}

\renewcommand{\subset}{\subseteq}
\newcommand{\la}{\left\langle}
\newcommand{\ra}{\right\rangle}

\newcommand{\Ll}{\left}
\newcommand{\Rr}{\right}
\renewcommand{\d}{\mathrm{d}}

\DeclareMathOperator{\tr}{tr}

\renewcommand{\bar}{\overline}
\renewcommand{\tilde}{\widetilde}
\newcommand{\td}{\widetilde}
\renewcommand{\hat}{\widehat}

\newcommand{\1}{\mathds{1}}

\newcommand{\mcl}{\mathcal}
\newcommand{\msf}{\mathsf}
\newcommand{\mfk}{\mathfrak}
\newcommand{\al}{\alpha}

\newcommand{\ga}{\gamma}
\newcommand{\de}{\delta}
\newcommand{\si}{\sigma}
\newcommand{\A}{{\mathcal{A}}}

\renewcommand{\H}{\mathsf{H}}
\newcommand{\M}{\mathbb{M}}

\newcommand{\dr}{\partial}
\newcommand{\n}{\mathbf{n}}

\newcommand{\vb}{\, \big \vert \, }

\begin{document}

\author[J.-C. Mourrat]{Jean-Christophe Mourrat}
\address[J.-C. Mourrat]{Courant Institute of Mathematical Sciences, New York University, New York, New York, USA, and CNRS, France}
\email{jcm777@nyu.edu}

\date{\today}

\title{Nonconvex interactions in mean-field spin glasses}

\begin{abstract}
We propose a conjecture for the limit free energy of mean-field spin glasses with a bipartite structure, and show that the conjectured limit is an upper bound. The conjectured limit is described in terms of the solution to an infinite-dimensional Hamilton-Jacobi equation. A fundamental difficulty of the problem is that the nonlinearity in this equation is not convex. We also question the possibility to characterize this conjectured limit in terms of a saddle-point problem.
\end{abstract}

\maketitle

%
%
%
%%%%%%%%%%%%%%%%%%%%%%%%%%%%
%%%%%%%%%%%%%%%%%%%%%%%%%%%%
%
%
%

%\begin{quotation}  %\label{}
%\begin{flushright}
%\emph{
%To the people who risk their lives\\
%to help those in distress.
%}
%\end{flushright}
%\end{quotation}

%\vspace{0.2cm}

\section{Introduction}
\label{s.intro}

%\begin{quote}
%Many more facts which look like striking coincidences will occur. Of course, the author does not believe that they are mere coincidences, but rather that there is some underlying structure yet to be understood. 
%\end{quote}

Let $(J_{ij})_{i,j \ge 1}$ be independent standard Gaussian random variables, and, for every $\si  = (\sigma_{1,1},\ldots,\si_{1,N}, \si_{2,1}, \ldots,\si_{2,N})\in \R^{2N}$, let
\begin{equation}  
\label{e.def.HN}
H_N(\sigma) := N^{-\frac 1 2 } \sum_{i,j = 1}^N J_{ij} \,  \sigma_{1,i} \, \si_{2,j}.
\end{equation}
The main goal of this paper is to study the large-$N$ behavior of the free energy
\begin{equation}  
\label{e.really.vanilla.FN}
\frac 1 N \E \log \int_{\R^{2N}} \exp \Ll( \beta H_N(\si) \Rr) \, \d P_N(\si),
\end{equation}
where $\beta \ge 0$ and $P_N$ is a ``simple'' probability measure over $\R^{2N}$. For convenience, we assume that there exist two probability measures $\pi_1$ and $\pi_2$ on $\R$ with compact support such that, for every $N \ge 1$,
\begin{equation}
\label{e.ass.P}
P_N = \pi_1^{\otimes N} \otimes \pi_2^{\otimes N}.
\end{equation}
Without loss of generality, we assume that the supports of $\pi_1$ and $\pi_2$ are subsets of $[-1,1]$. 
%
%\begin{equation}
%\label{e.ass.P}
%\Ll|
%\begin{aligned}
%& \mbox{either }  \mbox{$P_N = P_1^{\otimes N} \otimes P_2^{\otimes N}$ for every $N \ge 1$}, \\
%& \mbox{or }  \mbox{$P_N$ is the uniform measure on $\{\sigma \in \R^{2N} \ : \ |\sigma|^2 = 2N\}$ for every $N \ge 1$}.
%\end{aligned}
%\Rr.
%\end{equation}
For every metric space~$E$, we denote by $\mcl P(E)$ the space of Borel probability measures on $E$, and, for every $p \in [1,\infty]$, by $\mcl P_p(E)$ the subspace of $\mcl P(E)$ of probability measures with finite $p$-th moment. We write $\de_x$ for the Dirac probability measure at $x \in E$. For every $\nu \in \mcl P(\R_+)$ and $r \in [0,1]$, we define
\begin{equation}
\label{e.def.F-1}
F^{-1}_\nu(r) := \inf \Ll\{ s \ge 0 \ : \ \nu \Ll( [0,s] \Rr)  \ge r \Rr\} ,
\end{equation}
%We may identify the space $(\mcl P(\R_+))^2$ with a particular subspace of $\mcl P(\R_+^2)$ by setting, for every $\mu = (\mu_1, \mu_2) \in (\mcl P(\R_+))^2$,
%\begin{equation*}  %\label{e.}
%\hat \mu := (F_{\mu_1}^{-1}, F_{\mu_2}^{-1})\Ll(\mathrm{Leb}_{ {\Ll[0,1\Rr]}}\Rr),
%\end{equation*}
and, for $U$ a uniform random variable over $[0,1]$, we write
\begin{equation}
\label{e.def.Xnu}
X_\nu := F_{\nu}^{-1}(U).
\end{equation}
Recall that the law of $X_\nu$ is $\nu$, and that this construction provides us with a joint coupling of all probability measures over $\R_+$. 
For every $\mu = (\mu_1, \mu_2) \in (\mcl P(\R_+))^2$, we denote by $\hat \mu \in \mcl P(\R_+^2)$ the law of the pair $(X_{\mu_1}, X_{\mu_2})$. Here is the main result of this paper.
\begin{theorem}
\label{t.main}
For every $t \ge 0$, we have
\begin{equation}  
\label{e.main}
\liminf_{N \to \infty} -\frac 1 N \E \log \int \exp \Ll( \sqrt{2t} H_N(\sigma) - N^{-1}  t |\si_1|^2 \, |\si_2|^2  \Rr) \, \d P_N(\si) \ge f(t,(\de_0,\de_0)),
\end{equation}
where $f = f(t,\mu) : \R_+ \times (\mcl P_2(\R_+))^2 \to \R$ is the solution to 
\begin{equation}  
\label{e.hj}
\Ll\{
\begin{aligned}
& \dr_t f - \int \dr_{\mu_1} f \, \dr_{\mu_2} f \ \d \hat \mu  = 0 & \quad \text{on } \ \R_+ \times (\mcl P_2(\R_+))^2,
\\
& f(0,\cdot ) = \psi & \quad \text{on } \ (\mcl P_2(\R_+))^2,
\end{aligned}
\Rr.
\end{equation}
and the initial condition $\psi$ is defined below in \eqref{e.def.psi}.
\end{theorem}
We start by clarifying the meaning of the Hamilton-Jacobi equation in \eqref{e.hj}. Alternative expressions for the integral in \eqref{e.hj} read
\begin{align*}  %\label{e.}
\int \dr_{\mu_1} f \, \dr_{\mu_2} f \ \d \hat \mu 
& = \int_{\R_+^2} \dr_{\mu_1} f(t,\mu,x_1) \, \dr_{\mu_2} f(t,\mu,x_2) \ \d \hat \mu(x_1, x_2)  
\\
& = \E \Ll[ \dr_{\mu_1} f(t,\mu,X_{\mu_1}) \, \dr_{\mu_2} f(t,\mu,X_{\mu_2})\Rr].
\end{align*}
The notion of derivative at play here is not of Fr\'echet type (which would express the linear response to the addition of a small signed measure of zero total mass), but rather of transport type. 
Informally, for a ``smooth'' function $g = g(\nu) : \mcl P_2(\R_+) \to \R$, the derivative $\dr_\nu g(\nu,\cdot) \in L^2(\R_+, \nu)$ is characterized by the first-order expansion
\begin{equation*}  %\label{e.}
g(\nu') = g(\nu) + \E \Ll[ \dr_\nu g(\nu,X_\nu)(X_{\nu'} - X_\nu) \Rr]  + o \Ll( \E \Ll[ (X_{\nu'} - X_\nu)^2 \Rr] ^\frac 1 2  \Rr) .
\end{equation*}
More concretely, given some integer $k \ge 1$, and setting, for every $x = (x_1, \ldots, x_k) \in \R_+^k$ such that $x_1 \le \cdots \le x_k$,
\begin{equation*}  %\label{e.}
g^{(k)}(x_1, \ldots,x_k) := g \Ll( \frac 1 k \sum_{\ell = 1}^k \de_{x_k} \Rr) ,
\end{equation*}
we have, for every $n \in \{1,\ldots, k\}$,
\begin{equation*}  %\label{e.}
\dr_\nu g\Ll( \frac 1 k \sum_{\ell = 1}^k \de_{x_k} , x_n \Rr) = k \, \dr_{x_n} g^{(k)}(x).
\end{equation*}
This suggests natural finite-dimensional approximations of the equation \eqref{e.hj}. Denoting
\begin{equation*}  %\label{e.}
\bar U_k := \big\{ q = (q_{1,1} , q_{1,2}, \ldots, q_{1,k}, q_{2,1}, \ldots, q_{2,k}) \in \R_+^{2k} \ : \ \forall a \in \{1,2\},  \ q_{a,1} \le \cdots \le q_{a,k} \big\},
\end{equation*}
These approximations take the form
\begin{equation}  
\label{e.finite.dim}
\dr_t f^{(k)} - k \sum_{\ell = 1}^k \dr_{q_{1,\ell}} f^{(k)} \, \dr_{q_{2,\ell}} f^{(k)} = 0 \qquad \text{on } \  \R_+ \times \bar U_k.
\end{equation}
We will define the solution to \eqref{e.hj} as the limit of such finite-dimensional approximations.\footnote{It may seem somewhat contrived to impose the ordering of the variables $q_{a,1} \le \cdots \le q_{a,k}$. However, in the proof of Theorem~\ref{t.main}, this formulation will allow for a clearer treatment of the boundary condition on the ``diagonal part'', i.e.\ whenever $q_{a,\ell} = q_{a,\ell+1}$ for some $\ell \in \{1,\ldots, k-1\}$. (This point was overlooked in a preliminary version of the paper.) Moreover, in more general models, the relevant variables are matrix-valued, and there is no simple ``symmetrization'' of an ordered tuple of symmetric matrices, so there is no way around working with a set of the form of $\bar U_k$ in this more general setting.}

That there exists a connection between the free energy of spin glass models and certain infinite-dimensional Hamilton-Jacobi equations was first observed in the context of mixed $p$-spin models \cite{parisi}. In these models, the energy function $H_N$ is a centered Gaussian field such that the covariance between $H_N(\si)$ and $H_N(\tau)$ is proportional to $\xi(\si \cdot \tau/N)$, where the function $\xi$ is fixed and can be written in the form $\xi(r) = \sum_{p \ge 2} \beta_p r^p$, for some family of coefficients $\beta_p \ge 0$ that decays sufficiently fast. (The constraint $\beta_p \ge 0$ is necessary and sufficient in order for $\xi$ to define a covariance kernel for every $N$~\cite{schoenberg}.) For these models, the corresponding Hamilton-Jacobi equation takes the form
\begin{equation}  
\label{e.hj.xi}
\dr_t f - \int \xi(\dr_\mu f) \, \d \mu = 0 \qquad \text{on } \R_+ \times \mcl P_2(\R_+). 
\end{equation}
With this in mind, it is natural to distinguish between three increasingly large classes of models. The first is the class of models for which the mapping $\xi$ is convex over $\R$; roughly speaking, these are the models whose limit free energy can be identified using the methods of \cite{gue03, Tpaper,Tbook1,Tbook2} (in fact, the precise condition is slightly more restrictive, see \cite[(14.101)]{Tbook2}). An extension of this approach, developed in \cite{pan.aom, pan}, allows to cover all mixed $p$-spin models. The convexity property, once properly understood, is still fundamental in this setting. More precisely, one can check that the relevant solution to \eqref{e.hj.xi} satisfies $\dr_\mu f \ge 0$. On the other hand, in view of the form of $\xi$, this function is convex over $\R_+$. In other words, we can redefine the function~$\xi$ to be $+\infty$ over $(-\infty,0)$; with this new definition, the relevant Hamilton-Jacobi equation is still \eqref{e.hj.xi}, and now the convexity of the nonlinearity has been restored. This convexity is crucial to the validity of a Hopf-Lax formula for the solution, and this variational formula forms the basis of the arguments for identifying the limit free energy in these approaches. 

The third class of models corresponds to situations in which the nonlinearity in the Hamilton-Jacobi equation may be genuinely nonconvex; a representative example in this class is the focus of the present paper. In this case, it is unclear whether the limit free energy can be described as a (reasonable) variational problem. The classical Hopf-Lax variational formula requires that the nonlinearity in the equation be convex (or concave), which it is clearly not in our setting. 
%That the nonlinearity in~\eqref{e.hj} is not convex is clear, since the simple mapping $(x,y) \mapsto xy$ is not convex. 
Alternatively, irrespectively of the structure of the nonlinearity, the solution to a Hamilton-Jacobi equation can always be written as a saddle-point problem, provided that the initial condition is concave (or convex) \cite{hopf,bareva,lioroc}. This motivates to study the concavity of the initial condition in~\eqref{e.hj}, that is, the function~$\psi$ in~\eqref{e.hj}. In the context of mixed $p$-spin models, the main result of \cite{aufche} implies the concavity of this function. I do not know whether this argument can be generalized to cover the bipartite model investigated here. But in any case, this does not seem to be the appropriate notion of concavity to guarantee the validity of a saddle-point formulation for the solution to~\eqref{e.hj}. In order for this to work, we would need instead that the function $\psi$ be \emph{transport-concave} (one may also say ``displacement-concave''); but we will see that this is not so in general. At present, my impression is that it is \emph{not} possible to express the limit free energy as a saddle-point problem in general, and that it would be very difficult to circumvent a description of this limit involving Hamilton-Jacobi equations. 

We now discuss the intuition behind Theorem~\ref{t.main}. The simplest setting in which to explain the idea is that of the Curie-Weiss model, see for instance \cite{HJinfer}. The main point is to enrich the model to include ``non-interacting'' terms in the energy function, with the hope that, if these simpler terms are sufficiently ``expressive'', then certain asymptotic relations between the derivatives of the free energy will have to be satisfied. In our context, a first attempt is to try to compare $\sum J_{ij} \si_{1,i} \si_{2,j}$ with a linear combination of $z_1 \cdot \si_1$ and $z_2 \cdot \si_2$, where $z = (z_1,z_2) = (z_{1,1}, \ldots, z_{1,N}, z_{2,1}, \ldots, z_{2,N})$ is a vector of independent standard Gaussians. In other words, we consider, for every $t,p_1,p_2 \ge 0$,  the free energy
\begin{multline}  
\label{e.def.GN}
G_N(t,p_1,p_2)  :=
 \\
 -\frac 1 N \E \log \int \exp \Ll( \sqrt{2t} H_N(\sigma) - N^{-1}  t |\si_1|^2 \, |\si_2|^2   + \sum_{a = 1}^2 \Ll(\sqrt{2p_a} \, z_a \cdot \si_a - p_a |\si_a|^2\Rr) \Rr)\, \d P_N(\si) .
\end{multline}
(Parametrizations of the form $\sqrt{t} X$ where $X$ is a Gaussian random variable are of course natural: think of Brownian motion. Each random variable in the exponential comes with a compensating term, so that the expectation of the exponential is equal to 1.)
Denoting by $\la \cdot \ra$ the expectation with respect to the Gibbs measure proportional to $\exp (\cdots) \, \d P_N(\si)$, one can check that
\begin{equation*}  %\label{e.}
\dr_t G_N = N^{-2} \, \E \la (\si_{1} \cdot \si_1')(\si_2 \cdot \si_2') \ra,
\end{equation*}
where $\si'$ denotes an independent copy of $\si$ under $\la \cdot \ra$. On the other hand, 
\begin{equation*}  %\label{e.}
\dr_{p_a} G_N = N^{-1} \E \la \si_a \cdot \si_a' \ra \qquad (a \in \{1,2\}),
\end{equation*}
so that 
\begin{equation}  
\label{e.naive.pde}
\dr_t G_N - \dr_{p_1} G_N \, \dr_{p_2} G_N = N^{-2} \, \E \la \Ll( \si_1 \cdot \si_1' - \E \la \si_1 \cdot \si_1' \ra \Rr) \Ll( \si_2 \cdot \si_2' - \E \la \si_2 \cdot \si_2' \ra \Rr) \ra.
\end{equation}
Hence, if the overlaps $\si_a \cdot \si_a'$ were concentrated, we would then infer that $G_N$ converges to $g = g(t,p_1,p_2) : \R_+^3 \to \R$ solution to
\begin{equation}  
\label{e.naive}
\dr_t g - \dr_{p_1} g \, \dr_{p_2} g = 0.
\end{equation}
However, as is well-known, the concentration of the overlaps is only valid in a high-temperature (that is, small $t$) region; a more refined enriched system is necessary to ``close the equation'' in general. The formal manipulation allowing to obtain the true equation from the ``naive'' (or replica-symmetric) one given in \eqref{e.naive} consists simply in replacing the variables $(p_1,p_2)$ encoding the strength of the extraneous random magnetic field by probability measures on $\R_+$, thus leading to the equation in \eqref{e.hj}. Intuitively, the reason why this makes sense is as follows. In the term $\sqrt{2p_a} z_a \cdot \si_a$, the magnetic field acting on~$\si_a$ has a ``trivial'' structure. However, we need to have access to a richer term that allows to represent extraneous magnetic fields with an ultrametric structure, and this structure is described by its overlap distribution, a probability measure on $\R_+$. This construction, explained precisely below, defines an enriched free energy $\bar F_N = \bar F_N(t,\mu_1,\mu_2) : \R_+ \times (\mcl P(\R_+))^2 \to \R$, and we will show that this enriched free energy is asymptotically bounded from below by the solution to \eqref{e.hj}; see Theorem~\ref{t.main.full} for a precise statement. As will be seen in the next section, the corresponding enriched Gibbs measure features extraneous variables, denoted $\alpha$, which are in correspondence with the overlap structure of the random magnetic fields. A crucial step of the argument consists in showing that ``typically'', the overlaps $\si_a \cdot \si_a'$ can be inferred from the knowledge of the overlap between $\alpha$ and $\alpha'$.

We now discuss related works. Fundamental insights on spin glasses, most notably the ultrametricity property, were first identified in the physics literature \cite{parisi79,parisi80,MPV}, where variational formulas for limit free energies were predicted.
These predictions were then proved rigorously in \cite{gue03, Tpaper,Tbook1,Tbook2} in the setting of mixed $p$-spin models discussed above, under the assumption that the function $\xi$ is convex over $\R$. The extension to the case of general $\xi$ was achieved in \cite{pan.aom,pan}, and relies in particular on the justification that ``typical'' Gibbs measures are indeed organized along an asymptotically ultrametric structure. Further studies of particular relevance to the current paper concern the synchronization property, for models with multiple types of spins, or vector-valued spins \cite{pan.multi,pan.potts,pan.vec}. Earlier works on spin-glass models with spins of multiple types include \cite{tal.genform,bovklim,barramulti,barcon, aufchebi, bailee}. 
 
Heuristic connections between limit free energies and partial differential equations were first pointed out in \cite{gue01, barra1, abarra, barra2}, under a replica-symmetric or one-step replica symmetry breaking assumption. A rigorous identification of limit free energies of disordered systems in terms of Hamilton-Jacobi equations was obtained in \cite{HJinfer, HJrank,HB1,HBJ}, in the context of problems of statistical inference. In this latter context, particular properties of the models allow to ``close the equation'' using only a finite number of additional variables; in other words, the Hamilton-Jacobi equations appearing there are finite-dimensional. The relevant partial differential equation for mixed $p$-spin models, namely \eqref{e.hj.xi}, was then identified in \cite{parisi}; an extension of this convergence, valid for the relevant enriched free energy, was conjectured there, and then proved in \cite{HJsoft}. This last reference also describes how to ``remove'' compensating terms such as the term $N^{-1} t |\si_1|^2 |\si_2|^2$ appearing in \eqref{e.main}, so that we can indeed end up with an upper bound on the limit of \eqref{e.really.vanilla.FN}.

The rest of the paper is organized as follows. In Section~\ref{s.defs}, we define the enriched free energy, record some of its basic properties, and state a generalized version of Theorem~\ref{t.main}, see Theorem~\ref{t.main.full}. In Section~\ref{s.visc}, we define the precise notion of viscosity solution for \eqref{e.finite.dim}, and define the solution to \eqref{e.hj} as the limit of such finite-dimensional solutions. In Section~\ref{s.serious}, we show that if we restrict the free energy to measures that are sums of $k$ Dirac masses with equal weights, then the function we obtain is a supersolution to \eqref{e.finite.dim}, up to an error that goes to $0$ as $k$ goes to infinity; this allows us to conclude the proof of Theorem~\ref{t.main.full} (and thus also of Theorem~\ref{t.main}). A crucial ingredient used in Section~\ref{s.serious} is the fact that overlaps synchronize, and the justification of this is deferred to Section~\ref{s.synch}. In this section, we revisit the synchronization results of~\cite{pan.multi}, emphasizing the notion of monotone couplings, and giving a ``finitary'' version of the statement of asymptotic synchronization. Finally, in Section~\ref{s.concave}, we discuss possible attempts at writing the solution to \eqref{e.hj} as a saddle-point problem, and show that these tentative formulas are invalid. The appendix collects a handful of basic results on Gaussian integrals.

%
%
%
%%%%%%%%%%%%%%%%%%%%%%%%%%%%
%%%%%%%%%%%%%%%%%%%%%%%%%%%%
%
%
%

\section{Definitions and basic properties}
\label{s.defs}

%\begin{equation}
%\label{e.normalized.integral}
%\fint_U := |U|^{-1} \int_U
%\end{equation}

%In the normalization implied by \eqref{e.ass.P}, we will typically handle vectors $\sigma \in \R^{2N}$ such that each of their coordinates is of order one. While this has a number of advantages, it is also convenient at times to rescale the vector so that its norm is of order one. In order to conveniently go back and forth between the two conventions and yet keep relatively compact notation, we define, for every $\sigma \in \R^{2N}$,
%\begin{equation*}  %\label{e.}
%\hsi := N^{-\frac 1 2} \si. 
%\end{equation*}
We write $\N = \{0,1,\ldots\}$ to denote the set of natural numbers, $\N_* := \N \setminus \{0\}$, and $\R_+ := [0,\infty)$.
For every $x,y \in \R^N$, we write
\begin{equation*}  %\label{e.}
x \cdot y := \sum_{i = 1}^N x_i y_i, \qquad |x|^2 = x \cdot x.
\end{equation*}
We always implicitly understand that a vector $\sigma \in \R^{2N}$ is indexed according to $\sigma = (\sigma_{1,1}, \ldots, \sigma_{1,N}, \sigma_{2,1}, \ldots, \sigma_{2,N})$. We recall that $H_N(\si)$ was defined in \eqref{e.def.HN}, and notice that, for every $\sigma, \si' \in \R^{2N}$,
\begin{align}  %\label{e.}
\notag
\E \Ll[ H_N(\sigma) H_N(\si') \Rr] 
& = N^{-1} \sum_{i,j = 1}^N \si_{1,i} \si'_{1,i} \si_{2,j} \si'_{2,j}
\\
\label{e.correl.HN}
& = N^{-1} (\si_{1} \cdot \si'_1) (\si_2 \cdot \si'_2).
\end{align}
For every $t \ge 0$, we define
\begin{equation}  
\label{e.def.HNt}
H_N^t(\si) := \sqrt{2t} \,  H_N(\si) - N^{-1} t \, |\si_1|^2 |\si_2|^2.
\end{equation}
We are now going to introduce another energy function, parametrized by $\mu = (\mu_1, \mu_2) \in (\mcl P(\R_+))^2$. It is much more convenient to describe and to work with this object in the case when the measures are discrete, and then simply argue by continuity. We therefore give ourselves an integer $k \ge 0$, and parameters
\begin{equation}  
\label{e.def.zeta}
0 =  \zeta_{0}  < \zeta_1 \le \zeta_2 \le \cdots \le \zeta_{k-1} \le \zeta_{k} < \zeta_{k+1} = 1,
\end{equation}
\begin{equation}
\label{e.def.q}
0 = q_{a,-1} \le q_{a,0} \le q_{a,1} \le \cdots \le q_{a,k} < q_{a,k+1} =  \infty \qquad (a \in \{1,2\}),
\end{equation}
%satisfying
%\begin{equation}  
%\label{e.constraint}
%\text{for every } \ell \in \{0,\ldots, k\}, \ \text{there exists } a \in \{1,2\} \ \text{such that } q_{a,\ell} < q_{a,\ell+1},
%\end{equation}
and we set, for every $a \in \{1,2\}$,
\begin{equation}  
\label{e.mu.diracs}
\mu_a = \sum_{\ell = 0}^{k} (\zeta_{\ell+1} - \zeta_{\ell}) \de_{q_{a,\ell}}.
\end{equation}
%One can check that every pair $(\mu_1,\mu_2)$ of probability measures with finite support can be represented in the form of \eqref{e.mu.diracs} with parameters satisfying \eqref{e.def.zeta}, \eqref{e.def.q} and \eqref{e.constraint}, and that the choice of such parameters is unique. We will soon verify that the extra constraint \eqref{e.constraint} need not preoccupy us: if it is dropped, then there will be several possible representations for the measure, but the quantities we will define do not depend on the choice of the representative. \jccomment{One could also just drop the constraint altogether, define the quantity as a function of the vector of parameters per se, and verify a posteriori that it only really depends on the measure.}
%
%
%We now introduce the Poisson-Dirichlet cascade with parameters in \eqref{e.def.zeta}. 
These measures will serve to parametrize certain ultrametric structures with a prescribed overlap distribution.
We instantiate the rooted tree with (countably) infinite degree and depth $k$ as
\begin{equation}  
\label{e.def.mclA}
\mcl A := \N^{0} \cup \N \cup \N^2 \cup \cdots \cup \N^k,
\end{equation}
where $\N^{0} = \{\emptyset\}$, and $\emptyset$ represents the root of the tree. For every $\alpha \in \N^\ell$, we write $|\alpha| := \ell$ to denote the depth of the vertex $\alpha$ in the tree $\mcl A$. For every leaf $\alpha = (n_1,\ldots,n_k)\in \N^k$ and $\ell \in \{0,\ldots, k\}$, we write
\begin{equation*}  %\label{e.}
\alpha_{| \ell} := (n_1,\ldots, n_\ell),
\end{equation*}
with the understanding that $\alpha_{| 0} = \emptyset$.
%we denote the set of predecessors of $\alpha$, excluding the root, by
%\begin{equation*}  %\label{e.}
%p(\alpha) := (n_1, (n_1,n_2),\ldots, (n_1,\ldots,n_k)).
%\end{equation*}
We also give ourselves a family $(z_{\al,a,i})_{\alpha \in \mcl A, a \in \{1,2\}, 1 \le i \le N}$ of independent standard Gaussians, independent of $H_N$, and we let $(v_\alpha)_{\al \in \N^k}$ be a Poisson-Dirichlet cascade with weights given by the family $(\zeta_\ell)_{1 \le \ell \le k}$. 
We refer to \cite[(2.46)]{pan} for a precise definition, and only mention here a few important points. First, in the case $k = 0$, we simply set $v_{\emptyset} = 1$. Second, in the case $k = 1$, the weights $(v_\alpha)_{\al \in \N}$ are obtained by normalizing a Poisson point process on $(0,\infty)$ with intensity measure~$\zeta_1 x^{-1-\zeta_1} \, \d x$ so that $\sum_{\alpha} v_\alpha = 1$. Third, for general~$k \ge 1$, the progeny of each non-leaf vertex at level $\ell \in \{0,\ldots, k-1\}$ is decorated with the values of an independent Poisson point process of intensity measure $\zeta_{\ell+1} x^{-1-\zeta_{\ell+1}} \, \d x$, then the weight of a given leaf $\alpha \in \N^k$ is calculated by taking the product of the ``decorations'' attached to each parent vertex, including the leaf vertex itself (but excluding the root, which has no assigned ``decoration''), and finally, these weights over leaves are normalized so that their total sum is~$1$. We take this Poisson-Dirichlet cascade $(v_\alpha)_{\alpha \in \N^k}$ to be independent of $H_N$ and of the random variables $(z_{\alpha,a,i})_{\al \in \mcl A, a \in \{1,2\}, 1 \le i \le N}$. For every $\sigma \in \R^{2N}$ and $\al \in \N^k$, we set
\begin{equation}
\label{e.def.HNalpha}
H_N^\mu(\sigma,\alpha) := \sum_{a= 1}^2 \Ll(\sum_{\ell = 0}^k \Ll(2q_{a,\ell} - 2q_{a,\ell-1}\Rr)^\frac 1 2 z_{\alpha_{|\ell},a} \cdot \sigma_a - q_{a,k}|\si_a|^2 \Rr),
\end{equation}
where we write $z_{\al_{|\ell},a} \cdot \sigma_a = \sum_{i = 1}^N z_{\al_{|\ell},a,i} \, \sigma_{a,i}$. The random variables $(H_N^\mu(\sigma,\alpha))_{\sigma \in {\R^{2N}}, \alpha \in \N^k}$ form a Gaussian family which is independent of $(H_N(\sigma))_{\sigma \in \R^{2N}}$. We understand that the symbol $\E$ stands for the expectation with respect to $(J_{ij})$, $(z_\al)_{\al \in \mcl A}$ and $(v_\al)_{\al\in \N^k}$. 
%and has covariance, for every $\si,\si' \in \R^{2N}$ and $\al,\al' \in \N^k$,
%\begin{equation*}  %\label{e.}
%\mathbb{C}\mathrm{ov} \Ll[ H_N^\mu(\sigma,\al') , H_N^\mu(\si',\al') \Rr] = \sum_{a = 1}^2 2q_{a,\alpha \wedge \al'} \  \sigma_a \cdot \si_a'  ,
%\end{equation*}
%where we write $\mathbb{C}\mathrm{ov}[X,Y] = \E[XY] - \E[X] \E[Y]$, and, for every $\alpha, \al' \in \N^k$, 
Notice that, for each fixed choice of $\al, \al' \in \N^k$, we have
\begin{equation}
\label{e.overlap.ext.field}
\frac 1 N\E \Ll[\Ll(\sum_{\ell = 0}^k \Ll(2q_{a,\ell} - 2q_{a,\ell-1}\Rr)^\frac 1 2 z_{\alpha_{|\ell},a} \Rr)\cdot \Ll(\sum_{\ell = 0}^k \Ll(2q_{a,\ell} - 2q_{a,\ell-1}\Rr)^\frac 1 2 z_{\alpha'_{|\ell},a} \Rr) \Rr]= 2 q_{a,\al \wedge \al'},
\end{equation}
where we write
\begin{equation}  
\label{e.def.wedge}
\alpha \wedge \al' := \sup \{\ell \le k \ : \ \al_{|\ell} = \al'_{|\ell} \}.
\end{equation}
The point of the construction in \eqref{e.def.HNalpha} is to provide with a more refined ``external field'' than that introduced in \eqref{e.def.GN}. Indeed, if we sample two independent copies $\al, \al' \in \N^k$ according to the weights $(v_\al)_{\al \in \N^k}$, then the law of overlap 
\begin{equation*}  %\label{e.}
\frac 1 N\Ll(\sum_{\ell = 0}^k \Ll(2q_{a,\ell} - 2q_{a,\ell-1}\Rr)^\frac 1 2 z_{\alpha_{|\ell},a} \Rr)\cdot \Ll(\sum_{\ell = 0}^k \Ll(2q_{a,\ell} - 2q_{a,\ell-1}\Rr)^\frac 1 2 z_{\alpha'_{|\ell},a} \Rr)
\end{equation*}
under the measure in which we average over $(z_\al)$ and $(v_\al)$ is $\mu_a$ (this can be inferred from Lemma~\ref{l.overlap.pdc} below or, more directly, from \cite[(2.34)]{pan}).
We define
\begin{equation}  
\label{e.def.FN}
F_N(t,\mu)  := -\frac 1 N \log \int \sum_{\al \in \N^k} \exp \Ll( H_N^t(\si) + H_N^\mu(\si,\al) \Rr) \, v_\al \d P_N(\si).
%\\
%& = -\frac 1 N \log \int \exp \Ll( H_N^t(\si) + H_N^\mu(\si,\al) \Rr) \, \d P_N(\si,\alpha),
\end{equation}
We also define the Gibbs measure $\la \cdot \ra$, with canonical random variable $(\si, \al)$ taking values in $\R^{2N} \times \N^k$, in such a way that, for any bounded measurable function $f$,
\begin{equation}  
\label{e.def.gibbs}
\la f(\si,\alpha) \ra := \exp \Ll( N F_N(t,\mu) \Rr) \int \sum_{\al \in \N^k} f(\si,\al)  \exp \Ll( H_N^t(\si) + H_N^\mu(\si,\al) \Rr) \, v_\al \d P_N(\si),
\end{equation}
We also allow ourselves to consider multiple independent copies, or ``replicas'', of the random variable $(\si,\alpha)$, which we may denote by $(\si',\al')$, $(\si'',\al'')$, and so on. Alternatively, in situations where many independent replicas need to be considered, we also denote these replicas by $(\si^{\ell},\al^{\ell})_{\ell \ge 1}$. Recall that the measure $\la \cdot \ra$ is itself random; while the replicas are independent under $\la \cdot \ra$, conditionally on the randomness ``extraneous'' to the measure, they are no longer independent after we average further.

We denote by $\td F_N$ the average of $F_N$ with respect to the random variables $(z_\alpha)$ and $(v_\alpha)$. Since the only additional source of randomness in the problem comes from the $J$'s in the definition of $H_N$, and since these are independent random variables, we can write
\begin{equation}  
\label{e.def.tdFN}
\td F_N(t,\mu) = \E \Ll[ F_N(t,\mu) \vb (H_N(\si))_{\si \in \R^{2N}} \Rr] .
\end{equation}
We also define the fully averaged free energy
\begin{equation}  
\label{e.def.barFN}
\bar F_N(t,\mu) := \E \Ll[ F_N(t,\mu) \Rr] .
\end{equation}
The notation just introduced suggests that these quantities depend on the parameters~$\zeta$ and $q$ in \eqref{e.def.zeta} and \eqref{e.def.q} only insofar as they affect the measures $\mu_1$ and $\mu_2$. The next proposition states that this is indeed the case, at least as far as the quantities $\td F_N(t,\mu)$ and $\bar F_N(t,\mu)$ are concerned. (It would make more sense to speak of distributional identities for $H_N^\mu$ and $F_N(t,\mu)$; since such considerations will not play any role in this paper, we simply accept a slightly abusive notation for these latter two quantities.) It also states that $\td F_N(t,\mu)$, and therefore also $\bar F_N(t,\mu)$, satisfy a Lipschitz estimate in their dependence in $\mu$. Recall that the random variables of $X_\nu$ appearing in the statement were defined in~\eqref{e.def.Xnu}.

\begin{proposition}[Lipschitz continuity of $\td F_N$]
\label{p.continuity}
The functions $\td F_N(t,\mu)$ and $\bar F_N(t,\mu)$ depend in the parameters $\zeta$ and $q$ in \eqref{e.def.zeta} and \eqref{e.def.q} only through their effect on the measures $(\mu_1,\mu_2)$ in \eqref{e.mu.diracs}. Moreover, for every $t \ge 0$ and any two pairs $\mu, \mu' \in (\mcl P(\R_+))^2$ of measures of finite support, we have
\begin{equation}  
\label{e.continuity}
|\td F_N(t,\mu) - \td F_N(t,\mu')| \le \sum_{a = 1}^2 \E \Ll[|X_{\mu_a} - X_{\mu'_a}| \Rr],
\end{equation}
and the same inequality also holds with $\td F_N$ replaced by $\bar F_N$. In particular, $\td F_N$ and $\bar F_N$ can be extended by continuity to $\R_+ \times (\mcl P_1(\R_+))^2$. 
\end{proposition}
One possible way to prove Proposition~\ref{p.continuity} is to rely on the following two results. The first one describes a relatively concrete procedure for computing averages over Poisson-Dirichlet cascades; see \cite[Theorem~2.9]{pan} for a proof.
\begin{proposition}[Integration of Poisson-Dirichlet cascades]
\label{p.ipdc1}
Let $(\omega_\alpha)_{\alpha \in \mcl A}$ be independent and identically distributed random variables taking values in some measurable space~$E$, independent of the Poisson-Dirichlet cascade $(v_\alpha)_{\al \in \N^k}$. Let $X_k : E^r \to \R$ be a measurable function, and denote
\begin{equation*}  %\label{e.}
X_{-1} := \E \log \sum_{\alpha \in \N^k} \exp \Ll( X_k(\omega_{\alpha_{|0}}, \ldots, \omega_{\alpha_{|k}}) \Rr) \, v_\alpha.
\end{equation*}
In the expression above, the expectation $\E$ is with respect to the law of $(\omega_\alpha)_{\al \in \A}$ and $(v_\alpha)_{\al \in \N^k}$.
Define recursively, for every $\ell \in \{1,\ldots,k\}$, the measurable function $X_{\ell-1} : E^{\ell-1} \to \R \cup \{+\infty\}$ given by
\begin{equation*}  %\label{e.}
X_{\ell-1}(\omega_0,\ldots,\omega_{\ell-1}) := \zeta_\ell^{-1} \log \E_{\omega_\ell} \exp \Ll( \zeta_\ell X_\ell(\omega_0,\ldots,\omega_\ell) \Rr) ,
\end{equation*}
where, for every $\ell \in \{0,\ldots, k\}$, we write $\E_{\omega_\ell}$ to denote the integration of the variable~$\omega_\ell$ along the law of any of the variables $(\omega_\alpha)_{\alpha \in \A}$.
We have
\begin{equation*}  %\label{e.}
X_{-1} = \E_{\omega_{0}}\Ll[ X_0(\omega_0) \Rr].
\end{equation*}
\end{proposition}
In the statement above, the random variables under each expectation are implicitly assumed to be integrable. In our context, we can apply this lemma in the following way: we set $\omega_\al := z_\al$ and, for every $y_0 = (y_{0,a,i})_{a \in \{1,2\}, 1 \le i \le N}$, \ldots, $y_k = (y_{k,a,i})_{a \in \{1,2\}, 1 \le i \le N} \in \R^{2N}$,
\begin{multline}  
\label{e.def.concrete.Xk}
X_k(y_0,\ldots,y_k) 
\\:= \log \int \exp \Ll( H_N^t(\si) + \sum_{a =1}^2 \Ll( \sum_{\ell = 0}^k \Ll(2q_{a,\ell} - 2q_{a,\ell-1}\Rr)^\frac 1 2 y_{\ell,a} \cdot \sigma_a - q_{a,k}|\si_a|^2  \Rr)  \Rr) \, \d P_N(\si).
\end{multline}
We then define recursively, for every $\ell \in \{1,\ldots, k\}$,
\begin{equation}  
\label{e.def.Xell}
X_{\ell-1}(y_0,\ldots,y_{\ell-1}) := \zeta_{\ell}^{-1} \log \E_{y_{\ell}} \exp \Ll( \zeta_{\ell} X_{\ell}(y_0,\ldots,y_{\ell}) \Rr) ,
\end{equation}
where, for every $\ell \in \{0,\ldots,k\}$, we write $\E_{y_{\ell}}$ to denote the integration of the variable $y_{\ell} \in \R^{2N}$ along the standard Gaussian measure.
Proposition~\ref{p.ipdc1} then ensures that
\begin{equation*}  %\label{e.}
-N \, \td F_N(t,\mu)  = \E_{y_0}\Ll[X_0(y_0)\Rr].
\end{equation*}
(A more careful argument would start by using Proposition~\ref{p.ipdc1} to verify that $|F_N(t,\mu)|$ is indeed integrable.)

The next lemma identifies the law of the overlap $\alpha \wedge \alpha'$ under the averaged measure~$\E \la \cdot \ra$. The proof can be found for instance in \cite[(2.82)]{pan} or \cite[Lemma~2.3]{parisi}. 
\begin{lemma}[overlaps of Poisson-Dirichlet cascades]
\label{l.overlap.pdc}
For every $\ell \in \{0,\ldots, k\}$, we have
\begin{equation*}  %\label{e.}
\E \la \1_{\{ \al \wedge \al' = \ell\}} \ra = \zeta_{\ell+1} - \zeta_\ell.
\end{equation*}
\end{lemma}
The combination of Proposition~\ref{p.ipdc1} and Lemma~\ref{l.overlap.pdc} allows to prove Proposition~\ref{p.continuity}, see for instance \cite[Proposition~2.1]{parisi}.  

While we usually think of $\td F_N$ and $\bar F_N$ as functions of the pair of measures $\mu$, we also allow ourselves to speak of $\dr_{q_{a,\ell}} \td F_N$ and $\dr_{q_{a,\ell}} \bar F_N$; this is meant to refer to the point of view in which these are seen as functions of the families of parameters $q$ and $\zeta$ in \eqref{e.def.q} and \eqref{e.def.zeta}. Another consequence of Proposition~\ref{p.ipdc1}, which can be found for instance in \cite[Proposition~14.3.2]{Tbook2} or \cite[Lemma~2.4]{parisi}, is that the derivatives of $\td F_N$ with respect to each of the parameters $q_{a,\ell}$ in \eqref{e.def.q} are nonnegative, and they increase with $\ell$ after suitable normalization. The precise statement is as follows.
\begin{lemma}
\label{l.pos.dr_q}
For every $a \in \{1,2\}$ and $\ell\le \ell' \in \{0,\ldots, k\}$, we have
\begin{equation}
\label{e.pos.dr_q}
\dr_{q_{a,\ell}} \td F_N \ge 0,
\end{equation}
and 
\begin{equation}
\label{e.order.dr_q}
(\zeta_{\ell+1} - \zeta_{\ell})^{-1} \dr_{q_{a,\ell}} \td F_N \le (\zeta_{\ell'+1} - \zeta_{\ell'})^{-1} \dr_{q_{a,\ell'}} \td F_N .
\end{equation}
%as well as
%\begin{equation}
%\label{e.comp.dr_q}
%(\zeta_{\ell'+1} - \zeta_{\ell'}) \dr_{q_{a,\ell}} \bar F_N\le (\zeta_{\ell+1} - \zeta_{\ell}) \dr_{q_{a,\ell'}} \bar F_N.
%\end{equation}
\end{lemma}
\begin{remark}  
Clearly, the statement of Lemma~\ref{l.pos.dr_q} is also valid with $\td F_N$ replaced by $\bar F_N$. 
It is part of the statement of this lemma that the quantity $(\zeta_{\ell+1} - \zeta_{\ell})^{-1} \dr_{q_{a,\ell}} \td F_N$ can be defined even when $\zeta_{\ell+1} = \zeta_{\ell}$, by continuity.
\end{remark}
Yet another consequence of Proposition~\ref{p.ipdc1} concerns the ``initial condition'' for $\bar F_N$. Under the assumption of \eqref{e.ass.P}, the verification that $\bar F_N(0,\mu)$ converges as $N$ tends to infinity is particularly simple.
\begin{lemma}[Initial condition for product measures]
\label{l.init} 
Recall that we assume \eqref{e.ass.P}. For every $N \ge 1$ and $\mu \in (\mcl P_1(\R_+))^2$, we have
\begin{equation*}  %\label{e.}
\bar F_N(0,\mu) = \bar F_1(0,\mu).
\end{equation*}
\end{lemma}
\begin{proof}
The argument can be found for instance in \cite[(2.60)]{pan}; we present it briefly here for the reader's convenience. 
When $t = 0$, and under the assumption of \eqref{e.ass.P}, the definition of $X_k$ given in \eqref{e.def.concrete.Xk} can be rewritten as
\begin{equation*}  %\label{e.}
X_k(y_0,\ldots, y_k) = \sum_{i = 1}^N \log \int \exp \Ll(\sum_{a =1}^2 \Ll( \sum_{\ell = 0}^k \Ll(2q_{a,\ell} - 2q_{a,\ell-1}\Rr)^\frac 1 2 y_{\ell,a,i}  \sigma_a - q_{a,k}\si_a^2  \Rr)  \Rr) \, \d P_1(\si).
\end{equation*}
Recall that $P_1$ is a probability measure over $\R^2$, so in the integral above, the variable $\si$ takes the form $\si = (\si_a)_{a \in \{1,2\}} \in \R^2$. In particular, we have written $X_k$ as a sum of independent and identically distributed random variables. Moreover, the law of each of these random variables does not depend on $N$. These properties are preserved as we go along the recursive procedure described in \eqref{e.def.Xell}. As we reach $X_{-1}$, all randomness has been integrated out, and the result is thus $N$ times some constant, as desired. 
\end{proof}
With an eye towards the initial condition in \eqref{e.hj}, we set, for every $\mu \in (\mcl P_1(\R_+))^2$,
\begin{equation}  
\label{e.def.psi}
\psi(\mu) := \bar F_1(0,\mu). 
\end{equation}
%We will also derive a possibly more explicit description for $\psi$ in Section~\ref{s.concave}, which features a second-order Hamilton-Jacobi equation; see Proposition~\ref{p.alt.FN} and Remark~\ref{r.psi.pde}. As far as I understand, there is no relationship between the first-order equation appearing in \eqref{e.hj} and the second-order equation involved in the description of the initial condition $\psi$. 
It is worth keeping in mind that the relatively simple definition of the initial condition in \eqref{e.def.psi} is possible only because we made the assumption in \eqref{e.ass.P} that the underlying measure has a product structure. In general, we only want to ascertain that for every $\mu \in \mcl (P_1(\R_+))^2$, 
\begin{equation*}  %\label{e.}
\lim_{N \to \infty} \bar F_N(0,\mu) \quad \text{exists} \, ;
\end{equation*}
and in this case, we call the limit $\psi(\mu)$.
(We also use in the course of the proof that the support of $P_N$ lies in a ball of fixed radius.) Other choices of reference measure are thus possible: for instance, one may replace $P_N$ by the uniform measure on the product of two $N$-dimensional spheres of radius $\sqrt N$. See for instance \cite[part (2) of Proposition~3.1]{parisi} for a similar calcuation in this case (which itself borrows from \cite{tal.sph}).

%The proof Lemma~\ref{l.pos.dr_q} is identical to that of \cite[Lemma~2.4]{parisi}, so we do not reproduce it here. As noted in \cite[Remark~2.5]{parisi}, the monotonicity in \eqref{e.comp.dr_q} can be restated as, for every $\ell \le \ell' \in \{0,\ldots, k\}$,
%\begin{equation*}  %\label{e.}
%\E \la \si_a \cdot \si_a' \vb \al \wedge \al' = \ell \ra \le \E \la \si_a \cdot \si_a' \vb \al \wedge \al' = \ell' \ra.
%\end{equation*}
%This is consistent with the idea that the overlaps $\si_a \cdot \si_a'$ and $\al \wedge \al'$ synchronize, or more generally that they will be approximately coupled monotonically.

We now state the extended version of Theorem~\ref{t.main} that will be the main focus of the rest of the paper. 
\begin{theorem}
\label{t.main.full}
For every $t \ge 0$ and $\mu \in (\mcl P_2(\R_+))^2$, we have
\begin{equation*}  %\label{e.}
\liminf_{N \to \infty} \bar F_N(t,\mu) \ge f(t,\mu),
\end{equation*}
where $f : \R_+ \times (\mcl P_2(\R_+))^2 \to \R$ is the solution to \eqref{e.hj}.
\end{theorem}
The statement of Theorem~\ref{t.main} corresponds to the case $\mu = (\de_0,\de_0)$ in Theorem~\ref{t.main.full}.
We now discuss why one should expect that $\bar F_N$ indeed converges to the solution to \eqref{e.hj}. We first observe that
\begin{equation}  
\label{e.drt.FN}
\dr_t F_N = \frac 1 N \la \frac{1}{\sqrt{2t}} H_N(\si) - \frac 1 N |\si_1|^2 |\si_2|^2 \ra.
\end{equation}
Taking the expectation, recalling \eqref{e.correl.HN}, and using a Gaussian integration by parts, see \eqref{e.gibp1}, we obtain that
\begin{equation}  
\label{e.drt.barFN}
\dr_t \bar F_N = \frac 1 {N^2} \E \la (\si_1 \cdot \si_1') (\si_2 \cdot \si_2') \ra.
\end{equation}
By the same reasoning (or see for instance \cite[(2.17)]{parisi}), we have
\begin{equation}  
\label{e.drq.barFN}
\dr_{q_{a,\ell}} \bar F_N =\frac 1 N \E \la \si_a \cdot \si_a' \1_{\{\al \wedge \al' = \ell\}} \ra.
\end{equation}
Using Lemma~\ref{l.overlap.pdc}, we can rewrite this identity as
\begin{equation*}  %\label{e.}
(\zeta_{\ell+1} - \zeta_\ell)^{-1} \dr_{q_{a,\ell}} \bar F_N =\frac 1 N \E \la \si_a \cdot \si_a' \vb \al \wedge \al' = \ell \ra,
\end{equation*}
where the conditional expectation is understood with respect to the measure $\E \la \cdot \ra$. We deduce that 
\begin{align*}  %\label{e.}
\int \dr_{\mu_1} \bar F_N \, \dr_{\mu_2} \bar F_N \ \d \hat \mu 
& = \sum_{\ell = 1}^k (\zeta_{\ell+1} - \zeta_{\ell})^{-1} \dr_{q_{1,\ell}} \bar F_N \, \dr_{q_{2,\ell}} \bar F_N 
\\
& = \frac 1 {N^2} \sum_{\ell = 1}^k \E \la \1_{\{\al \wedge \al' = \ell\}} \ra  \E \la \si_1 \cdot \si_1' \vb \al \wedge \al' = \ell \ra \E \la \si_2 \cdot \si_2' \vb \al \wedge \al' = \ell \ra  
\\
& = \frac 1 {N^2} \E \la \E \la \si_1 \cdot \si_1' \vb \al \wedge \al' \ra \E \la \si_2 \cdot \si_2' \vb \al \wedge \al' \ra  \ra.
\end{align*}
We can now compare this expression with \eqref{e.drt.barFN}, and also with the situation encountered in the more naive attempt leading to \eqref{e.naive.pde}. In the naive attempt, we could only hope to close the equation in situations for which the overlaps $\si_a \cdot \si_a'$ are concentrated. In our current more refined attempt, we have instead
\begin{multline*}  %\label{e.}
\dr_t \bar F_N - \int \dr_{\mu_1} \bar F_N \, \dr_{\mu_2} \bar F_N \ \d \hat \mu 
\\
= \frac{1}{N^2} \E \la \Ll( \si_1\cdot \si_1' - \E \la \si_1 \cdot \si_1' \vb \al \wedge \al' \ra  \Rr)\Ll( \si_2\cdot \si_2' - \E \la \si_2 \cdot \si_2' \vb \al \wedge \al' \ra  \Rr)  \ra,
\end{multline*}
and in particular,
\begin{equation}  
\label{e.approxN.hj}
\Ll| \dr_t \bar F_N - \int \dr_{\mu_1} \bar F_N \, \dr_{\mu_2} \bar F_N \ \d \hat \mu \Rr| \le \frac 1 {N^2} \sum_{a \in \{1,2\}} \E \la \Ll( \si_a\cdot \si_a' - \E \la \si_a \cdot \si_a' \vb \al \wedge \al' \ra  \Rr)^2 \ra.
\end{equation}
In other words, we need to argue that the \emph{conditional variance} of the overlaps $\si_a \cdot \si_a'$, given the overlap $\al \wedge \al'$, is small. 
This is precisely what the synchronization property should give us. (Moreover, there is some flexibility in that we do not need that this conditional variance be small for any single choice of the parameters.) From this point of view, the synchronization property becomes central even for models with a single type, since the point is to monitor synchronization with the extraneous random variables provided by the Poisson-Dirichlet cascade.

%
%
%
%%%%%%%%%%%%%%%%%%%%%%%%%%%%
%%%%%%%%%%%%%%%%%%%%%%%%%%%%
%
%
%

\section{Viscosity solutions}
\label{s.visc}

The first goal of this section is to clarify the exact notion of solution for finite-dimensional approximations to \eqref{e.hj}, and show comparison principles for these finite-dimensional problems. The second goal is to show that as we increase the dimension, the sequence of finite-dimensional solutions converges to some limit. We then interpret the limit as the solution to \eqref{e.hj}. 

In this context, the convex cone $\R_+^2$ plays a fundamental role. In more general models of mean-field spin glasses, this convex cone would have to be replaced by the set of positive semidefinite matrices. In the setting of the bipartite model, matrices are not obviously showing up because, in some sense, we are only looking at the diagonal entries of a $2$-by-$2$ matrix: observe that there is no term of the form $\sigma_1 \cdot\sigma_2'$ on the right side of~\eqref{e.correl.HN}. In order to avoid future repetitions, I found it useful to write this section so that it covers the two settings at once. Throughout this section, we keep the integer $D \ge 1$ fixed, and denote by $S^D$, $S^D_+$, and $S^D_{++}$ the set of symmetric $D$-by-$D$ matrices, and the subsets of positive semidefinite and positive definite matrices respectively. We define
\begin{equation}
\label{e.def.mclC}
\mcl C := (0,\infty)^D \qquad \text{ or } \qquad \mcl C := S^D_{++},
\end{equation}
its closure, given respectively by
\begin{equation}
\label{e.def.bar.mclC}
\bar{\mcl C} := \R_+^D \qquad \text{ or } \qquad \bar{\mcl C} := S^D_+,
\end{equation}
and observe that $\bar{\mcl C}$ is a convex cone within its natural ambient vector space, namely
\begin{equation}
\label{e.def.mclE}
\mcl E := \R^D \qquad \text{ or } \qquad \mcl E := S^D.
\end{equation}
When $\mcl E = S^D$, we interpret the scalar product between two matrices $a = (a_{dd'})_{1 \le d,d' \le D}$ $b = (b_{dd'})_{1 \le d,d' \le D} \in S^D$ according to 
\begin{equation*}  %\label{e.}
a \cdot b := \sum_{d,d' = 1}^D a_{dd'} b_{dd'} = \tr(a^*b),
\end{equation*}
with $a^*$ denoting the transpose of the matrix $a$. We also write $|a| := (a \cdot a)^\frac 1 2$. In both settings, the convex cone $\bar{\mcl C}$ defines a partial order: for every $x, y \in \mcl E$, we write $x \le y$ whenever $y-x \in \bar{\mcl C}$. We could also use the notation $x < y$ whenever $y-x \in \mcl C$, however I will refrain from doing so, in order to avoid possible confusions that would arise from the fact that the conjunction of $x \le y$ and $x \neq y$ does not imply $x < y$.

Let $K \ge 1$ be an integer. We define the open set
\begin{equation}
\label{e.def.Uk}
U_K := \Ll\{ x  = (x_1,\ldots, x_K) \in \mcl C^K \ : \ \forall k \in \{1,\ldots, K-1\}, \ x_{k+1} - x_k \in \mcl C \Rr\},
\end{equation}
and its closure
\begin{equation}
\label{e.def.barUk}
\bar U_K := \Ll\{ x  = (x_1,\ldots, x_K) \in \bar{\mcl C}^K \ : \ x_1 \le \cdots \le x_K  \Rr\}.
\end{equation}
The first goal of this section is to study the existence and uniqueness of solutions to the equation
\begin{equation}
\label{e.hj.finite.nobd}
\dr_t f - \msf H(\nabla f) = 0 \qquad \text{in } (0,T) \times U_K,
\end{equation}
for a given locally Lipschitz function $\msf H : \mcl E^K \to \R$, $T \in (0,\infty]$, and with a prescribed initial condition at $t = 0$. In the expression above, we use the notation, with the understanding that $f = f(t,x)$ with $x = (x_1,\ldots, x_K)$, 
\begin{equation*}  %\label{e.}
\nabla f := (\dr_{x_1} f, \ldots, \dr_{x_K} f),
\end{equation*}
where in this expression, each $\dr_{x_k} f$ takes values in the set $\mcl E$. We will also impose a Neumann boundary condition on $\dr U_K$ for solutions to \eqref{e.hj.finite.nobd}. Since the domain $\bar U_K$ has corners, we define the outer normal to a point $x \in \dr U_K$ as the set
\begin{equation*}  %\label{e.}
\n(x) := \Ll\{ \nu \in \mcl E^K \ : |\nu| = 1 \ \text{ and } \ \forall y \in \bar U_K, \ (y-x) \cdot \nu  \le 0\Rr\} .
\end{equation*}
(This definition would have to be modified for non-convex domains.)
To display the Neumann boundary condition, we write the equation formally as
\begin{equation}
\label{e.hj.finite}
\Ll\{
\begin{aligned}  %\label{}
& \dr_t f - \msf H(\nabla f) = 0 & \qquad & \text{in } (0,T) \times U_K, \\
& \n \cdot \nabla f = 0 & \qquad & \text{on } (0,T)\times \dr U_K.
\end{aligned}
\Rr.
\end{equation}
In order to study the equation~\eqref{e.hj.finite}, we rely on the notion of viscosity solutions. Although the techniques used here to handle the equation \eqref{e.hj.finite} do not differ much from classical arguments, I could not find results in the literature that would prove the well-posedness of viscosity solutions in non-smooth domains such as $\bar U_K$. The best result I could find is~\cite{dupish}, where the authors consider the case where the domain is the intersection of a finite number of open sets with a smooth boundary that satisfy certain conditions. 

Our main interest for studying solutions to \eqref{e.hj.finite} resides in the fact that we will then define the solution to \eqref{e.hj} as the limit of solutions to such finite-dimensional problems. Several other works have also considered Hamilton-Jacobi equations posed on spaces of probability measures or other infinite-dimensional spaces \cite{cl1,cl2,cl3, fenkat, fenkur, carqui, carsou, gan08, ganswi, amb14, car10}. However, I am not aware of results that show the well-posedness of equations of the type of \eqref{e.hj}; or that include the handling of a boundary condition; or that discuss the convergence of finite-dimensional approximations. These aspects will be covered here.

The remainder of this section is made of two parts. We first study finite-dimensional equations of the form \eqref{e.hj.finite}; and then show how to pass to the limit and identify the solution to \eqref{e.hj}.

\subsection{Analysis of finite-dimensional equations}
The precise definition of solution to the equation \eqref{e.hj.finite} reads as follows.

\begin{definition}  
\label{def.solution}
We say that a function $f \in C([0,T) \times \bar U_K)$ is a \emph{viscosity subsolution} to~\eqref{e.hj.finite} if for every $(t,x) \in (0,T)\times \bar U_K$ and $\phi \in C^\infty((0,T)\times  \bar U_K)$ such that $(t,x)$ is a local maximum of $f-\phi$, we have
\begin{equation}  
\label{e.interior.cond.subsol}
\Ll(\dr_t \phi - \H(\nabla \phi)\Rr)(t,x) \le 0 \quad \text{ if } x \in U_K,
\end{equation}
while, if $x \in \dr U_K$, 
\begin{equation}  
\label{e.bdy.cond.subsol}
\min \Ll( \inf_{\nu \in \n(x)} \nabla \phi \cdot \nu, \dr_t \phi - \H(\nabla \phi)\Rr) (t,x) \le 0.
\end{equation}
We say that a function $f \in C([0,T) \times \bar U_K)$ is a \emph{viscosity supersolution} to~\eqref{e.hj.finite} if for every $(t,x) \in (0,T)\times \bar U_K$ and $\phi \in C^\infty((0,T)\times \bar U_K)$ such that $(t,x)$ is a local minimum of $f-\phi$, we have 
\begin{equation*}  %\label{e.}
\Ll(\dr_t \phi - \H(\nabla \phi)\Rr)(t,x) \ge 0 \quad \text{ if } x \in U_K,
\end{equation*}
while, if $x \in \dr U_K$, 
\begin{equation}  
\label{e.bdy.cond.supersol}
\max \Ll( \sup_{\nu \in \n(x)} \nabla \phi \cdot \nu, \dr_t \phi - \H(\nabla \phi)\Rr) (t,x) \ge 0.
\end{equation}
We say that a function $f \in C([0,T) \times \bar U_K)$ is a \emph{viscosity solution} to~\eqref{e.hj.finite} if it is both a viscosity subsolution and a viscosity supersolution to \eqref{e.hj.finite}. 
\end{definition}
We may drop the qualifier \emph{viscosity} and simply talk about subsolutions, supersolutions, and solutions to~\eqref{e.hj.finite}. We say that a function $f \in C([0,T) \times \bar U_K)$ is a \emph{solution} to
\begin{equation}
\label{e.hj.finite.subsol}
\Ll\{
\begin{aligned}  %\label{}
& \dr_t f - \msf H(\nabla f) \le 0 & \qquad & \text{in } (0,T) \times U_K, \\
& \n \cdot \nabla f \le 0 & \qquad & \text{on } (0,T)\times \dr U_K,
\end{aligned}
\Rr.
\end{equation}
whenever it is a subsolution to \eqref{e.hj.finite}; and similarly with the inequalities reversed for supersolutions.

Historically, the notion of viscosity solutions emerged from the following construction of solutions: for a small parameter $\ep > 0$, one considers the solution of the partial differential equation 
\begin{equation*}  %\label{e.}
\Ll\{
\begin{aligned}  %\label{}
& \dr_t f_\ep - \H(\nabla f_\ep) = \ep \Delta f_\ep,
 & \qquad & \text{in } (0,T) \times U_K, \\
& \n \cdot \nabla f_\ep = 0 & \qquad & \text{on } (0,T)\times \dr U_K,
\end{aligned}
\Rr.
\end{equation*}
and then one identifies the viscosity solution to \eqref{e.hj.finite} as the limit of $f_\ep$ as $\ep$ tends to zero. As will be seen below, the limit satisfies a form of maximum principle, as each of these approximations do. One can also consult \cite[Section III.10.1]{evans} for more intuition concerning the definition of viscosity solutions.

The most useful result concerning solutions to \eqref{e.hj.finite} for our purposes is a comparison principle. For every $x \in \mcl E^K$, we write
\begin{equation*}  
%\label{e.def.rho.norm}
|x| := \Ll( \sum_{k = 1}^K |x_k|^2\Rr)^\frac 1 2,
\end{equation*}
where $|x_k|$ stands for the standard Euclidean norm in $\mcl E$, 
and, for every $r \in \R$, we write $r_+ := \max(r,0)$.
\begin{proposition}[Comparison principle]
\label{p.comp}
Let $T \in (0,\infty)$, and let $u$ and $v$ be respectively a sub- and a super-solution to \eqref{e.hj.finite} that are both uniformly Lipschitz continuous in the $x$ variable. 
We have
\begin{equation}  
\label{e.comp}
\sup_{\Ll[0,T\Rr) \times \bar U_K} (u-v) = \sup_{\{0\}\times \bar U_K} (u-v).
\end{equation}
More precisely, let
\begin{equation}  
\label{e.def.L}
L := \max(\|\, |\nabla u| \, \|_{L^\infty([0,T)\times U_K)}, \|\, |\nabla v| \, \|_{L^\infty([0,T)\times U_K)}),
\end{equation}
and, for some arbitrary $\delta > 0$, let
\begin{equation}  
\label{e.def.V}
V := \sup \Ll\{ \frac{|\H(p') - \H(p)|}{|p'-p|} \ : \ |p|, |p'| \le L + \delta \Rr\} .
\end{equation}
%For every $R \ge VT$, the mapping
%\begin{equation}  
%\label{e.comp.precise}
%\Ll\{
%\begin{array}{lll}  %\label{}
%[0,T) \times \R^k_+& \to & \R \\
%(t,x) & \mapsto & u(t,x) - v(t,x) - (2L+1) \Ll( |x| + Vt - R \Rr)_+
%\end{array}
%\Rr.
%\end{equation}
For every $R,M \in \R$ such that
\begin{equation}  
\label{e.growth.M}
M > \sup_{0 \le t < T, x \in \bar U_K} \frac{u(t,x) - v(t,x)}{1+|x|},
\end{equation}
the mapping
\begin{equation}  
\label{e.comp.precise}
(t,x)  \mapsto  u(t,x) - v(t,x) - M \Ll( |x| + Vt - R \Rr)_+
\end{equation}
achieves its supremum at a point in $\{0\} \times \bar U_K$.
\end{proposition}
Before turning to the proof of this proposition, it will be useful to identify the cone dual to the convex cone $\bar U_K$. 
\begin{lemma}[Dual cone to $\bar U_K$]
\label{l.dual.cone}
Let $\bar U_K^*$ denote the cone dual to $\bar U_K$, that is,
\begin{equation}  
\label{e.def.baruk*}
\bar U_K^* := \Ll\{ x \in \mcl E^K \ : \ \forall v\in \bar U_K, \ x \cdot v \ge 0\Rr\} .
\end{equation}
We have
\begin{equation}  
\label{e.ident.baruk*}
\bar U_K^* = \Ll\{ x \in \mcl E^K \ : \ \forall k \in \{ 1,\ldots, K\}, \ \sum_{\ell = k}^K x_\ell \ge 0 \Rr\} ,
\end{equation}
and
\begin{equation}  
\label{e.double.dual.uk}
\bar U_K = \Ll\{ v \in \mcl E^K \ : \ \forall x \in \bar U_K^*, \ x \cdot v \ge 0 \Rr\} .
\end{equation}
\end{lemma}
\begin{proof}
For concreteness, we write the proof in the case when $\mcl C = S^D_{++}$. 
Let $x,v \in (S^D)^K$. Setting $v_0 := 0$, we have
\begin{equation}  
\label{e.discr.ibp}
x \cdot v  = \sum_{k = 1}^K x_k \cdot v_k
 = \sum_{k = 1}^K \sum_{\ell = k}^K x_\ell \cdot (v_k - v_{k-1}).
\end{equation}
It is therefore clear that the set on the right side of \eqref{e.ident.baruk*} is contained in $\bar U_K^*$ (recall that if $a, b \in S^D_+$, then $a\cdot b = |\sqrt{a} \sqrt{b}|^2 \ge 0$). Conversely, if $x \in (S^D)^K$ does not belong to the set on the right side of \eqref{e.ident.baruk*}, then there exists $k \in \{1,\ldots, K\}$ such that 
\begin{equation*}  %\label{e.}
\sum_{\ell = k}^K x_\ell \notin S^D_+. 
\end{equation*}
Letting $\mathfrak{p}$ denote the orthogonal projection onto the eigenspaces with negative eigenvalues of the matrix on the left side of the display above, and setting
\begin{equation*}  %\label{e.}
v = \Big(\underbrace{0,\ldots, 0}_{k-1\text{ terms}}, \mathfrak{p},\ldots , \mathfrak{p}\Big) \in \bar U_K,
\end{equation*}
we find that $x \cdot v < 0$, so $x \notin \bar U_K^*$. This shows \eqref{e.ident.baruk*}. 

The proof of \eqref{e.double.dual.uk} could be derived from a general statement concerning the bidual of closed convex cones; see Step 1 of the proof of Proposition~\ref{p.boundary} below. We rather provide with a more elementary and explicit argument. It is clear from \eqref{e.discr.ibp} that $\bar U_K$ is contained in the set on the right side of \eqref{e.double.dual.uk}. Conversely, recall that a matrix $a \in S^D$ belongs to $S^D_+$ if and only if, for every $b \in S^D_+$, we have $a \cdot b \ge 0$ (see also \cite[Lemma~2.2]{HJrank}). Let $v$ belong to the set on the right side of~\eqref{e.double.dual.uk}. To see that $v \in \bar U_K$, it thus suffices to show that for every $b_1, \ldots, b_K \in S^D_+$, one can find $x_1, \ldots, x_K \in S^D$ such that, for every $k \in \{1,\ldots, K\}$, we have
\begin{equation*}  %\label{e.}
\sum_{\ell = k}^K x_\ell = b_k.
\end{equation*}
It suffices to set, for every $k \in \{1,\ldots, K\}$, $x_k := b_k - b_{k+1}$, with the notation $b_{K+1} := 0$.
\end{proof}

\begin{proof}[Proof of Proposition~\ref{p.comp}]
For concreteness, we write the proof in the case when $\mcl C = S^D_{++}$. The case of $\mcl C = (0,\infty)^D$ is only easier.

Since the second part of the statement implies the first part, we focus on the former. Without loss of generality, we assume that the functions $u$ and $v$ are continuous on $[0,T] \times \bar U_K$ (once the result is proved in this case, we can obtain the general case by approximating $T$ with a sequence that converges to $T$ increasingly). We argue by contradiction, and assume that the mapping
\begin{equation*}  %\label{e.}
(t,x) \mapsto u(t,x) - v(t,x) - M \Ll( |x| + Vt - R \Rr)_+
\end{equation*}
does not achieve its supremum on $\{0\} \times \bar U_K$. Let $\ep_0 > 0$, let $\theta \in C^\infty(\R)$ be an increasing smooth function such that
\begin{equation*}  %\label{e.}
\forall r \in \R, \quad (r-\ep_0)_+ \le \theta(r) \le r_+,
\end{equation*}
and  consider 
\begin{equation}  
\label{e.def.Phi}
\Phi(t,x) := M\theta \Bigg( \Ll( \ep_0 + \sum_{k = 1}^K |x_k|^2\Rr)^\frac 1 2   + Vt - R \Bigg).
\end{equation}
We fix $\ep_0 \in (0,1]$ sufficiently close to $0$ so that
\begin{equation}
\label{e.to.contradict}
\sup_{\Ll[0,T\Rr) \times \bar U_K} \Ll( u-v-\Phi \Rr) > \sup_{\{0\} \times \bar U_K} \Ll( u-v-\Phi \Rr).
\end{equation}
For every $k \in \{1,\ldots, K\}$, we have
\begin{equation}  
\label{e.drx.phi}
\dr_{x_k} \Phi(t,x) =  \frac{M x_k}{ \Ll( \ep_0 + \sum_{\ell = 1}^K |x_\ell|^2\Rr)^\frac 1 2}\theta' \Bigg( \Ll( \ep_0 + \sum_{\ell = 1}^K |x_\ell|^2\Rr)^\frac 1 2   + Vt - R \Bigg)
\end{equation}
In particular,  we see that 
\begin{equation}
\label{e.comp.drt.drx.phi}
\dr_t \Phi \ge V |\nabla \Phi|.
\end{equation}
We also record for future use that for every $t \ge 0$ and $x \in\bar U_K$,
\begin{equation}
\label{e.lower.phi}
\Phi(t,x) \ge M \Ll( |x|+ Vt - R - 1 \Rr)_+. 
\end{equation}
We set $C_0 := 2 V |\vec{\imath}|+1$, where the vector $\vec{\imath}$ is explicitly defined below in \eqref{e.def.vecimath}. For some constant $\ep > 0$ to be determined, we define the function
\begin{equation*}  %\label{e.}
\chi(t,x,x') := \Phi(t,x)+ C_0\ep t + \frac{\ep}{T-t} - \ep \sum_{k = 1}^K k\, \mathrm{Id} \cdot (x_k + x'_k),
\end{equation*}
and set $\td{\chi}(t,x) := \chi(t,x,x)$.
In view of \eqref{e.to.contradict}, we can choose $\ep > 0$ sufficiently small that 
\begin{equation}
\label{e.to.contradict.2}
\sup_{\Ll[0,T\Rr) \times \bar U_K} \Ll( u-v-\td{\chi} \Rr) > \sup_{\{0\} \times \bar U_K} \Ll( u-v-\td{\chi} \Rr).
\end{equation}
For later purposes, we also impose that $2 \ep |\vec{\imath}| \le \de$. 
We now introduce, for every $\al \ge 1$, $t \in [0,T)$, $t' \in [0,T]$, and $x,x' \in \bar U_K$, the function
\begin{equation*}  %\label{e.}
\Psi_\al(t,x,t',x') := u(t,x) - v(t',x')  - \frac \al 2 \Ll( |t-t'|^2 + |x-x'|^2 \Rr) - \chi(t,x,x').
%\\
%- C \ep t - \frac{\ep}{T-t} -  \Phi(t,x) + \ep \sum_{k = 1}^K k\, \mathrm{Id} \cdot (x_k + x'_k).
% \sum_{\ell = 1}^k |x_\ell - x_{\ell-1} - x'_\ell +x'_{\ell-1}|^2 
\end{equation*}
By the definitions of $M$, $\chi$, \eqref{e.lower.phi}, and the fact that the functions $u$ and $v$ are uniformly Lipschitz, we see that the supremum of $\Psi_\al$ is achieved, at a point which we denote by $(t_\al,x_\al,t'_\al,x'_\al)$. We also see that this maximizing quadruple stays in a bounded region as~$\al$ tends to infinity, and thus that the quantity 
\begin{equation}  
\label{e.remain.bdd}
\al \Ll( |t_\al-t'_\al|^2 + |x_\al-x'_\al|^2 \Rr)
\end{equation}
must remain bounded as $\al$ tends to infinity. We infer that, up to the extraction of a subsequence, there exist $t_0 \in [0,T)$ and $x_0 \in \bar U_K$ such that, as $\al$ tends to infinity, we have $t_\al \to t_0$, $t'_\al \to t_0$, $x_\al \to x_0$, $x'_\al \to x_0$. Since
\begin{equation*}  %\label{e.}
\Psi_\alpha(t_\al,x_\al,t_\al',x_\al') \le u(t_\alpha,x_\al) - v(t'_\al,x'_\alpha) - \chi(t,x_\al,x'_\al),
\end{equation*}
and
\begin{equation*}  
%\label{e.Psi.sup}
\Psi_\al(t_\al,x_\al,t_\al',x_\al') \ge \sup_{\Ll[0,T\Rr)\times \bar U_K} (u-v-\td{\chi}) \ge (u - v-\td{\chi})(t_0,x_0),
\end{equation*}
we deduce, by continuity of $u$, $v$, and $\td{\chi}$, that 
\begin{equation*}  %\label{e.}
(u-v-\td{\chi})(t_0,x_0) = \sup_{ \Ll[ 0,T \Rr) \times \bar U_K} (u-v-\td{\chi}).
\end{equation*}
In particular, by \eqref{e.to.contradict.2}, we must have $t_0 > 0$, and thus $t_\al > 0$ and $t'_\al > 0$ for every $\al$ sufficiently large. By construction, the function
\begin{equation}  
\label{e.map.xalpha}
(t,x) \mapsto u(t,x) - v(t'_\al,x'_\al) - \frac \al {2} \Ll( |t-t'_\al|^2 + |x-x'_\al|^2 \Rr) - \chi(t,x,x'_\al)
\end{equation}
reaches its maximum at $(t_\al,x_\al)$. Since $u$ is a subsolution, at least one of the following two statements hold:
\begin{equation}
\label{e.constraint.from.u}
\al (t_\al - t'_\al)  + C_0 \ep + \dr_t \Phi(t_\al,x_\al)- \msf H \Ll(  \al  (x_\al - x'_\al) + \nabla_x \chi(t_\al,x_\al,x'_\al)\Rr) \le 0,
\end{equation}
\begin{equation}
\label{e.constraint.from.u.2}
x_\al \in \dr U_K \quad \text{ and } \quad \inf_{\nu \in \n(x_\al)} \Ll(  \al  (x_\al - x'_\al) + \nabla_x \chi(t_\al,x_\al,x'_\al)\Rr) \cdot \nu \le 0.
\end{equation}
In \eqref{e.constraint.from.u}, we dropped an additional term of $\frac{\ep}{(T-t)^2}$ on the left side; this is possible since this term is nonnegative. Notice also that we somewhat reorganized the set of two conditions in \eqref{e.interior.cond.subsol}-\eqref{e.bdy.cond.subsol}, so that we also allow for the possibility that $x_\al \in \dr U_K$ in \eqref{e.constraint.from.u}. 
We now argue that \eqref{e.constraint.from.u.2} cannot hold.
By definition of $\n(x_\al)$, for every $\nu \in \n(x_\al)$, we have $\nu \cdot (x_\al - x'_\al) \ge 0$. We observe that 
\begin{equation}  
\label{e.decomp.gradx}
\nabla_x \chi = \nabla \Phi - \ep \vec{\imath},
\end{equation}
where we have set
\begin{equation}  
\label{e.def.vecimath}
\vec{\imath} := \begin{pmatrix} \mathrm{Id} \\ 2 \, \mathrm{Id} \\ \vdots \\ K\, \mathrm{Id} \end{pmatrix} \in (S^D)^K.
\end{equation}
Moreover, $\nabla \Phi(t,x)$ is a vector proportional to $x$. We now see that, for every $x \in \dr U_K$ and $\nu \in \n(x)$, we have $x \cdot \nu = 0$. Indeed, since $\bar U_K$ is a cone, we have that $\lambda x \in \bar U_K$ for every $\lambda \ge 0$. In particular, we must have that $(\lambda x - x)\cdot \nu \le 0$ for every $\lambda \ge 0$. This can only happen if $x \cdot \nu = 0$. 
Finally, we show that there exists a constant $c > 0$ such that for every $x \in \dr U_K$ and $\nu \in \n(x)$, we have
\begin{equation}
\label{e.lowerbound.dot}
-\vec{\imath} \cdot \nu \ge c .
\end{equation}
By \eqref{e.discr.ibp}, we have
\begin{equation*}  %\label{e.}
\vec{\imath} \cdot \nu = \sum_{k = 1}^K \mathrm{Id} \cdot \sum_{\ell = k}^K \nu_\ell.
\end{equation*}
Notice also that $-\nu \in \bar U_K^*$. By Lemma~\ref{l.dual.cone}, each matrix $-\sum_{\ell = k}^K \nu_\ell$ therefore belongs to $S^D_+$. Moreover, for every $a \in S^D_+$, we have
\begin{equation*}  %\label{e.}
\mathrm{Id} \cdot a \ge |a|,
\end{equation*}
(the left side is the $\ell^1$ norm of the eigenvalues of $a$, the right side the $\ell^2$ norm),
and thus
\begin{equation*}  %\label{e.}
-\vec{\imath} \cdot \nu \ge  \sum_{k = 1}^K \Ll|\sum_{\ell = k}^K \nu_\ell \Rr|.
\end{equation*}
The right side of the inequality above, as a function of $\nu$, defines a norm on $(S^D)^K$. Using the equivalence of norms and that $|\nu| = 1$, we conclude that \eqref{e.lowerbound.dot} holds. Combining the preceding observations, we conclude that \eqref{e.constraint.from.u.2} cannot be valid, so \eqref{e.constraint.from.u} holds instead.

Similarly, since the function
\begin{equation*}  %\label{e.}
(t',x') \mapsto v(t',x') - u(t_\al,x_\al) + \frac \al {2} \Ll( |t'-t_\al|^2 + |x'-x_\al|^2 \Rr) + \chi(t_\al,x_\al,x')
\end{equation*}
has a local minimum at $(t'_\al,x'_\al)$, and since $v$ is a supersolution, at least one of the following two statements must be valid:
\begin{equation}
\label{e.constraint.from.v}
\al (t_\al - t'_\al)  - \msf H \Ll(  \al  (x_\al - x'_\al) + \ep \vec{\imath}\Rr) \ge 0,
\end{equation}
\begin{equation}
\label{e.constraint.from.v.2}
x_\al' \in \dr U_K \quad \text{ and } \quad \sup_{\nu \in \n(x_\al')} \Ll(  \al  (x_\al - x'_\al) + \ep \vec{\imath}\Rr) \cdot \nu \ge 0.
\end{equation}
In view of \eqref{e.lowerbound.dot}, we see that \eqref{e.constraint.from.v.2} cannot hold, and therefore~\eqref{e.constraint.from.v} is valid. 

We now show that \eqref{e.constraint.from.u} and \eqref{e.constraint.from.v} cannot hold simultaneously, thereby reaching the desired contradiction. We temporarily admit that the vectors
\begin{equation}  
\label{e.two.vectors}
\al  (x_\al - x'_\al) + \nabla_x \chi(t_\al,x_\al,x'_\al) \quad \text{ and } \quad  \al  (x_\al - x'_\al) + \ep \vec{\imath}
\end{equation}
are both of norm smaller than $L+\delta$. Admitting this, we use \eqref{e.decomp.gradx}, the Lipschitz property of $\msf H$, the fact that $C_0 = 2 V |\vec{\imath}|+1$, and \eqref{e.comp.drt.drx.phi}, to deduce that \eqref{e.constraint.from.u} implies
\begin{equation*}  %\label{e.}
\al (t_\al - t'_\al)  + \ep - \msf H \Ll(  \al  (x_\al - x'_\al) + \ep \vec{\imath}\Rr) \le 0,
\end{equation*}
in contradiction with \eqref{e.constraint.from.v}.

There remains to verify that the vectors in \eqref{e.two.vectors} are bounded by $L+\delta$. For convenience, we rewrite the mapping in \eqref{e.map.xalpha} as
\begin{equation}
\label{e.map.xalpha.2}
(t,x) \mapsto u(t,x) - \psi(t,x).
\end{equation}
We recall that this mapping achieves its maximum at $(t_\al,x_\al)$, and we aim to show that $|\nabla \psi(t_\al,x_\al)| \le L+\de$. Since $u$ is $L$-Lipschitz, we have, for every $y \in \bar U_K$,
\begin{equation*}  %\label{e.}
\psi(t_\al,y) - \psi(t_\al,x_\al) \ge u(t_\al,y) - u(t_\al,x_\al) \ge -L|y-x_\al|. 
\end{equation*}
If $x_\al \in U_K$, the desired conclusion follows. Otherwise, we can only infer that, for every $v$ in the set 
\begin{equation*}  %\label{e.}
\mfk C := \Ll\{ \lambda (y-x_\al) \ : \ \lambda \in [0,\infty), \ y \in  \bar U_K \Rr\} ,
\end{equation*}
we have 
\begin{equation*}  %\label{e.}
v \cdot \nabla \psi(t_\al,x_\al) \ge -L|v|.
\end{equation*}
We now recall that 
\begin{equation*}  %\label{e.}
\nabla \psi(t_\al,x_\al) = \al (x_\al - x'_\al) + \nabla \Phi(t_\al,x_\al) - \ep \vec{\imath}.
\end{equation*}
Moreover, we have that $\nabla \Phi(t_\al,x_\al)$ is proportional to $x_\al$, say $\nabla \Phi(t_\al,x_\al) = \beta x_\al$, for some $\beta \ge 0$. Since
\begin{equation*}  %\label{e.}
\al (x_\al' - x_\al) - \beta x_\al = (\al+\beta) \Ll( \frac{\al}{\al+\beta} x_\al' - x_\al \Rr) \in \mfk C,
\end{equation*}
we deduce that 
\begin{equation*}  %\label{e.}
(\nabla \psi(t_\al,x_\al)+\ep \vec{\imath}) \cdot \nabla \psi(t_\al,x_\al) \le L |\nabla \psi(t_\al,x_\al)+\ep \vec{\imath}|.
\end{equation*}
This yields
\begin{equation*}  %\label{e.}
|\nabla \psi(t_\al,x_\al)+\ep \vec{\imath}| \le L + \ep |\vec{\imath}|,
\end{equation*}
and thus
\begin{equation*}  %\label{e.}
|\nabla \psi(t_\al,x_\al)| \le L + 2 \ep |\vec{\imath}|.
\end{equation*}
This is the desired result, since we have chosen $\ep > 0$ sufficiently small that $2 \ep |\vec{\imath}| \le \de$. The argument for the second vector in \eqref{e.two.vectors} is similar.
\end{proof}
We next provide with the following result on existence of solutions. 
\begin{proposition}[Existence of solutions]
\label{p.existence}
For every uniformly Lipschitz initial condition $f_0 : U_K \to \R$, there exists a viscosity solution $f$ to \eqref{e.hj.finite} that satisfies $f(0,\cdot) = f_0$. Moreover, the function $f$ is Lipschitz continuous, and we have 
\begin{equation}  
\label{e.lip.hj.finite}
\| \, |\nabla f| \, \|_{L^\infty(\R_+\times U_K)} = \| \, |\nabla f_0| \, \|_{L^\infty(U_K)} .
\end{equation}
\end{proposition}
%\begin{remark}  %\label{}
%The proof of Proposition~\ref{p.existence} below gives that any viscosity solution to \eqref{e.hj.finite} with Lipschitz initial condition must be Lipschitz. Combining this with the comparison principle (Proposition~\ref{p.comp}), we thus obtain existence and uniqueness of solutions under this assumption. This assumption could be weakened, but provided with a convenient framework for the more refined second part of Proposition~\ref{p.comp}. 
%\end{remark}
\begin{proof}[Proof of Proposition~\ref{p.existence}]
We will prove below that the proposition is valid if we assume furthermore that $\msf H$ is uniformly Lipschitz and that the initial condition is bounded. We first explain why this is sufficient. Denote the right side of \eqref{e.lip.hj.finite} by $L$. The proof of Proposition~\ref{p.comp} makes it clear that, if $u$ and $v$ are solutions to \eqref{e.hj.finite} with the same $L$-Lipschitz initial condition and with the nonlinearity $\msf H$ replaced by $\msf H_1$ and $\msf H_2$ respectively, and if $\msf H_1$ and $\msf H_2$ coincide on a ball of radius $L+1$, then $u = v$. It follows that, in order to build a solution to \eqref{e.hj.finite}, we may as well replace $\msf H$ by a globally Lipschitz nonlinearity that coincides with $\msf H$ on the ball of radius $L+1$. Finally, once this is done, we can use the property of finite speed of propagation proved in Proposition~\ref{p.comp} to remove the constraint that the initial condition is bounded.

The argument for the existence of a solution is as in \cite{guide} or \cite[Theorem~7.1]{barles.intro} (in the latter, the initial condition is not assumed continuous, but this additional assumption allows to conclude that the solution is continuous as well). For bounded initial conditions, this construction provides with bounded solutions.

We now turn to the proof of the fact that the solution thus constructed, which we denote by~$f$, is Lipschitz, and that the identity \eqref{e.lip.hj.finite} holds. Again we fix $\mcl C = S^D_{++}$ for concreteness, the case $\mcl C = (0,\infty)^D$ being only easier.  We argue by contradiction, assuming instead that 
\begin{equation*}  %\label{e.}
\| \, |\nabla f| \, \|_{L^\infty(\R_+\times U_K)} > \| \, |\nabla f_0| \, \|_{L^\infty(U_K)} = L.
\end{equation*}
We recall that we assume here that $\msf H$ is uniformly Lipschitz; we denote its Lipschitz constant by~$V$. For a constant $R \in \R$ to be chosen, and $M = 1$, we define the function $\Phi$ as in \eqref{e.def.Phi}. We then set, for constants $\ep > 0$, and $T < \infty$ to be chosen, and every $t \in [0,T)$,
\begin{equation*}  %\label{e.}
u_\ep(t,x) := f(t,x) - \Phi(t,x) - \frac{\ep}{T-t},
\end{equation*}
as well as, for every $t \ge 0$ and $x \in \bar U_K$,
\begin{equation*}  %\label{e.}
v(t,x) := f(t,x) + \Phi(t,x). 
\end{equation*}
Proceeding as in the proof of Proposition~\ref{p.comp}, we verify that $u_\ep$ and $v$ are a sub- and a supersolution to \eqref{e.hj.finite} respectively. We have, for every $t \in [0,T)$ and $x \in \bar U_K$,
\begin{equation}  
\label{e.second.upper.uv}
u_\ep(t,x) - v(t,x)  = -2\Phi(t,x)  - \frac \ep{T-t}.
\end{equation} 
We then choose $\ep > 0$ sufficiently small and $T$ and $R \in \R$ sufficiently large that 
\begin{equation}  
\label{e.not-at-zero}
\sup_{\substack{0 \le t < T \\ x,x' \in \bar U_K}} \Ll(u_\ep(t,x) - v(t,x') - L|x-x'|\Rr) > \sup_{x,x' \in \bar U_K} \Ll(u_\ep(0,x) - v(0,x') - L|x-x'|\Rr).
\end{equation}
We denote by $\eta > 0$ the difference between the left side and the right side of this inequality. We also remark that, by \eqref{e.second.upper.uv},
\begin{equation}  
\label{e.remark.diag}
\sup_{0 \le t < T, x\in \bar U_K} \Ll(u_\ep(t,x) - v(t,x) \Rr) = \sup_{x,x' \in \bar U_K} \Ll(u_\ep(0,x) - v(0,x)\Rr).
\end{equation}
We now let $\delta' \le \delta \in (0,1]$ to be chosen (we will first fix $\delta$ in terms of $\eta$, and then $\delta'$ in terms of $\delta$, $\eta$ and moduli of continuity of $u_\ep$, $v$, and $\msf H$), and for every $\al \in [1,\infty)$, $t < T$, $t' \ge 0$, and $x,x' \in \bar U_K$, we consider 
\begin{equation*}  %\label{e.}
\Psi_\al (t,x,t',x') := u_\ep(t,x) - v(t',x') - (L+\delta t)|x-x'| - \frac{\al}{2} |t_\al - t'_\al|^2 + \delta' \sum_{k = 1}^K k\, \mathrm{Id} \cdot (x_k+x'_k). 
\end{equation*}
Since we assume $f$ to be bounded, the maximum of $\Psi_\al$ is achieved at a point, which we denote by $(t_\al,x_\al,t'_\al,x'_\al)$, and this point remains in a bounded region as $\alpha$ tends to infinity (this bounded region can be chosen irrespectively of our choice of $\de$ and $\de'$ sufficiently small). Extracting a subsequence if necessary, we can further assume that $t_\al \to t_0$, $t'_\al \to t_0$, $x_\al \to x_0$, and $x'_\al \to x'_0$ (the limits of $t_\al$ and $t'_\al$ must be the same, since $|t_\al - t'_\al|^2 = O (\al^{-1})$). 
We also have that, for some constant $C < \infty$,
\begin{equation*}  %\label{e.}
\Psi_\al(t_0,x_0,t_0,x'_0)  \ge -C\de + \sup_{\substack{0 \le t < T \\ x,x' \in \bar U_K}} \Ll(u_\ep(t,x) - v(t,x') - L|x-x'|\Rr),
\end{equation*}
while
\begin{equation*}  %\label{e.}
\Psi_\al(0,x_0,0,x_0')  \le C \de + \sup_{x,x' \in \bar U_K} \Ll(u_\ep(0,x) - v(0,x') - L|x-x'|\Rr),
\end{equation*}
and, using \eqref{e.remark.diag},
\begin{equation*}  %\label{e.}
\Psi_\al(t_0,x_0,t_0,x_0)  \le C \de + \sup_{x,x' \in \bar U_K} \Ll(u_\ep(0,x) - v(0,x') - L|x-x'|\Rr).
\end{equation*}
Choosing $\delta > 0$ such that $4C \delta \le \eta$, we can thus guarantee that $t_0 \neq 0$ and $x_0 \neq x_0'$. More precisely, with this choice of $\delta > 0$, and using the continuity of $u_\ep$ and $v$, we can ensure that there exists $\ga > 0$, not depending on $\delta'$, such that
$|x_0 - x'_0| \ge \ga$. As a consequence, we have $t_\al > 0$, and $|x_\al - x'_\al| \ge \gamma/2$ for every $\al$ sufficiently large.
We use again the notation $\vec{\imath}$ from \eqref{e.def.vecimath}. 
Since $u_\ep$ is a subsolution, at least one of the following statements holds:
\begin{equation}
\label{e.yet.another1}
\al (t_\al - t_\al') + \de |x_\al - x'_\al|- \H \Ll( (L+\de t_\al) \frac{x_\al - x'_\al}{|x_\al - x'_\al|} - \de' {\vec{\imath}}   \Rr) \le 0,
\end{equation}
\begin{equation}
\label{e.yet.another2}
x_\al \in \dr U_K \quad \text{ and } \quad \inf_{\nu \in \n(x_\al)} \nu \cdot \Ll(   (L+\de t_\al) \frac{x_\al - x'_\al}{|x_\al - x'_\al|} - \de' {\vec{\imath}}   \Rr) \le 0.
\end{equation}
By \eqref{e.lowerbound.dot} and the definition of $\n(x_\al)$, the statement in \eqref{e.yet.another2} cannot hold, and therefore~\eqref{e.yet.another1} is valid.
Conversely, since $v$ is a supersolution, at least one of the following statements holds:
\begin{equation}
\label{e.yet.another3}
\al (t_\al - t_\al') - \de |x_\al - x'_\al|- \H \Ll( (L+\de t_\al) \frac{x_\al - x'_\al}{|x_\al - x'_\al|} + \de' {\vec{\imath}}  \Rr) \ge 0,
\end{equation}
\begin{equation}
\label{e.yet.another4}
x'_\al \in \dr U_K \quad \text{ and } \quad \sup_{\nu \in \n(x'_\al)} \nu \cdot \Ll(   (L+\de t_\al) \frac{x_\al - x'_\al}{|x_\al - x'_\al|} + \de' {\vec{\imath}}  \Rr) \ge 0.
\end{equation}
As above, we see that \eqref{e.yet.another4} cannot hold. We thus conclude that \eqref{e.yet.another1} and \eqref{e.yet.another3} are both valid. But, since $|x_\al - x'_\al|$ is bounded away from zero by a quantity not depending on $\de'$, and since $\msf H$ is 	Lipschitz, we reach a contradiction by selecting $\de'$ sufficiently small.
\end{proof}

We now point out a convenient way to verify that certain functions satisfy the boundary condition for being a subsolution to \eqref{e.hj.finite}. The condition is a sort of monotonicity property, which we call being ``tilted'', and is inspired by Lemma~\ref{l.pos.dr_q}. 
Let $V$ be a subset of $\mcl E^K$, and $f : V \to \R$. 
%We say that the function $f$ is \emph{increasing} if, for every $x,y \in U$, we have
%\begin{equation*}  %\label{e.}
%x \le y \quad \implies \quad f(x) \le f(y). 
%\end{equation*}
We say that the function $f$ is \emph{tilted} if, for every $x,y \in V$, we have
\begin{equation*}  %\label{e.}
y - x \in \bar U_K^*\quad \implies \quad f(x) \le f(y),
\end{equation*}
where we recall that $\bar U_K^*$ was defined in Lemma~\ref{l.dual.cone}.
We may also consider functions $f$ defined on $\R_+ \times V$ (or with $\R_+$ replaced by a subinterval); in this case, we say that the function $f$ is tilted if the function $f(t,\cdot)$ is tilted for every fixed $t \ge 0$. 
The next lemma provides with a simple characterization of being tilted for Lipschitz functions. 
\begin{lemma}[Characterization of tilted functions]
\label{l.tilted}
Let $V$ be an open subset of $\mcl E^K$, and let $f : V \to \R$ be a Lipschitz function. The function $f$ is tilted if and only if $\nabla f \in \bar U_K$ almost everywhere in $V$.
\end{lemma}
\begin{proof}
We decompose the proof into two steps.

\emph{Step 1.} We assume that $f$ is tilted, and show that $\nabla f \in \bar U_K$ almost everywhere. By Rademacher's theorem, the function $f$ is differentiable almost everywhere. Let $z \in V$ be a point of differentiability of $f$, and $x \in  \bar U_K^*$. Since $V$ is open and $f$ is tilted, we have, for every $\ep > 0$ sufficiently small,
\begin{equation*}  %\label{e.}
f(z + \ep x) - f(z) \ge 0.
\end{equation*}
Dividing by $\ep$ and letting $\ep > 0$ tend to zero, we conclude that $x \cdot \nabla f (z) \ge 0$. By Lemma~\ref{l.dual.cone}, this means that $\nabla f(z) \in \bar U_K$. 

\emph{Step 2.} We assume that $\nabla f \in \bar U_K$ almost everywhere, and show that $f$ is tilted. By Fubini's theorem, the set 
\begin{equation*}  %\label{e.}
\Ll\{ (x,y) \in V^2  \ : \ \Ll|\{s \in [0,1] \ : \ \mbox{$f$ is differentiable at $sy + (1-s)x$} \}\Rr|=1 \Rr\} 
\end{equation*}
has full measure (for $I \subset [0,1]$, the notation $|I|$ above denotes its Lebesgue measure). We fix a pair $(x,y)$ in this set. Since the mapping $s \mapsto f(sy + (1-s)x)$ is Lipschitz, we have
\begin{equation*}  %\label{e.}
f(y) -f(x)= \int_0^1 (y-x) \cdot \nabla f(sy + (1-s)x) \, \d s.
\end{equation*}
The result then follows using Lemma~\ref{l.dual.cone} once more.
\end{proof}
Notice that, by Lemmas~\ref{l.pos.dr_q} and \ref{l.tilted}, the function
\begin{equation*}  %\label{e.}
(t,q) \mapsto \td F_N \Ll( t,\Ll(\frac 1 k \sum_{\ell = 1}^k \de_{q_{a,\ell}}\Rr)_{a \in \{1,2\}} \Rr) 
\end{equation*}
is tilted. 
As announced, the next proposition states that a tilted function automatically satisfies the boundary condition \eqref{e.bdy.cond.subsol}.
\begin{proposition}[Boundary condition for subsolution]
\label{p.boundary}
Let $f \in C([0,T) \times \bar U_K)$ be a tilted function, $(t,x) \in (0,T) \times \dr U_K$, and $\phi \in C^\infty((0,T)\times \bar U_K)$ be such that $(t,x)$ is a local maximum of $f-\phi$. We have
\begin{equation*}  %\label{e.}
\inf_{\nu \in \n(x)} \nu \cdot \nabla \phi(t,x) \le 0.
\end{equation*}
\end{proposition}
\begin{proof}
We decompose the proof into three steps.

\emph{Step 1.} In this step, we prove a general (and classical) statement concerning the bidual of a closed convex cone.
Let $\mfk C$ be a closed convex cone, which for simplicity we assume to be in some Euclidean space $E$. Let $\mfk C'$ be, up to a sign, the cone dual to $\mfk C$:
\begin{equation}  
\label{e.def.C'}
\mfk C' := \Ll\{ y \in E \ : \ \forall x \in \mfk C, \ x \cdot y \le 0 \Rr\} ,
\end{equation}
and let
\begin{equation}  
\label{e.def.C''}
\mfk C'' := \Ll\{ x \in E \ : \ \forall y \in \mfk C', \ x \cdot y \le 0 \Rr\} .
\end{equation}
In this step, we show that $\mfk C'' = \mfk C$.
Let $f : E \to E$ be such that $f = 0$ on $\mfk C$ and $f = +\infty$ otherwise. Its convex dual $f^*$ is such that, for every $y \in E$,
\begin{equation*}  %\label{e.}
f^*(y) = \sup_{x \in E} \Ll(x \cdot y - f(x)\Rr) = \sup_{x \in \mfk C} x \cdot y = 
\Ll|
\begin{array}{rcl}  %\label{}
0  & \text{if } y \in \mfk C', \\
+\infty & \text{otherwise}.
\end{array}
\Rr.
\end{equation*}
In the same way, using that $\mfk C'$ is a cone, we see that the bidual $f^{**}$ is such that $f^{**} = 0$ on $\mfk C''$, and $f^{**} = +\infty$ otherwise. Since $f$ is convex and lower semicontinuous, it is equal to its bidual. This shows that $\mfk C'' = \mfk C$. 

\emph{Step 2.} We now prove another general (and possibly less classical) statement about closed convex cones. Let $E$ be some Euclidean space, and for any $A \subset E$, let $\msf{int}(A)$ and $\msf{conv}(A)$ denote the interior and the convex hull of $A$ respectively. Let $\mfk C \subset E$ be a closed convex cone, and $\mfk C'$ be as in \eqref{e.def.C'}. Our aim is to show that if the interior of $\mfk C$ is not empty, then 
\begin{equation}  
\label{e.separating}
\msf{int}(\mfk C) \cap (-\mfk C') \neq \emptyset,
\end{equation}
where we write $-\mfk C' := \Ll\{ -z  \ : \ z \in \mfk C' \Rr\}$.
Without loss of generality, we may assume that $\mfk C' \neq \{0\}$ (otherwise we have $\mfk C = E$, by the result of the previous step, and $0$ belongs to the set on the left side of \eqref{e.separating}).
We first show that 
\begin{equation}
\label{e.no.flat.piece}
\mfk C' \cap (-\mfk C') = \{0\}. 
\end{equation}
Indeed, if $z \in \mfk C' \cap (-\mfk C')$, then by the result of the previous step, we must have that $y \cdot z = 0$ for every $y \in \mfk C$. Since we assume that the interior of $\mfk C$ is not empty, this is only possible if $z = 0$. Let
\begin{equation*}  %\label{e.}
\mathbf{m} := \{z \in \mfk C' \ : \ |z| = 1\}.
\end{equation*}
We now show that
\begin{equation}
\label{e.no.zero}
0 \notin \msf{conv}(\mathbf m). 
\end{equation}
Assume instead that $0 \in \msf{conv}(\mathbf{m})$. By Carath\'eodory's theorem, the point $0$ can then be represented as the barycenter of a finite number of points in $\mathbf{m}$. Since $\mathbf{m} \subset \mfk C'$ and $\mfk C'$ is convex and contains the origin, this allows us to contradict \eqref{e.no.flat.piece}. Using Carath\'eodory's theorem once more, we can also verify that $\msf{conv}(\mathbf m)$ is compact.

For every $\ep > 0$, we define
\begin{equation*}  %\label{e.}
\mathbf{m}_\ep := \Ll\{ z \in E \ : \ \msf{dist}(z,\msf{conv}(\mathbf{m})) \le \ep\Rr\},
\end{equation*}
and aim to show that there exists $\ep > 0$ such that
\begin{equation}  
\label{e.disjoint}
\mfk C \cap \{\lambda \nu \ : \ \lambda > 0, \ \nu \in \mathbf m_\ep\} = \emptyset.
\end{equation}
Assume the contrary: for every $\ep > 0$, we could then find $\lambda_\ep > 0$ and $\nu_\ep \in \mathbf m_\ep$ such that $\lambda_\ep \nu_\ep \in \mfk C$. Since $\mfk C$ is a cone, the latter condition means that $\nu_\ep \in \mfk C$. Since the sets $(\mathbf m_\ep)_{\ep > 0}$ are compact and nested, and since $\mfk C$ is closed, we can find a limit point $\nu \in \mathbf m_0 = \msf{conv}(\mathbf m)$ such that $\nu \in \mfk C$. We have in particular that $\nu \in \mfk C'$, but by \eqref{e.def.C'}, we have $\mfk C \cap \mfk C' = \{0\}$. This implies that $\nu = 0$. But since $\nu \in \msf{conv}(\mathbf{m})$, this contradicts \eqref{e.no.zero}. 

Notice next that the set $\{\lambda \nu \ : \ \lambda > 0, \ \nu \in \mathbf m_\ep\}$ is convex: indeed, for every $\lambda, \lambda' > 0$, $\nu,\nu' \in \mathbf{m_\ep}$, and $\al \in (0,1)$, we have
\begin{equation*}  %\label{e.}
\alpha \lambda \nu + (1-\alpha) \lambda' \nu' = (\alpha \lambda + (1-\alpha)\lambda')\Ll(\frac{\alpha \lambda}{\alpha \lambda + (1-\alpha)\lambda'} \nu + \frac{(1-\alpha) \lambda'}{\alpha \lambda + (1-\alpha)\lambda'} \nu' \Rr),
\end{equation*}
and the quantity between parentheses on the right side belongs to $\mathbf{m}_\ep$, since this set is convex. Since $\mfk C$ is also convex, we can find a hyperplane that separates the two disjoint sets appearing in \eqref{e.disjoint}: there exists $v \in E$, which we may assume to be of unit norm, such that
\begin{equation*}  %\label{e.}
\forall z \in \mfk C,  \ z \cdot v \ge 0 \quad \text{ and } \quad \forall z \in \mathbf m_\ep,  \ v \cdot z \le 0.
\end{equation*}
The first property in the previous display yields that $-v \in \mfk C'$. We will now see that $v \in \msf{int}(\mfk C)$, which will complete the proof of \eqref{e.separating}. For every $v' \in E$ satisfies $|v'-v| \le \ep$ and $z \in \mathbf m$, we have
\begin{equation*}  %\label{e.}
v' \cdot z \le v \cdot z + \ep = v \cdot \Ll( z + \ep v \Rr)\le 0,
\end{equation*}
where we used that $z + \ep v \in \mathbf m_\ep$ in the last inequality. Using the notation in \eqref{e.def.C''}, this shows that every such $v'$ belongs to $\mfk C''$. By the result of the previous step, we have $\mfk C'' = \mfk C$, and we have thus verified that $v \in \msf{int}(\mfk C)$.

\emph{Step 3.}
We fix $(t,x)$, $\phi$ as in the statement of Proposition~\ref{p.boundary}, let
\begin{equation}
\label{e.def.C.explicit}
\mfk C_0 := \Ll\{ \lambda (y-x) \ : \ \lambda > 0, \ y \in  U_K \Rr\} ,
\end{equation}
and let $\mfk C$ denote its closure. Since $\mfk C_0$ is open, we have $\msf{int}(\mfk C) = \mfk C_0$.
For every $\lambda, \lambda' \in (0,\infty)$ and $y,y' \in U_K$, we have
\begin{equation*}  %\label{e.}
\lambda (y-x) + \lambda'(y'-x) = (\lambda + \lambda') \Ll( \frac{\lambda}{\lambda + \lambda'} y + \frac{\lambda'}{\lambda + \lambda'} y'-x\Rr)  \in \mfk C_0.
\end{equation*}
It follows that $\mfk C_0$ is convex, and thus that $\mfk C$ is a closed convex cone.  
Let $\mfk C'$ be defined by \eqref{e.def.C'}, with $E = \mcl E^K$. Since $\mfk C$ has nonempty interior, we can apply the result of the previous step to infer that
\begin{equation*}  %\label{e.}
\mfk C_0 \cap (-\mfk C') \neq \emptyset. 
\end{equation*}
Let $v$ denote an element of this set; without loss of generality, we may assume that $v$ is of unit norm. By definition of $\mfk C_0$, there exists $\lambda > 0$ and $y \in U_K$ such that $v = \lambda(y-x)$. Since $-v \in \mfk C'$ and is of unit norm, we also have that $-v \in \n(x)$. Since the set $\mfk C$ contains $\bar U_K$, we have that $-\mfk C' \subset \bar U_K^*$, and thus $v \in \bar U_K^*$. By convexity of $\bar U_K$, for every $\ep \in [0,\lambda^{-1}]$, we have that 
\begin{equation*}  %\label{e.}
x+ \ep v = (1-\ep \lambda) x + \ep \lambda y \in \bar U_K.
\end{equation*}
By the assumption that $(t,x)$ is a local maximum of $f-\phi$, for every $\ep > 0$ sufficiently small, we have 
\begin{equation*}  %\label{e.}
(f-\phi)(t,x+\ep v) \le (f-\phi)(t,x).
\end{equation*}
Since $f$ is tilted and $v \in \bar U_K^*$, we deduce that $\phi(t,x+\ep v)- \phi(t,x) \ge 0$. Dividing by $\ep > 0$ and letting it tend to zero, we obtain that
\begin{equation*}  %\label{e.}
v \cdot \nabla \phi(t,x) \ge 0.
\end{equation*}
Since $-v \in \n(x)$, this is the desired result.
\end{proof}

\subsection{Convergence of finite-dimensional approximations}
We now turn to the identification of the solution to \eqref{e.hj}, which we define to be the limit of the solutions to suitable finite-dimensional approximations. 
From now on, we specialize the results of the previous subsection to the case of 
\begin{equation*}  %\label{e.}
\mcl C := (0,\infty)^2.
\end{equation*}
We aim to approximate each measure in a given pair $(\mu_1,\mu_2) \in (\mcl P(\R_+))^2$ by a measure of the form
\begin{equation}  
\label{e.map.measure}
\frac 1 K \sum_{k = 1}^K \de_{x_{k,a}}, \qquad (a \in \{1,2\}),
\end{equation}
for some (ultimately large) integer $K$ and some $x \in \bar U_K$, where we set $x_{k} = (x_{k,1}, x_{k,2})$, and $x = (x_1,\ldots, x_K)$. 
We can clearly map any element of $\bar U_K$ to a pair of probability measures in $(\mcl P(\R_+))^2$ through the mapping defined in \eqref{e.map.measure}. We can also define a converse operation, from a given pair of measures in $(\mcl P(\R_+))^2$ to an element of $\bar U_K$. Fixing $K \ge 1$, we set, for every $\mu = (\mu_1,\mu_2) \in (\mcl P(\R_+))^2$ and $k \in \{1,\ldots, K\}$,
\begin{equation}
\label{e.def.mu.to.xk}
x^{(K)}_k(\mu)  = (x^{(K)}_{k,1}(\mu), x^{(K)}_{k,2}(\mu)) := K \int_{\frac{k-1}{K}}^{\frac k K} (F_{\mu_1}^{-1}(u), F_{\mu_2}^{-1}(u)) \, \d u \in \R_+^2,
\end{equation}
where we recall that the functions $F_{\mu_a}^{-1}$ were introduced in \eqref{e.def.F-1}. 
This defines a mapping $\mu \mapsto x^{(K)}(\mu)$ from $(\mcl P(\R_+))^2$ to $\bar U_K$. Notice that if we map an element $x$ of $\bar U_K$ to a pair of measures according to \eqref{e.map.measure}, and then back into an element of~$\bar U_K$ through the mapping above, we recover $x$ (but obviously, some information is lost when we go from a pair of measures to an element of $\bar U_K$ and then back). We also use the notation
\begin{equation*}  %\label{e.}
\mu^{(K)} = (\mu^{(K)}_1, \mu^{(K)}_2) := \Ll(\frac 1 K \sum_{k = 1}^K \de_{x^{(K)}_{k,a}(\mu)} \Rr)_{a \in \{1,2\}}.
\end{equation*}
The mapping $\mu \mapsto \mu^{(K)}$ thus takes an element $\mu$ of $(\mcl P(\R_+))^2$, and returns a pair in $(\mcl P(\R_+))^2$, made of two measures with $K$ atoms of equal masses (the latter is in some sense the ``representative'' of $x^{(K)}(\mu) \in \bar U_K$ within the set $(\mcl P(\R_+))^2$).  

The following proposition is the main result of this subsection. 
\begin{proposition}[Convergence of finite-dimensional approximations]
\label{p.conv.finite.dim}
Let $\psi$ be the function defined in \eqref{e.def.psi}, and for each integer $K \ge 1$, let $f^{(K)} : \R_+ \times \bar U_K \to \R$ be the viscosity solution to
\begin{equation}
\label{e.def.finite.dim}
\Ll\{
\begin{aligned}  %\label{}
& \dr_t f^{(K)} - K \sum_{k = 1}^K  \dr_{x_{k,1}} f^{(K)} \, \dr_{x_{k,2}} f^{(K)} = 0 & \quad \text{on } (0,\infty) \times \bar U_K,
\\
& \n \cdot \nabla f^{(K)} = 0 & \quad \text{on } (0,\infty) \times \dr U_K,
\end{aligned}
\Rr.
\end{equation}
with initial condition given, for every $x \in \bar U_K$, by
\begin{equation}
\label{e.def.finite.dim.init}
f^{(K)}(0,x) = \psi \Ll( \frac 1 K \sum_{k = 1}^K \de_{x_{k,1}},\frac 1 K \sum_{k = 1}^K \de_{x_{k,2}}  \Rr) .
\end{equation}
For every $t \ge 0$ and $\mu \in (\mcl P_2(\R_+))^2$, the following limit exists and is finite:
\begin{equation}  
\label{e.conv.finite.dim}
f(t,\mu) := \lim_{K \to \infty} f^{(K)}\Ll(t,x^{(K)}(\mu)\Rr),
\end{equation}
where on the right side, we use the notation defined in~\eqref{e.def.mu.to.xk}. By definition, we interpret this limit as the solution to \eqref{e.hj}.
Moreover, there exists a constant $C < \infty$ such that, for every integer $K \ge 1$, $t \ge 0$, and $\mu,\nu \in (\mcl P_2(\R_+))^2$, we have
\begin{equation}
\label{e.quant.conv}
\Ll| f(t,\mu) - f^{(K)}(t,x^{(K)}(\mu)) \Rr| \le \frac{C}{\sqrt{K}} \Ll( t + \Ll(\E \Ll[ X_{\mu_1}^2 + X_{\mu_2}^2 \Rr]\Rr) ^\frac 1 2   \Rr) ,
\end{equation}
as well as
\begin{equation}  
\label{e.lip.f}
\Ll| f(t,\mu) - f(t,\nu) \Rr| \le \Ll(\E \Ll[ |X_{\mu_1} - X_{\nu_1}|^2 + |X_{\mu_2} - X_{\nu_2}|^2 \Rr]\Rr)^\frac 1 2.
\end{equation}
\end{proposition}
Before turning to the proof of Proposition~\ref{p.conv.finite.dim}, we introduce some notation for norms that are rescaled to be consistent with Wasserstein-type distances on the space of probability measures, according to the correspondences discussed at the beginning of this subsection. For every $\rho \in [1,\infty]$ and $x \in \mcl E^K = (\R^2)^K$, we write
\begin{equation}  
\label{e.def.rho.norm}
|x|_\rho := \Ll(\frac 1 K \sum_{k = 1}^K |x_k|^\rho\Rr)^\frac 1 \rho,
\end{equation}
with the usual interpretation as a supremum if $\rho = \infty$. We also define the norm dual to $|\cdot|_\rho$ by setting, for $\tau \in [1,\infty]$ such that $\frac 1 \rho + \frac 1 {\tau} = 1$ and every $x \in \mcl E^K = (\R^2)^K$,
\begin{equation}
\label{e.def.rho*.norm}
|x|_{\tau*} := K^\frac 1 {\rho} \Ll( \sum_{k = 1}^K |x_k|^{\tau} \Rr) ^\frac 1 {\tau} = \Ll( \frac 1 K \sum_{k = 1}^K (K |x_k|)^{\tau} \Rr) ^{\frac 1 {\tau}}.
\end{equation}
We also observe that, by a simple rescaling, the statement of Proposition~\ref{p.comp} also holds if we replace the displays \eqref{e.def.L}-\eqref{e.comp.precise}, by, respectively,
\begin{equation*}  %\label{e.}
L := \max(\|\, |\nabla u|_{2*} \, \|_{L^\infty([0,T)\times U_K)}, \|\, |\nabla v|_{2*} \, \|_{L^\infty([0,T)\times U_K)}),
\end{equation*}
\begin{equation}  
\label{e.alt.def.V}
V := \sup \Ll\{ \frac{|\H(p') - \H(p)|}{|p'-p|_{2*}} \ : \ |p|_{2*}, |p'|_{2*} \le L + \delta \Rr\},
\end{equation}
\begin{equation*}  
%\label{e.growth.M}
M > \sup_{0 \le t < T, x \in \bar U_K} \frac{u(t,x) - v(t,x)}{1+|x|_2},
\end{equation*}
and
\begin{equation*}  
%\label{e.alt.comp.precise}
(t,x)  \mapsto  u(t,x) - v(t,x) - M \Ll( |x|_2 + Vt - R \Rr)_+.
\end{equation*}

\begin{proof}[Proof of Proposition~\ref{p.conv.finite.dim}]
We decompose the proof into three steps.

\emph{Step 1.} We start with a simple but crucial observation regarding the relationship between the $f^{(K)}$'s for different values of $K$. 
For all integers $K, R \ge 1$, we set $K' := RK$, and for every $x \in \bar U_{K'}$, we define
\begin{equation*}  %\label{e.}
x^{(K,K')} := \Ll( \frac 1 R \sum_{r = 1}^R x_r, \frac 1 R \sum_{r = 1}^R x_{R+r}, \ldots, \frac 1 R \sum_{r = 1}^R x_{(K-1)R + r}   \Rr) \in \bar U_K,
\end{equation*}
as well as
\begin{equation*}  %\label{e.}
f^{(K,K')}(t,x) := f^{(K)}(t,x^{(K,K')}).
\end{equation*}
In some sense, the function $f^{(K,K')}$ is a ``lifting'' of the function $f^{(K)}$ to the space $\bar U_{K'}$ (and this ``lifting'' is consistent with the identification between measures and elements of~$\bar U_K$ discussed at the beginning of this section). 
Formally, we have for every $k \in \{1,\ldots, K\}$ and $r \in \{1,\ldots, R\}$ that
\begin{equation*}  %\label{e.}
\dr_{x_{(k-1) R + r}} f^{(K,K')}(t,x) = \frac 1 R \dr_{x_k} f^{(K)} (t,x^{(K,K')}),
\end{equation*}
and thus, on a formal level, the function $f^{(K,K')}$ solves the same equation as $f^{(K')}$ does, but with a different initial condition. It is not difficult to justify rigorously that $f^{(K,K')}$ indeed solves this equation in the viscosity sense. In a few words, for instance to verify that $f^{(K,K')}$ is a subsolution: suppose that $(t,x)$ is a local maximum of $f^{(K,K')}-\phi$ for some smooth function~$\phi$. Then we can build $\td \phi \in C^\infty((0,\infty)\times \bar U_K)$ by setting, for every $x' \in \bar U_K$,
\begin{equation*}  %\label{e.}
\td \phi(t,x') := \phi \Ll( t,x + \Ll( x'_1-x^{(K,K')}_1,\ldots, x'_1-x^{(K,K')}_1, \ldots, x'_K-x^{(K,K')}_K, \ldots, x'_K-x^{(K,K')}_K \Rr)  \Rr) ,
\end{equation*}
where each coordinate in the inner parenthesis above is repeated $R$ times. This ensures that $f^{(K)} - \td \phi$ has a local maximum at $(t,x^{(K,K')})$. We then use that $f^{(K)}$ is a subsolution, and the simple relationship between the deriatives of $\td \phi$ and those of $\phi$, to conclude.

\emph{Step 2.}
We next leverage on this observation to evaluate the difference between $f^{(K,K')}$ and $f^{(K')}$, using Proposition~\ref{p.comp}. Precisely, we will show that there exists a constant $C < \infty$ such that for every $t \ge 0$ and $x \in \bar U_{K'}$,
\begin{equation}
\label{e.comp.fkk'.fk'}
\Ll| f^{(K,K')}(t,x) - f^{(K')}(t,x) \Rr| \le \frac{C}{\sqrt{K}} \Ll( |x|_2 + t \Rr) .
\end{equation}
By the definition of $\psi$ in \eqref{e.def.psi} and Proposition~\ref{p.continuity}, we have that, for every $x,y \in \bar U_K$,
\begin{equation}
\label{e.lipschitz.init.cond}
\Ll| f^{(K)}(0,y) - f^{(K)}(0,x) \Rr| \le   |x-y|_1.
\end{equation}
By Jensen's inequality, we also have 
\begin{equation}
\label{e.lip.f.KK'}
\Ll| f^{(K,K')}(0,y) - f^{(K,K')}(0,x) \Rr| \le   |x^{(K,K')}-y^{(K,K')}|_1 \le  |x-y|_1.
\end{equation}
Since $|\cdot|_1 \le |\cdot|_2$, we can appeal to Proposition~\ref{p.existence} to infer that 
\begin{equation}  
\label{e.lip2.fk}
\|\, |\nabla f^{(K)}|_{2*} \, \|_{L^\infty(\R_+\times U_K)} \le 1,
\end{equation}
and
\begin{equation}  
\label{e.lip2.fkk}
\|\, |\nabla f^{(K,K')}|_{2*} \, \|_{L^\infty(\R_+\times U_K)} \le 1.
\end{equation}
In view of the observation preceding this proof, it is thus legitimate to apply Proposition~\ref{p.comp} with $u$ and $v$ replaced by $f^{(K,K')}$ and $f^{(K')}$, and with the choice of $L = 1$. We also observe that, for every $p, p' \in \mcl E^K = (\R^2)^K$,
\begin{align*}  %\label{e.}
\Ll|K \sum_{k = 1}^K p_{k,1} p_{k,2} - K \sum_{k = 1}^K p'_{k,1} p'_{k,2}\Rr|
& = 
\Ll| {K} \sum_{k = 1}^{K} \Ll( p_{k,1} \Ll( p_{k,2} -  p'_{k,2} \Rr) + (p_{k,1} - p'_{k,1})p_{k,2}'  \Rr)  \Rr| 
\\
& 
\le
\Ll({K} \sum_{k = 1}^{K} p_{k,1}^2\Rr)^\frac 1 2\Ll(K \sum_{k = 1}^{K} (p_{k,2} - p'_{k,2})^2\Rr)^\frac 1 2
\\
& \qquad  + \Ll(K \sum_{\ell = 1}^{K} (p'_{k,2})^2\Rr)^\frac 1 2\Ll(K \sum_{k = 1}^{K} (p_{k,1} - p'_{k,2})^2\Rr)^\frac 1 2,
\\
&  \le (|p|_{2*} + |p'|_{2*})|p-p'|_{2*}.
\end{align*}
We can thus for instance choose $V = 3$ when applying Proposition~\ref{p.comp} to our current setting. For convenience, we also fix $M := 2L+1 = 3$, which clearly satisfies \eqref{e.growth.M} for such choices of $u$ and $v$. We thus deduce that, for every $R \in \R$, the mapping
\begin{equation}
\label{e.map.sup}
(t,x) \mapsto f^{(K,K')}(t,x) - f^{(K')}(t,x) - 3 \Ll( |x|_2 + 3t - R  \Rr)_+
\end{equation}
achives its supremum on $\{0\}\times \bar U_{K'}$. We now derive two different bounds on this supremum, the first one being simple and convenient for large $|x|_2$, the second one covering the case of more moderate values of this quantity. The first bound is a consequence of the estimates \eqref{e.lip2.fk} and \eqref{e.lip2.fkk}, and of the fact that $f^{(K,K')}(0,0) = f^{(K')}(0,0)$: we have
\begin{equation}
\label{e.first.bound.init}
f^{(K,K')}(0,x) -f^{(K')}(0,x) - 3 \Ll( |x|_2  - R  \Rr)_+ \le 3R - |x|_2.
\end{equation}
For the second bound, we first rewrite the supremum of \eqref{e.map.sup} over $\{0\} \times \bar U_{K'}$ as
\begin{equation}  
\label{e.write.sup}
\sup_{x \in \bar U_{K'}} \Ll\{\Ll| \psi \Ll( \frac 1 {K'} \sum_{ k = 1}^{K'} \de_{x_k} \Rr) - \psi \Ll( \frac 1 {K} \sum_{ k = 1}^{K} \de_{x^{(K,K')}_k} \Rr)  \Rr|  - 3 \Ll(|x|_2 - R  \Rr)_+ \Rr\}.
\end{equation}
By \eqref{e.lipschitz.init.cond}, the difference of $\psi$'s in the supremum above can be bounded by
\begin{align}  
\label{e.diffy}
\frac 1 {K'} \sum_{k = 1}^{K} \sum_{r = 1}^R \Ll| x_{(k-1)R + r} - x_k^{(K,K')} \Rr| \le \frac 1 K \sum_{k = 1}^K \frac 1 {R^2} \sum_{r,r' = 1}^R \Ll| x_{(k-1)R +r} - x_{(k-1)R +r'} \Rr| .
\end{align}
For every $B \in (0,\infty)$, we have
\begin{equation}  
\label{e.diffy.2}
\frac 1 {K'} \sum_{k = 1}^{K'} |x_{k}| \1_{\{ |x_{k}| \ge B \}} \le \frac 1 {BK'} \sum_{k = 1}^{K'} |x_{k}|^2  = \frac{|x|_2^2}{B}.
\end{equation}
On the other hand,
\begin{align}  
\label{e.diffy.3}
& \frac 1 K \sum_{k = 1}^K \frac 2 {R^2} \sum_{1 \le r' < r \le R}\Ll| x_{(k-1)R +r} - x_{(k-1)R +r'} \Rr| \1_{\{ |x_{(k-1)R +r}| < B \}} 
\\
\notag
& \le \frac 2 {R^2} \sum_{1 \le r' < r \le R} \frac 1 K \sum_{k = 1}^K \Ll| x_{(k-1)R +r} - x_{(k-1)R +r'} \Rr| \1_{\{ |x_{(K-1)R +r}| < B \}} .
\end{align}
By equivalence of norms over $\R^2$, up to a constant factor, we can replace the Euclidean norm in $|x_{(k-1)R +r} - x_{(k-1)R +r'}|$ by the $\ell^1$ norm; and in this case, since $x_{(k-1)R +r'} \le x_{(k-1)R +r}$, the sum above becomes telescopic. We thus have that, for some absolute constant $C < \infty$,
\begin{equation*}  %\label{e.}
\frac 2 {R^2} \sum_{1 \le r' < r \le R} \frac 1 K \sum_{k = 1}^K \Ll| x_{(k-1)R +r} - x_{(k-1)R +r'} \Rr| \1_{\{ |x_{(K-1)R +r}| < B \}} \le \frac{C B}{K}. 
\end{equation*}
Summarizing, we have shown that, for every $x \in \bar U_{K'}$ amd $B \in (0,\infty)$,
\begin{align*}  %\label{e.}
\Ll| \psi \Ll( \frac 1 {K'} \sum_{ k = 1}^{K'} \de_{x_k} \Rr) - \psi \Ll( \frac 1 {K} \sum_{ k = 1}^{K} \de_{x^{(K,K')}_k} \Rr)  \Rr| \le \frac{2|x|_2^2}{B}+ \frac{C B}{K}.
\end{align*}
For $B = \sqrt{K} \, |x|_2$, this becomes, up to a redefinition of $C < \infty$,
\begin{equation*}  %\label{e.}
\Ll| \psi \Ll( \frac 1 {K'} \sum_{ k = 1}^{K'} \de_{x_k} \Rr) - \psi \Ll( \frac 1 {K} \sum_{ k = 1}^{K} \de_{x^{(K,K')}_k} \Rr)  \Rr| \le \frac {C} {\sqrt{K}} \, |x|_2.
\end{equation*}
Summarizing, we have thus shown that the quantity inside the supremum in \eqref{e.write.sup} is bounded by
\begin{equation*}  %\label{e.}
 \frac {C} {\sqrt{K}} \, |x|_2 - 3\Ll( |x|_2  - R \Rr)_+ . 
\end{equation*}
Notice that the bound in \eqref{e.first.bound.init} is already negative for $|x|_2 \ge 3 R$. On the complementary event, the quantity above is clearly bounded by $\frac{3 C R}{\sqrt{K}}$. Up to a redefinition of $C < \infty$, we have thus shown that, for every $R > 0$,
\begin{equation*}  %\label{e.}
\sup_{t \ge 0, x \in \bar U_{K'}} \Ll\{f^{(K,K')}(t,x) - f^{(K')}(t,x) - 3 \Ll( |x|_2 + 3t - R  \Rr)_+\Rr\} \le \frac {C R}{\sqrt{K}}.
\end{equation*}
Choosing $R = 3|x|_2 + 3t$ then yields one bound for \eqref{e.comp.fkk'.fk'}. The converse bound is obtained in the same way.

\emph{Step 3.} We complete the proof, by showing that there exists a constant $C$ such that for every $t \ge 0$, $\mu = (\mu_1,\mu_2)\in (\mcl P_2(\R_+))^2$, and $1 \le K_1 \le K_2$, we have
\begin{equation}
\label{e.cauchy}
\Ll| f^{(K_1)}(t,x^{(K_1)}(\mu)) - f^{(K_2)}(t,x^{(K_2)}(\mu)) \Rr| \le \frac{C}{\sqrt{K_1}} \Ll( \Ll(\E \Ll[ X_{\mu_1}^2 + X_{\mu_2}^2 \Rr]\Rr) ^\frac 1 2  + t \Rr) .
\end{equation}
In order to show \eqref{e.cauchy}, it suffices to verify that, for all integers $K, R \ge 1$, and with $K' := RK$, we have
\begin{equation}
\label{e.cauchy.easy}
\Ll| f^{(K)}(t,x^{(K)}(\mu)) - f^{(K')}(t,x^{(K')}(\mu)) \Rr| \le \frac{C}{\sqrt{K}} \Ll( \Ll(\E \Ll[ X_{\mu_1}^2 + X_{\mu_2}^2 \Rr]\Rr) ^\frac 1 2  + t \Rr) .
\end{equation}
Indeed, once \eqref{e.cauchy.easy} is proved, we can apply it with $(K,K')$ replaced by $(K_1, K_1K_2)$ and $(K_2,K_1K_2)$ and obtain \eqref{e.cauchy} by the triangle inequality. But \eqref{e.cauchy.easy} is almost identical to \eqref{e.comp.fkk'.fk'}: indeed, the latter identity states that the left side of \eqref{e.cauchy.easy} is bounded by 
\begin{equation*}  %\label{e.}
\frac{C}{\sqrt{K}} \Ll( |x^{(K')}(\mu)|_{2} + t \Rr) ,
\end{equation*}
and, by Jensen's inequality,
\begin{equation}  
\label{e.easy.l2.jensen}
|x^{(K')}(\mu)|_{2} = \Ll(\frac{1}{K'} \sum_{k = 1}^{K'} \Ll| K' \int_{\frac{k-1}{K'}}^{\frac k {K'}} (F_{\mu_1}^{-1}(u), F_{\mu_2}^{-1}(u)) \, \d u \Rr|^2 \Rr)^\frac 1 2  \le \Ll(\E \Ll[ X_{\mu_1}^2 + X_{\mu_2}^2 \Rr]\Rr) ^\frac 1 2 .
\end{equation}
Now, it is clear that \eqref{e.cauchy} guarantees the existence of the limit in \eqref{e.conv.finite.dim}. It also yields the estimate \eqref{e.quant.conv}, by letting $K_2$ tend to infinity. To show the Lipschitz estimate \eqref{e.lip.f}, we start from \eqref{e.lip2.fk}, which can be rewritten as, for every $t \ge 0$ and $\mu, \nu \in (\mcl P_2(\R_+))^2$,
\begin{equation*}  %\label{e.}
\Ll|f^{(K)}(t,x^{(K)}(\mu)) - f^{(K)}(t,x^{(K)}(\nu)) \Rr| \le  |x^{(K)}(\mu) - x^{(K)}(\nu)|_2.
\end{equation*}
As in \eqref{e.easy.l2.jensen}, we can then bound the right side above by
\begin{equation*}  %\label{e.}
\Ll(\E \Ll[ |X_{\mu_1} - X_{\nu_1}|^2 + |X_{\mu_2} - X_{\nu_2}|^2 \Rr]\Rr)^\frac 1 2.
\end{equation*}
The estimate \eqref{e.lip.f} then follows by letting $K$ tend to infinity.
\end{proof}

%
%
%
%%%%%%%%%%%%%%%%%%%%%%%%%%%%
%%%%%%%%%%%%%%%%%%%%%%%%%%%%
%
%
%
\section{The free energy is a supersolution}
\label{s.serious}

The main goal of this section is to show that finite-dimensional approximations of $\bar F_N$ are supersolutions to the finite-dimensional approximations of \eqref{e.hj}, up to a small error. Compared with the previous section, we change the indexing convention and write, for every integer $k \ge 1$, 
\begin{multline}
\label{e.redef.Uk}
U_k := \big\{ q = (q_{1,1},\ldots, q_{1,k}, q_{2,1},\ldots, q_{2,k}) \in (0,\infty)^{2k} \ : 
\\ \forall a \in \{1,2\}, \forall \ell \in \{1,\ldots, k-1\}, \ q_{a,\ell} < q_{a,\ell+1} \big\} ,
\end{multline}
and we denote the closure of $U_k$ by $\bar U_k$. 
\begin{theorem}[approximate HJ equation]
\label{t.approx.HJ}
For each integer $k \ge 1$, $t \ge 0$, and $q \in \bar U_k$ indexed as $q = (q_{1,1},\ldots, q_{1,k}, q_{2,1},\ldots, q_{2,k})$, denote
\begin{equation}
\label{e.def.barFNk}
\bar F_N^{(k)}(t,q) := \bar F_N \Ll( t, \frac 1 k \sum_{\ell = 1}^k \de_{q_{1,\ell}} , \frac 1 k \sum_{\ell = 1}^k \de_{q_{2,\ell}} \Rr) ,
\end{equation}
and let $f$ be any subsequential limit of $\bar F_N^{(k)}$ as $N$ tends to infinity. We have, in the sense of viscosity solutions,
\begin{equation}  
\label{e.forward}
\Ll\{
\begin{aligned}
& \dr_t f - k \sum_{\ell= 1}^k \dr_{q_{1,\ell}} f \, \dr_{q_{2,\ell}} f \ge - \frac{13}{k} & \quad  \text{ on } (0,\infty) \times U_k, 
\\
& \n \cdot \nabla f \ge 0 & \quad \text{on } (0,\infty) \times \dr U_k.
\end{aligned}
\Rr.
\end{equation}
\end{theorem}
In the statement above, we understand the notion of subsequential limit in the sense of locally uniform convergence. (By Proposition~\ref{p.continuity}, the functions involved are uniformly Lipschitz, and by Lemma~\ref{l.init}, the initial condition does not depend on $N$, so the existence of converging subsequences is clear.) Once Theorem~\ref{t.approx.HJ} is proved, we will combine it with the results of the previous section to obtain a proof of Theorem~\ref{t.main}. 

As was announced in Section~\ref{s.defs}, see in particular \eqref{e.approxN.hj}, we need to show that the overlaps $\si_a \cdot \si_a'$ are ``typically'' synchronized with the overlap $\al \wedge \al'$. The argument for achieving this relies on the fact that, possibly after a small perturbation of the energy function, we can ensure that the structure of the Gibbs measure is ultrametric \cite{pan.aom}. That the ultrametricity can be used to infer synchronization was first observed in \cite{pan.multi}; we revisit the argument in Section~\ref{s.synch} below to provide us with a ``finitary'' version of the statement of synchronization, which is more adapted to the needs of the proof of Theorem~\ref{t.approx.HJ}. The small perturbations of the energy function are meant to ensure the validity of the Ghirlanda-Guerra identities. The reader may want to have a brief look at Section~\ref{s.synch} to understand better the motivation behind the introduction of such perturbations. 

The bird's eye view proposed above is that ``typically'', the overlaps synchronize. One subtle point is to uncover what ``typically'' should actually mean. We cannot hope for this synchronization property to hold for \emph{every} choice of the parameters. Conversely, knowing that synchronization occurs for almost every choice of the parameters is not sufficient. Indeed, this would boil down to considering a limit free energy that satisfies the partial differential equation \eqref{e.hj.finite} at almost every point. But such a property is unfortunately \emph{not} sufficient to identify the solution to \eqref{e.hj.finite} uniquely; and this is the reason why the more involved notion of viscosity solutions is introduced. Roughly speaking, we will be able to show that synchronization occurs at any contact point appearing in the definition of supersolution. More precisely, in the notation of Definition~\ref{def.solution}, at a point where $f-\phi$ is minimal, we will be able to leverage on the fact that the Hessian of $f$ must be bounded from below to deduce the validity of the Ghirlanda-Guerra identities, and therefore the synchronization of the overlaps.

We now introduce the small perturbations of the energy function alluded to above.
We fix $(\lambda_n)_{n \ge 1}$ an enumeration of the set of rational numbers in $[0,1]$. For every integer triple $h = (h_1,h_2,h_3) \in \N_*^3$ and $a \in \{1,2\}$, we define the random energy $(H_N^{a,h}(\si,\al))_{\si \in \R^{2N}, \al \in \N^k}$, which is a centered Gaussian field with covariance given, for every $\si,\si' \in \R^{2N}$ and $\al,\al' \in \N^k$, by
\begin{equation}  
\label{e.cov.HNah}
\E \Ll[ H_N^{a,h}(\si,\al) H_N^{a,h}(\si',\al') \Rr] = N \Ll( \lambda_{h_1} \frac{\si_{a}\cdot \si_{a}'}{N} + \lambda_{h_2} \frac{\al \wedge \al'}{k} \Rr)^{h_3} .
\end{equation}
The fact that such a Gaussian random field exists is shown in Lemma~\ref{l.exist.HNah} of the appendix. (It is also seen there that the variables $\al$ can be embedded into a Hilbert space in such a way that $\al \wedge \al'$ becomes the scalar product of the ``embedded'' variables. Strictly speaking, this observation is required to use the results of Section~\ref{s.synch} with these variables.) We impose the fields $(H_N^{a,h})_{a \in \{1,2\}, h \in \N_*^3}$ to be independent, and to be independent of the other random variables in the problem. Enlarging the probability space if necessary, we assume that these additional random fields are defined on the probability space with measure $\P$. Let $h_+$ be an integer that will be chosen sufficiently large (in terms of $k$) in the course of the argument. For convenience, we understand that every element $x \in \R^{2+2 h_+^3}$ is indexed according to
\begin{equation*}  %\label{e.}
x = (x_1,(x_{1,h})_{h \in \{1,\ldots, h_+\}^3}, x_2, (x_{2,h})_{h \in \{1,\ldots, h_+\}^3}) .
\end{equation*}
With this understanding, we set
\begin{equation*}  %\label{e.}
H_N^x(\si,\al) := N^{-\frac 1 {16}} \sum_{a \in \{1,2\}} \Ll(x_a |\si_a|^2 + \sum_{h \in \{1,\ldots, h_+\}^3} x_{a,h} H_N^{a,h}(\si,\al) \Rr).
\end{equation*}
The prefactor $N^{-\frac 1 {16}}$ is meant to ensure that $H_N^x$ will not contribute to the limit free energy, see \eqref{e.def.FNx} and \eqref{e.dr.barF.xah} below. The exponent $\frac 1 {16}$ is relatively arbitrary; one could replace it by any exponent in the interval $(0,\frac 1 8)$. 
%In practice, we will use the parameter $x$ perturbatively; more precisely, we will only really appeal to values of $x$ that go to zero like a negative power of $N$. This will guarantee that the limit free energy is unchanged by this perturbation.
We now define a new free energy that includes the perturbative terms: for every $t \ge 0$, $\mu \in (\mcl P(\R_+))^2$ of the form \eqref{e.mu.diracs}, and $x \in \R^{2 + 2 h_+^3}$, we set, with $H_N^t$ defined in~\eqref{e.def.HNt} and $H^\mu_N$ defined in \eqref{e.def.HNalpha},
\begin{equation}  
\label{e.def.FNx}
F_N(t,\mu,x)  := -\frac 1 N \log \int \sum_{\al \in \N^k} \exp \Ll( H_N^t(\si) + H_N^\mu(\si,\al) + H_N^x(\si,\al) \Rr)  \, v_\al \d P_N(\si),
\end{equation}
as well as
\begin{equation*}  %\label{e.}
\bar F_N(t,\mu,x) := \E \Ll[ F_N(t,\mu,x) \Rr].
\end{equation*}
In the last two displays, we slightly abuse notation in that we keep denoting the free energy by $F_N$ (or $\bar F_N$ for its average), although there are now additional variables compared to the quantity defined in \eqref{e.def.FN}. This abuse of notation does not seem to risk causing much confusion. Indeed, every identity we have seen so far is still valid if $F_N(t,\mu)$ is replaced by $F_N(t,\mu,x)$, provided that we redefine the Gibbs measure in \eqref{e.def.gibbs} to include the perturbation terms. Moreover, whenever a risk of confusion arises, we can always write the variables explicitly to dispel it.

We now record a few identities involving the derivatives of $F_N$ and $\bar F_N$ with respect to this new variable~$x$. We have
\begin{equation}
\label{e.dr.F.xah}
\dr_{x_{a,h}} F_N = - N^{-1 - \frac 1 {16}} \la H_N^{a,h}(\si,\al) \ra,
\end{equation}
and, by \eqref{e.cov.HNah} and Gaussian integration by parts, see \eqref{e.gibp1},
\begin{align}  
\notag
\dr_{x_{a,h}} \bar F_N 
& =  -N^{-1 - \frac 1 {16}} \E \la H_N^{a,h}(\si,\al) \ra 
\\
\label{e.dr.F.xah2}
& =  N^{-\frac 1 8} x_{a,h} \,   \E \la  \Ll( \lambda_{h_1} \frac{ \si_a \cdot \si_a'}{N} + \lambda_{h_2} \frac{\al \wedge \al'}{k}  \Rr)^{h_3} - \Ll( \lambda_{h_1} \frac{|\si_a|^2 }{N} + \lambda_{h_2}  \Rr)^{h_3} \ra.
\end{align}
In particular, recalling that $\lambda_{h_1}, \lambda_{h_2} \in [0,1]$, we have
\begin{equation}  
\label{e.dr.barF.xah}
\Ll| \dr_{x_{a,h}} \bar F_N  \Rr| \le 2^{h_3+2} N^{-\frac 1 8} |x_{a,h}|.
\end{equation}
Similarly, for every $a \in \{1,2\}$,
\begin{equation}
\label{e.dr.F.xa}
\dr_{x_{a}} F_N = -N^{-1 -\frac 1 {16}} \la |\si_a|^2 \ra,
\end{equation}
and in particular,
\begin{equation}
\label{e.dr.barF.xa}
\Ll|\dr_{x_{a}} \bar F_N \Rr| \le N^{-\frac 1 {16}}.
\end{equation}
We also have
\begin{equation}  
\label{e.dr2.xph}
\dr_{x_{a,h}}^2 \bar F_N =  -N^{-1-\frac 1 8} \E \Ll[\la  \Ll( H_N^{a,h}(\si,\al) \Rr) ^2  \ra - \la  H_N^{a,h}(\si,\al) \ra^2\Rr] ,
\end{equation}
and
\begin{equation}  
\label{e.dr2.xa}
\dr_{x_{a}}^2 \bar F_N = -N^{-1-\frac 1 8} \E  \Ll[\la  |\si_a|^4  \ra - \la  |\si_a|^2 \ra^2\Rr].
\end{equation}
Before we turn to the proof of Theorem~\ref{t.approx.HJ}, we record a useful concentration estimate for the function $F_N$. 

\begin{proposition}[Concentration of $F_N$]
\label{p.concentration}
Let $k,h_+ \in \N_*$ and, for every $(t,q,x) \in \R_+ \times \bar U_k \times \R^{2+h_+^3}$, let
\begin{equation}
\label{e.def.FNk}
F_N^{(k)}(t,q,x) := F_N \Ll( t,  \,\frac 1 k \sum_{\ell = 1}^k \de_{q_{1,\ell}} , \, \frac 1 k \sum_{\ell = 1}^k \de_{q_{2,\ell}} , \, x\Rr),
\end{equation}
as well as
\begin{equation}
\label{e.def.bar.FNk}
\bar F_N^{(k)}(t,q,x) := \E \Ll[ F_N^{(k)}(t,q,x) \Rr] .
\end{equation}
For every $M < \infty$, $p \in [1,\infty)$ and $\ep > 0$, there exists $C  < \infty$ such that for every $N \ge 1$,
\begin{equation*}  %\label{e.}
\E \bigg[ \sup_{B_M} \Ll| F_N^{(k)} - \bar F_N^{(k)} \Rr|^p \bigg]^\frac 1 p \le C N^{-\frac 1 2 + \ep},
\end{equation*}
where $B_M$ denotes the set of $(t,q,x) \in \R_+ \times \bar U_k \times \R^{2+h_+^3}$ for which each coordinate is contained in $[-M,M]$ (in other words, $B_M$ is the intersection of $\R_+ \times \bar U_k \times \R^{2+h_+^3}$ with the $(3+2k+h_+^3)$-dimensional $L^\infty$ ball of radius $M$). 
\end{proposition}
\begin{proof}
By \cite[Theorem~1.2]{pan}, there exists $C < \infty$ such that for every $(t,q,x) \in  B_M$ and $a \ge 0$,
\begin{equation}  
\label{e.concentration.FN}
\P \Ll[  \Ll| (F_N^{(k)} - \bar F_N^{(k)}) (t,q,x) \Rr|^2 \ge \frac{a}{N} \Rr] \le 2 \exp \Ll( - \frac{a}{C} \Rr) .
\end{equation}
In order to conclude, we need some estimate on the modulus of continuity of $F_N^{(k)}$. We denote
\begin{equation*}  %\label{e.}
X := 1 + \frac{\Ll|\la H_N(\si) \ra\Rr|}{N} + \frac 1 N\sum_{a \in \{1,2\}} \Ll( \sum_{\ell = 0}^k \Ll|\la z_{\al_{|\ell},a} \cdot \si_a \ra\Rr| + \sum_{h \in \{1,\ldots, h_+\}^3} \Ll|\la H_N^{a,h}(\si,\al) \ra \Rr|\Rr) .
\end{equation*}
By integration of \eqref{e.drt.FN}, we see that for every $t,t' \in [0,M]$, $q \in \bar U_k$, and $x \in \R^{2+h_+^3}$,
\begin{align*}  %\label{e.}
\Ll|F_N^{(k)}(t',q,x) - F_N^{(k)}(t,q,x) \Rr| 
& \le C \Ll(1+\frac{\Ll|\la H_N(\si) \ra\Rr|}{N} \Rr) |t'-t|^\frac 1 2
\\
& \le C X |t'-t|^\frac 1 2.
\end{align*}
Similarly, we can compute, for every $\ell \in \{1,\ldots, k-1\}$,
\begin{equation*}  %\label{e.}
\dr_{q_{a,\ell}} F_N^{(k)} = -\frac 1 N \la (2q_{a,\ell} - 2q_{a,\ell-1})^{-\frac 12} z_{\al_{|\ell},a} \cdot \si_a - (2q_{a,\ell+1} - 2q_{a,\ell})^{-\frac 12} z_{\al_{|\ell+1},a} \cdot \si_a \ra,
\end{equation*}
with the understanding that $q_{a,0} = 0$ here, and, in the case $\ell = k$,
\begin{equation*}  %\label{e.}
\dr_{q_{a,k}} F_N^{(k)} = -\frac 1 N \la (2q_{a,k} - 2q_{a,k-1})^{-\frac 12} z_{\al_{|k},a} \cdot \si_a\ra.
\end{equation*}
By integration, we find that, for every $q,q' \in \bar U_k$, $t \ge 0$, and $x \in \R^{2+h_+^3}$,
\begin{equation*}  %\label{e.}
\Ll|F_N^{(k)}(t,q',x) - F_N^{(k)}(t,q,x) \Rr| \le C X \sum_{a \in \{1,2\}} \sum_{\ell = 1}^k |q'_{a,\ell}-q_{a,\ell}|^\frac 1 2.
\end{equation*}
Finally, by \eqref{e.dr.F.xah} and \eqref{e.dr.F.xa}, we also have that
\begin{equation*}  %\label{e.}
\Ll| F_N^{(k)}(t,q,x')  - F_N^{(k)}(t,q,x)\Rr| \le C N^{-\frac 1 {16}} X |x'-x|. 
\end{equation*}
On the other hand, it follows from \eqref{e.drt.barFN}, \eqref{e.drq.barFN}, \eqref{e.dr.barF.xah}, and \eqref{e.dr.barF.xa}, that the function $\bar F_N^{(k)}$ is Lipschitz continuous (globally in $t$ and $q$, and locally in $x$). For every $\ep \in (0,1]$, we denote
\begin{equation*}  %\label{e.}
A_\eps := \Ll(\ep \Z^{3 + 2k+h_+^3}\Rr) \cap B_M .
\end{equation*}
The previous estimates imply that
\begin{equation*}  %\label{e.}
\sup_{B_M} \Ll| F_N^{(k)} - \bar F_N^{(k)} \Rr| \le \sup_{A_\ep} \Ll| F_N^{(k)} - \bar F_N^{(k)} \Rr| + C X \sqrt{\ep},
\end{equation*}
and therefore, for every $p \ge 1$, 
\begin{equation*}  %\label{e.}
\E \bigg[ \sup_{B_M} \Ll| F_N^{(k)} - \bar F_N^{(k)} \Rr|^p \bigg] \le C \, \E \bigg[ \sup_{A_\ep} \Ll| F_N^{(k)} - \bar F_N^{(k)} \Rr|^p \bigg] + C \ep^\frac p 2 \E[X^p],
\end{equation*}
with a constant $C < \infty$ that may depend on $p$ (in addition to $k,h_+$, and $M$).
We bound the supremum over $A_\ep$ by the sum over $A_\ep$ and use \eqref{e.concentration.FN} to get 
\begin{equation*}  %\label{e.}
\E \bigg[ \sup_{A_\ep} \Ll| F_N^{(k)} - \bar F_N^{(k)} \Rr|^p \bigg] \le C |A_\ep| N^{-\frac p 2} = C \ep^{-\Ll( 3 + 2k+h_+^3 \Rr) } N^{-\frac p 2}.
\end{equation*}
Using \eqref{e.gibpp}, we see that, for every $(t,q,x) \in B_M$,
\begin{equation*}  %\label{e.}
\E \la (H_N(\si))^{2p} \ra \le C N^{2p},
\end{equation*}
and similarly, for every $a \in \{1,2\}$ and $\ell \in \{0,\ldots,k\}$, 
\begin{equation*}  %\label{e.}
\E \la (z_{\al_{|\ell,a}} \cdot \si_a)^{2p} \ra \le C N^{2p},
\end{equation*}
as well as, for every $h \in \{1,\ldots, h_+\}^3$,
\begin{equation*}  %\label{e.}
\E \la (H_N^{a,h}(\si,\alpha))^{2p} \ra \le C N^{2p}. 
\end{equation*}
By Jensen's inequality, this implies that $\E [X^p] \le C$ (in other words, $\E[X^p]$ is bounded uniformly over $N$). We have thus shown that, with $\al :=3 + 2k+h_+^3$,
\begin{equation*}  %\label{e.}
\E \bigg[ \sup_{B_M} \Ll| F_N^{(k)} - \bar F_N^{(k)} \Rr|^p \bigg]^\frac 1 p \le  C \ep^{-\frac \al p } N^{-\frac 1 2} + C \ep^\frac 1 2. 
\end{equation*}
Choosing $\ep = N^{-\frac p {p + 2\al}}$, we can bound the right side above by $C N^{-\frac{p}{2(p+2\al)}}$. By taking $p$ sufficiently large, we can bring the exponent $\frac{p}{2(p+2\al)}$ as close to $\frac 1 2$ as desired. By Jensen's inequality, this proves the claim.
\end{proof}
%\jccomment{by tweaking the argument a little bit, it seems possible to show that any limit function actually satisfies the equation everywhere in the interior! Because in the interior, we can always touch the function from one side, and then by the perturbation thing we can get that the touching test function satisfies the equation at the contact point at arbitrary precision. The trick is to let $\alpha$ be the more abstract Ruelle variable with uniform overlap distribution over $\Ll[0,1\Rr]$, and then mesh $\alpha$ as finely as desired, for each fixed value of $k$.}
\begin{proof}[Proof of Theorem~\ref{t.approx.HJ}]
We fix the integer $h_+$ sufficiently large that, with the choice of $\ep = k^{-4}$, the statement of Proposition~\ref{p.cond.exp} holds for some $\de \ge h_+^{-1}$. Recall that we slightly abuse notation and write $\bar F_N^{(k)}$ both to denote the function in \eqref{e.def.barFNk} and the function in \eqref{e.def.bar.FNk}. We can dispel the confusion by writing $(t,q) \mapsto \bar F_N^{(k)}(t,q)$ for the former and $(t,q,x) \mapsto \bar F_N^{(k)}(t,q,x)$ for the latter. In order to lighten the notation, we drop the superscript $^{(k)}$ and simply write $\bar F_N$ in place of $\bar F_N^{(k)}$ throughout (and similarly for $F_N^{(k)}$).
Let $f$ be a subsequential limit of the mapping $(t,q) \mapsto \bar F_N(t,q)$. For convenience, we omit to denote the particular subsequence along which the convergence of $(t,q) \mapsto \bar F_N(t,q)$ to~$f$ holds.

%We fix $U$ a bounded open subset of $(0,\infty) \times \T_k \times \R^{2+h_0^3}$ containing $(t_\infty,q_\infty,0)$. 

Let $(t_\infty,q_\infty) \in (0,\infty) \times \bar U_k$ and $\phi \in C^\infty((0,\infty)\times \bar U_k)$ be such that $f - \phi$ has a local minimum at $(t_\infty,q_\infty)$. 
Without loss of generality, we may assume that 
\begin{equation}
\label{e.boundary.Neumann}
q_\infty \in \dr \bar U_k \quad \implies \quad \max_{\nu \in \n(q_\infty)} \nu \cdot \nabla \phi(t_\infty,q_\infty) < 0,
\end{equation}
since in the complementary event, the Neumann boundary condition is satisfied. Under this condition, we will show that
\begin{equation}  
\label{e.goal}
\Ll(\dr_t \phi - k \sum_{\ell= 1}^k \dr_{q_{1,\ell}} \phi \, \dr_{q_{2,\ell}} \phi\Rr)(t_\infty,q_\infty) \ge -\frac {13}{k}. 
\end{equation}
Throughout the rest of this proof, we denote by $C < \infty$ a constant whose value may change from one occurence to another, and is allowed to depend on $k$, $h_+$ (which itself has already been fixed in terms of $k$), $t_\infty$, $q_\infty$, and the function $\phi$.
We write 
\begin{equation}
\label{e.def.xinfty}
x_\infty := (1,\ldots, 1) \in \R^{2+2h_+^3}.
\end{equation}
For every $(t,q) \in (0,\infty) \times \bar U_k$, and $x \in \R^{2+2h_+^3}$, we set
\begin{equation}  
\label{e.def.tdphi}
\td \phi(t,q,x) := \phi(t,q) - (t-t_\infty)^2 - |q-q_\infty|^2 - |x-x_\infty|^2.
\end{equation}
The mapping $(t,q,x) \mapsto f(t,q)- \td \phi(t,q,x)$ has a strict local minimum at $(t_\infty,q_\infty,x_\infty)$. In view of \eqref{e.dr.barF.xah} and \eqref{e.dr.barF.xa}, the mapping $(t,q,x) \mapsto \bar F_N(t,q,x)$ converges to the mapping $(t,q,x) \mapsto f(t,q)$ locally uniformly. We deduce that there exist $(t_N,q_N,x_N) \in (0,\infty) \times \bar U_k  \times \R^{2+2h_+^3}$ satisfying
\begin{equation}
\label{e.lim.tN}
\lim_{N \to \infty} (t_N,q_N,x_N) = (t_\infty,q_\infty,x_\infty)
\end{equation}
and such that, for every $N$ sufficiently large, the function $\bar F_N - \td \phi$ has a local minimum at $(t_N,q_N,x_N)$; more precisely, for every $N$ sufficiently large,
\begin{multline}  
\label{e.loc.min}
(\bar F_N - \td \phi)(t_N,q_N,x_N) 
\\
= \inf \Ll\{ (\bar F_N - \td \phi)(t,q,x) : \ |t - t_N| + |q - q_N|+ |x - x_N| \le C^{-1} \Rr\}  .
\end{multline}
(In the infimum above, we also have the implicit restriction $q \in \bar U_k$, and we may choose $C$ sufficiently large that the condition $|t-t_N| \le C^{-1}$ implies that $t > 0$.)
In particular,
\begin{equation}
\label{e.nablaFN.nablat}
\dr_t (\bar F_N-\td \phi)(t_N,q_N,x_N) = 0,
\end{equation}
\begin{equation}
\label{e.nablaFN.nablaphi.int}
q_N \in U_k \quad \implies \quad \nabla_q (\bar F_N-\td \phi)(t_N,q_N,x_N) = 0,
\end{equation}
and 
\begin{equation}
\label{e.nablaFN.nablaphi2}
\nabla_{x}  (\bar F_N-\td \phi)(t_N,q_N,x_N) = 0.
\end{equation}
We decompose the rest of the proof into five steps.

\emph{Step 1.} In this step, we show that for every $N$ sufficiently large and $|x| \le C^{-1}$,
\begin{equation}  
\label{e.C11.in.q}
-C|x|^2 \le \bar F_N(t_N,q_N,x_N + x) - \bar F_N(t_N,q_N,x_N) - x \cdot \nabla_x \bar F_N(t_N,q_N,x_N)  \le 0.
\end{equation}
The second inequality follows from the fact that $\bar F_N$ is a concave function of $x$ (it is classical to verify that the function $F_N$ itself is concave in $x$, since the Hessian of this function is a covariance matrix, up to a minus sign). To show the first inequality in~\eqref{e.C11.in.q}, we start by writing Taylor's formula:
\begin{multline}  
\label{e.Taylor.FN}
\bar F_N(t_N,q_N,x_N+x) - \bar F_N(t_N,q_N,x_N) 
\\
= x \cdot \nabla_x \bar F_N(t_N,q_N,x_N) 
+ 
\int_0^1 (1-s) x \cdot \nabla^2_x \bar F_N(t_N,q_N,x_N + sx ) x \, \d s,
\end{multline}
where $\nabla_x^2 \bar F_N$ denotes the Hessian of the function $\bar F_N$ in the $x$ variable. Naturally, the formula above is also valid if we replace $\bar F_N$ by $\td \phi$. By \eqref{e.loc.min}, we have that for every $|x| \le C^{-1}$,
\begin{equation*}  %\label{e.}
\bar F_N(t_N,q_N,x_N + x) - \bar F_N(t_N,q_N,x_N) \ge \td \phi(t_N,q_N,x_N + x) - \td \phi(t_N,q_N,x_N) .
\end{equation*}
Using also \eqref{e.nablaFN.nablaphi2}, we obtain that
\begin{multline*}  %\label{e.}
\int_0^1 (1-s) x \cdot \nabla^2_x \bar F_N(t_N,q_N, x_N + s x)x \, \d s 
\\
\ge \int_0^1 (1-s) x \cdot \nabla^2_x \td \phi(t_N,q_N, x_N + s x)x \, \d s \ge - C |x|^2.
\end{multline*}
Combining this with \eqref{e.Taylor.FN} yields \eqref{e.C11.in.q}.

\emph{Step 2.} We show that, for every $\ep > 0$,
\begin{equation}
\label{e.estim.annealed.fluc}
\E \Ll[ \Ll| \nabla_x(F_N - \bar F_N)(t_N,q_N,x_N) \Rr| ^2   \Rr]  \le C N^{-\frac 1 2 + \ep}, 
\end{equation}
where now we also allow the constant $C < \infty$ to depend on the choice of $\ep > 0$. 
As observed in the previous step, the function $F_N$ is concave in the $x$ variable. We thus have, for every $x \in \R^{2 + 2 h_+^3}$,
\begin{equation*}  %\label{e.}
F_N(t_N,q_N,x_N + x) \le F_N(t_N,q_N,x_N) + x \cdot \nabla_x F_N(t_N,q_N,x_N).
\end{equation*}
By \eqref{e.C11.in.q}, we also have, for every $|x| \le C^{-1}$,
\begin{equation*}  %\label{e.}
\bar F_N(t_N,q_N,x_N + x) \ge \bar F_N(t_N,q_N,x_N) + x \cdot \nabla_x \bar F_N(t_N,q_N,x_N) - C |x|^2.
\end{equation*}
For a (deterministic) parameter $\lambda \in [0,C^{-1}]$ to be determined in the course of the argument, we combine the two inequalities above and fix
\begin{equation*}  %\label{e.}
x = \lambda \, \frac{\nabla_x(\bar F_N - F_N)(t_N,q_N,x_N)}{|\nabla_x(\bar F_N - F_N)(t_N,q_N,x_N)|},
\end{equation*}
so that $|x| \le C^{-1}$, and, for this choice of $x$,
\begin{multline*}  %\label{e.}
\lambda  \Ll| \nabla_x(\bar F_N - F_N) (t_N,q_N,x_N)\Rr| 
\\
\le (\bar F_N - F_N)(t_N,q_N,x_N + x) - (\bar F_N - F_N)(t_N,q_N,x_N) + C \lambda^2.
\end{multline*}
By Proposition~\ref{p.concentration}, we infer that
\begin{equation*}  %\label{e.}
\lambda^2 \E \Ll[ \Ll| \nabla_x(\bar F_N - F_N) (t_N,q_N,x_N)\Rr|^2 \Rr] 
\le 
C N^{-1+2\ep} + C \lambda^4.
\end{equation*}
Choosing $\lambda = N^{-\frac 1 4+\frac{\ep}2}$ yields \eqref{e.estim.annealed.fluc}.

\emph{Step 3.} We show that the Gibbs measure associated with the choice of parameters $(t_N,q_N,x_N)$ satisfies approximate Ghirlanda-Guerra identities, in the following sense. Recall that we denote by $(\si^\ell, \al^\ell)_{\ell \ge 1}$ a family of independent copies of $(\si,\al)$ under $\la \cdot \ra$. For each $\ell, \ell' \in \N_*$ and $a \in \{1,2\}$, we write
\begin{equation*}  %\label{e.}
R^{\ell,\ell'}_0 := \frac{\al^\ell \wedge \al^{\ell'}}{k}, \quad R^{\ell,\ell'}_{a} := \frac{\si_a^{\ell} \cdot \si_a^{\ell'}}{N},  
\end{equation*}
and, for each $n \in \N_*$, we denote by $R^{\le n}$ the array
\begin{equation*}  %\label{e.}
R^{\le n} := \Ll( R^{\ell,\ell'}_{a} \Rr)_{a \in \{0,1,2\}, \ell, \ell' \in \{1,\ldots, n\}}.
\end{equation*}
In this step, we show that, for every $\ep > 0$, $a \in \{1,2\}$, $n, h_1,h_2,h_3 \in \{1,\ldots, h_+\}$, and $g \in C({\R^{3n^2}})$ satisfying $\|g\|_{L^\infty} \le 1$, we have
\begin{multline}  
\label{e.gg2.explicit}
 \bigg| \E \la f ( R^{\le n} ) \Ll(\lambda_{h_1} R_a^{1,n+1}+ \lambda_{h_2} R_0^{1,n+1}\Rr)^{h_3} \ra  - \frac 1 n \E \la g(R^{\le n})  \ra \E \la \Ll(\lambda_{h_1} R_a^{1,2}+ \lambda_{h_2} R_0^{1,2}\Rr)^{h_3} \ra 
 \\
  - \frac 1 n\sum_{\ell = 2}^n \E \la g(R^{\le n}) \Ll(\lambda_{h_1} R_a^{1,\ell}+ \lambda_{h_2} R_0^{1,\ell}\Rr)^{h_3}  \ra \bigg| 
\le 
C N^{-\frac 1 8+ \ep},
\end{multline}
where we understand that the Gibbs measure $\la \cdot \ra$ is with the parameters $(t_N,q_N,x_N)$. 
It follows from \eqref{e.C11.in.q} that
\begin{equation*}  %\label{e.}
-C \le \nabla_x^2 \bar F_N(t_N,q_N,x_N) \le 0.
\end{equation*}
In particular, by \eqref{e.dr2.xph}, for every $a \in \{1,2\}$ and $h = (h_1,h_2,h_3) \in \{1,\ldots, h_+\}^3$, we have 
\begin{equation*}  %\label{e.}
\E \Ll[\la  \Ll( H_N^{a,h}(\si) \Rr) ^2  \ra - \la  H_N^{a,h}(\si) \ra^2\Rr] \le C N^{1+\frac 1 8},
\end{equation*}
and similarly, by \eqref{e.dr2.xa},
\begin{equation*}  
\E  \Ll[\la  |\si_a|^4  \ra - \la  |\si_a|^2 \ra^2\Rr] \le C N^{1 + \frac 1 8}.
\end{equation*}
By \eqref{e.dr.F.xah} and \eqref{e.estim.annealed.fluc}, we also have
\begin{equation*}  %\label{e.}
\E \Ll[ \Ll(\la H_N^{a,h}(\si) \ra - \E \la H_N^{a,h}(\si) \ra\Rr)^2 \Rr] \le C N^{\frac 3 2+\frac 1 8 + \ep},
\end{equation*}
and similarly, by \eqref{e.dr.F.xa} and \eqref{e.estim.annealed.fluc}, 
\begin{equation*}  %\label{e.}
\E \Ll[ \Ll(\la |\si_a|^2 \ra - \E \la |\si_a| \ra\Rr)^2 \Rr] \le C N^{\frac 3 2+\frac 1 8+ \ep}.
\end{equation*}
Since, for any random variable $X$, we have the variance decomposition
\begin{equation*}  %\label{e.}
\E \la \Ll( X - \E \la X \ra \Rr) ^2 \ra = \E \la \Ll( X - \la X \ra \Rr) ^2\ra + \E\Ll[ \Ll( \la X \ra - \E \la X \ra \Rr) ^2\Rr],
\end{equation*}
we deduce that
\begin{equation}  
\label{e.concentr.Hnph}
\E \la \Ll( H_N^{a,h}(\si) - \E \la H_N^{a,h}(\si) \ra \Rr)^2 \ra\le C N^{\frac 3 2 + \frac 1 8 + \ep},
\end{equation}
and
\begin{equation}  
\label{e.concentr.sia}
\E \la \Ll( |\si_a|^2 - \E \la |\si_a|^2 \ra \Rr)^2 \ra\le C N^{\frac 3 2 + \frac 1 8+ \ep}.
\end{equation}
It follows from \eqref{e.concentr.Hnph} that
\begin{equation*}  %\label{e.}
\Ll| \E \la g(R^{\le n}) H_N^{a,h}(\si^1,\al^1) \ra - \E \la g(R^{\le n}) \ra \E \la  H_N^{a,h}(\si^1,\al^1) \ra  \Rr| \le C N^{\frac 3 4 + \frac 1 {16} + \ep}.
\end{equation*}
Recall the expression for $\E \la H_N^{a,h}(\si,\al) \ra$ in \eqref{e.dr.F.xah2}. By Gaussian integration by parts, see~\eqref{e.gibp1}, we also have
\begin{multline*}  %\label{e.}
N^{-1 + \frac 1 {16}}\E \la g(R^{\le n}) H_N^{a,h}(\si^1,\al^1) \ra = \sum_{\ell = 1}^n x_{a,h} \E \la g(R^{\le n}) \Ll( \lambda_{h_1} \frac{ \si_a^1 \cdot \si_a^\ell}{N} + \lambda_{h_2} \frac{\al^1 \wedge \al^\ell}{k}  \Rr)^{h_3}\ra 
\\
- n x_{a,h} \E \la  g(R^{\le n})\Ll( \lambda_{h_1} \frac{ \si_a^1 \cdot \si_a^{n+1}}{N} + \lambda_{h_2} \frac{\al^1 \wedge \al^{n+1}}{k}  \Rr)^{h_3}\ra,
\end{multline*}
where we dropped the dependence on $N$ and simply wrote $x_{a,h}$ for the $(a,h)$ coordinate of the vector $x_N$. Recall that $x_N \to x_\infty$ with $x_\infty$ defined in \eqref{e.def.xinfty}, so that for $N$ sufficiently large, we have $x_{a,h} \ge \frac 1 2$ (this is the point of defining $x_\infty$ in this way, as opposed to setting $x_\infty = 0$). 
Matching the term indexed by $\ell = 1$ in the sum above with the last term in~\eqref{e.dr.F.xah2}, we would like to show that the difference 
\begin{equation*}  %\label{e.}
\E \la g(R^{\le n}) \Ll( \lambda_{h_1} \frac{ |\si_a^1|^2}{N} + \lambda_{h_2}   \Rr)^{h_3}\ra - \E \la g(R^{\le n}) \ra \, \E \la \Ll( \lambda_{h_1} \frac{ |\si_a^1|^2}{N} + \lambda_{h_2}   \Rr)^{h_3}\ra 
\end{equation*}
is small. Using \eqref{e.concentr.sia} and the fact that the mapping $r \mapsto r^{h_3}$ is Lipschitz over $[0,2]$, we can bound this difference (in absolute value) by $C N^{-\frac 1 4+ \frac 1 {16}+ \ep}$.
Collecting the terms, we obtain~\eqref{e.gg2.explicit}.

\emph{Step 4.} We now use the synchronization result of Section~\ref{s.synch}. Recall from \eqref{e.approxN.hj} that 
\begin{equation*}  %\label{e.}
\Ll|\dr_t \bar F_N - k \sum_{\ell = 1}^k  \dr_{q_{1,\ell}} \bar F_N \, \dr_{q_{2,\ell}} \bar F_N \Rr| \le \frac 1 {N^2} \sum_{a \in \{1,2\}} \E \la \Ll( \si_a\cdot \si_a' - \E \la \si_a \cdot \si_a' \vb \al \wedge \al' \ra  \Rr)^2 \ra.
\end{equation*}
By \eqref{e.gg2.explicit}, Proposition~\ref{p.cond.exp} (with the quantities $R_1$ and $R_2$ appearing there being substituted by $R_a$ and $R_0$ in our current notation), and our choice of $h_+$, we infer that for every $N$ sufficiently large,
\begin{equation}  
\label{e.equation.barfn}
\Ll|\dr_t \bar F_N - k \sum_{\ell = 1}^k  \dr_{q_{1,\ell}} \bar F_N \, \dr_{q_{2,\ell}} \bar F_N \Rr|(t_N,q_N,x_N) \le \frac {13} k.
\end{equation}
We will argue in the next step that, for every $N$ sufficiently large, 
\begin{equation}
\label{e.comp.H}
\Ll(\sum_{\ell = 1}^k  \dr_{q_{1,\ell}} \bar F_N \, \dr_{q_{2,\ell}} \bar F_N  - \sum_{\ell = 1}^k  \dr_{q_{1,\ell}} \td \phi \, \dr_{q_{2,\ell}} \td \phi \Rr)(t_N,q_N,x_N) \ge 0.
\end{equation}
Temporarily assuming this, and using also \eqref{e.nablaFN.nablat}, we thus infer that
\begin{equation}  
\label{e.equation.tdphi}
\Ll(\dr_t \td \phi - k \sum_{\ell = 1}^k  \dr_{q_{1,\ell}} \td \phi \, \dr_{q_{2,\ell}} \td \phi\Rr)(t_N,q_N,x_N)  \ge  -\frac {13} k.
\end{equation}
Using \eqref{e.lim.tN} and the fact that $\td \phi$ is a smooth function, we deduce that the statement \eqref{e.equation.tdphi} also holds at $(t_\infty,q_\infty,x_\infty)$. Recalling also the definition of~$\td \phi$, see~\eqref{e.def.tdphi}, we conclude that \eqref{e.goal} holds. 

\emph{Step 5.} There only remains to show that \eqref{e.comp.H} holds. If $q_N \in U_k$, then this is immediate, by \eqref{e.nablaFN.nablaphi.int}. In particular, since $q_N$ converges to $q_\infty$ as $N$ tends to infinity, we know that~\eqref{e.comp.H} holds for every $N$ sufficiently large whenever $q_\infty \in U_k$. From now on, we assume that $q_\infty \in \dr U_k$. Using that $(t_N,q_N,x_N)$ tends to $(t_\infty, q_\infty,x_\infty)$, the smoothness of the function $\phi$, and the definition of $\td \phi$ in \eqref{e.def.tdphi}, we can infer from \eqref{e.boundary.Neumann} that, for every~$N$ sufficiently large, 
\begin{equation}  
\label{e.Neumann.2}
\max_{\nu \in \n(q_\infty)} \nu \cdot \nabla \td \phi(t_N,q_N,x_N) \le 0.
\end{equation}
We now argue that
\begin{equation}  
\label{e.Neumann.3}
q_N \in \dr \bar U_k \quad \implies \quad \max_{\nu \in \n(q_N)} \nu \cdot \nabla \td \phi(t_N,q_N,x_N) \le 0.
\end{equation}
The set $\bar U_k$ can be written as the intersection of $2k$ half-spaces, and the condition that $q_\infty \in \dr U_k$ is equivalent to the statement that $q_\infty$ lies on the boundary of some of those half-spaces, say $\mcl D_1,\ldots, \mcl D_\ell$. For $N$ sufficiently large, we know that $q_N$ will not be on the boundary of any other half-space than those $\mcl D_1,\ldots, \mcl D_\ell$. It thus follows that, for $N$ sufficiently large, and whenever $q_N \in \bar U_k$, we have that $\n(q_N)$ is a subset of $\n(q_\infty)$. This yields that \eqref{e.Neumann.3} indeed follows from \eqref{e.Neumann.2}. 

We now argue that 
\begin{equation}
\label{e.belong.nablaphi}
\nabla \td \phi (t_N,q_N,x_N) \in \bar U_k.
\end{equation}
Since the arguments $t_N$ and $x_N$ will be kept fixed, we omit them from the notation. We denote by $(e_{a,\ell})_{a \in \{1,2\}, \ell \in \{1,\ldots k\}}$ the canonical basis of $\R^{2k}$, using our indexing convention. We also write $q_N = (q_{a,\ell})_{a \in \{1,2\}, 1 \le \ell \le k}$, dropping the dependence on $N$ when writing $q_N$ in coordinates. We fix $a \in \{1,2\}$, and first show that $\dr_{q_{a,1}} \td \phi(q_N) \ge 0$. If $q_{a,1} = 0$, then this follows from~\eqref{e.Neumann.3}, since $-e_{a,1} \in \n(q_N)$ in this case. Otherwise, since the function $\bar F_N - \td \phi$ has a local minimum at $q_N$, we have, for every $\ep > 0$ sufficiently small,
\begin{equation*}  %\label{e.}
\td \phi(q_N-\ep e_{a,1}) - \td \phi(q_N) \le \bar F_N(q_N-\ep e_{a,1}) - \bar F_N(q_N).
\end{equation*}
Passing to the limit $\ep \to 0$, and using Lemma~\ref{l.pos.dr_q}, we obtain that $\dr_{q_{q,1}} \td \phi(q_N) \ge 0$. We now fix $\ell \in \{1,\ldots, k-1\}$, and argue that $\dr_{q_{a,\ell}} \td \phi(q_N) \le \dr_{q_{a,\ell+1}} \td \phi(q_N)$. If $q_{a,\ell} = q_{a,\ell+1}$, then this follows from \eqref{e.Neumann.3}, since $2^{-\frac 1 2}(e_{a,\ell} - e_{a,\ell+1}) \in \n(q_N)$ in this case. Otherwise, we have, for every $\ep > 0$ sufficiently small,
\begin{equation*}  %\label{e.}
\td \phi(q_N+\ep e_{a,\ell} - \ep e_{a,\ell+1}) - \td \phi(q_N) \le \bar F_N(q_N+\ep e_{a,\ell} - \ep e_{a,\ell+1}) - \bar F_N(q_N).
\end{equation*}
Passing to the limit, and using Lemma~\ref{l.pos.dr_q}, we obtain indeed that $\dr_{q_{a,\ell}} \td \phi(q_N) \le \dr_{q_{a,\ell+1}} \td \phi(q_N)$.

We are now ready to conclude. For every $p = (p_{a,\ell})_{a \in \{1,2\}, \ell \in \{1,\ldots, k\}} \in \R^{2k}$, denote
\begin{equation*}  %\label{e.}
\msf H(p) := \sum_{\ell = 1}^k p_{1,\ell} \, p_{2,\ell}. 
\end{equation*}
Our aim is to show that 
\begin{equation}  
\label{e.diff.H}
\msf H(\nabla_q \bar F_N) - \msf H(\nabla_q \td \phi) \ge 0,
\end{equation}
where we kept implicit that the gradients are evaluated at $(t_N,q_N,x_N)$. We rewrite the left side of \eqref{e.diff.H} in the form
\begin{equation}  
\label{e.rewrite.diff}
\int_0^1 (\nabla_q \bar F_N - \nabla_q \td \phi)\cdot \nabla \msf H(s \nabla_q \bar F_N + (1-s) \nabla_q \td \phi) \, \d s.
\end{equation}
(Evaluation at $(t_N,q_N,x_N)$ is still kept implicit here.)
Since  $\bar F_N - \td \phi$ has a local minimum at $q_N$, we must have that
\begin{equation}  
\label{e.normal.cone}
\forall y \in \bar U_k, \qquad 
(y-q_N) \cdot \nabla_q (\bar F_N - \td \phi)(t_N,q_N,x_N) \ge 0. 
\end{equation}
Recalling the notation $\bar U_k^*$ from \eqref{e.def.baruk*}, we infer from \eqref{e.normal.cone}  that 
\begin{equation}
\label{e.barUk*}
\nabla_q (\bar F_N - \td \phi)(t_N,q_N,x_N) \in \bar U_k^*.
\end{equation}
(Indeed, for every $z \in \bar U_k$, we can choose $y = q_N +z \in \bar U_k$ in \eqref{e.normal.cone}.) By \eqref{e.belong.nablaphi} and Lemma~\ref{l.pos.dr_q}, we have that, for every $s \in [0,1]$,
\begin{equation*}  %\label{e.}
\Ll(s  \nabla_q \bar F_N + (1-s) \nabla_q \td \phi\Rr)(t_N,q_N,x_N) \in \bar U_k.
\end{equation*}
Finally, notice that, for each $p = (p_{a,\ell})_{a \in \{1,2\}, \ell \in \{1,\ldots, k\}}$, the vector $\nabla \msf H(p)$ is obtained by a simple interchange of the index $a \in \{1,2\}$. In particular, it is clear that $\nabla \msf H$ maps $\bar U_k$ into itself. Combining this with the previous display, we infer that the quantity $\nabla \msf H(\cdots)$ appearing in \eqref{e.rewrite.diff} belongs to $\bar U_k$. This and \eqref{e.barUk*} yield \eqref{e.diff.H}, as desired.
\end{proof}

We can now prove the main theorem of the paper.

\begin{proof}[Proof of Theorem~\ref{t.main.full}]
The argument consists in combining the results of Proposition~\ref{p.conv.finite.dim} and Theorem~\ref{t.approx.HJ}. We denote by $(f^{(k)})_{k \ge 1}$ and $f$ the functions appearing in the statement of Proposition~\ref{p.conv.finite.dim}. By Theorem~\ref{t.approx.HJ} and the comparison principle \eqref{e.comp}, we have, for every integer $k \ge 1$, $t \ge 0$ and $q \in \R^{2k}$,
\begin{equation}  
\label{e.main.estim1}
\liminf_{N \to \infty} \bar F_N\Ll(t,\frac 1 k \sum_{\ell = 1}^k \de_{q_{1,\ell}} , \frac 1 k \sum_{\ell = 1}^k \de_{q_{2,\ell}} \Rr) + \frac{13t}{k} \ge f^{(k)}(t,q). 
\end{equation}
Let $\mu = (\mu_1,\mu_2) \in (\mcl P_2(\R_+))^2$, and, for every integer $k \ge 1$, $a \in \{1,2\}$, and $\ell \in \{1,\ldots, k\}$, denote
\begin{equation*}  %\label{e.}
q_{a,\ell}^{(k)} := k \int_{\frac{\ell-1}{k}}^{\frac \ell k}F_{\mu_a}^{-1} (u) \, \d u , \quad \text{ and } \quad \mu_a^{(k)} := \sum_{\ell = 1}^k \de_{q_{a,\ell}}. 
\end{equation*}
Recall from \eqref{e.quant.conv} that
\begin{equation}  
\label{e.main.estim2}
\Ll| f^{(k)}(t,q^{(k)}) - f(t,\mu) \Rr| \le \frac{C}{\sqrt{k}}\Ll( t + \Ll(\E \Ll[ X_{\mu_1}^2 + X_{\mu_2}^2 \Rr]\Rr) ^\frac 1 2   \Rr) .
\end{equation}
On the other hand, we have from Proposition~\ref{p.continuity} that 
\begin{align*}  %\label{e.}
\Ll|\bar F_N(t,\mu) - \bar F_N(t,\mu^{(k)})\Rr| 
& \le \sum_{a = 1}^2 \E \Ll[ \Ll|X_{\mu_a} - X_{\mu_a^{(k)}} \Rr|\Rr] 
\\
& = \sum_{a = 1}^2 \E \Ll[ \Ll|F_{\mu_a}^{-1}(U) - k \int_{\frac{\lfloor kU \rfloor}{k}}^{\frac{\lfloor kU \rfloor+1}{k}}F_{\mu_a}^{-1}(u) \, \d u\Rr|\Rr] .
\end{align*}
We can bound this term by arguing as in the paragraph starting with \eqref{e.diffy}. Indeed, this is the same argument, with the understanding that $K'$ is now infinite. Explicitly, 
\begin{align*}  %\label{e.}
 \E \Ll[ \Ll|F_{\mu_a}^{-1}(U) -  k \int_{\frac{\lfloor kU \rfloor}{k}}^{\frac{\lfloor kU \rfloor+1}{k}}F_{\mu_a}^{-1}(u) \, \d u\Rr| \Rr] 
& =  \sum_{\ell = 1}^{k} \int_{\frac{\ell-1}{k}}^{\frac{\ell}{k}}\Ll|F_{\mu_a}^{-1}(r) - k \int_{\frac{\ell-1}{k}}^{\frac{\ell}{k}} F_{\mu_a}^{-1}(u) \, \d u\Rr|  \, \d r
\\
& \le k\sum_{\ell = 1}^{k} \int_{\frac{\ell-1}{k}}^{\frac{\ell}{k}}\int_{\frac{\ell-1}{k}}^{\frac{\ell}{k}}\Ll|F_{\mu_a}^{-1}(r) -  F_{\mu_a}^{-1}(u) \Rr|   \, \d u \, \d r .
\end{align*}
Paralleling \eqref{e.diffy.2}, we write, for a cutoff value $B \in (0,\infty)$ to be determined, 
\begin{equation*}  %\label{e.}
\int_0^1 |F_{\mu_a}^{-1} (u)| \1_{\{F_{\mu_a}^{-1} (u) \ge B\}} \, \d u \le \frac 1 B \int |F_{\mu_a}^{-1} (u)|^2 \, \d u = \frac{\E[X_{\mu_a}^2]}{B},
\end{equation*}
while, as in \eqref{e.diffy.3},
\begin{align*}  %\label{e.}
& 2k\sum_{\ell = 1}^{k} \int_{\frac{\ell-1}{k}}^{\frac{\ell}{k}}\int_{\frac{\ell-1}{k}}^{\frac{\ell}{k}}\Ll|F_{\mu_a}^{-1}(r) -  F_{\mu_a}^{-1}(u) \Rr|    \1_{\{u \le r, F_{\mu_a}^{-1}(r) \le B\}}\, \d u \, \d r 
\\
& \qquad =2k\sum_{\ell = 1}^{k} \int_{\Ll[0,\frac 1 k\Rr]^2}\Ll(F_{\mu_a}^{-1}\Ll(\frac{\ell-1}{k} +r\Rr) -  F_{\mu_a}^{-1}\Ll(\frac{\ell-1}{k} +u\Rr) \Rr)    \1_{\{u \le r, F_{\mu_a}^{-1}\Ll(\frac{\ell-1}{k} +r\Rr) \le B\}}\, \d u \, \d r  
\\
& \qquad \le \frac{2B}{k}. 
\end{align*}
Combining the displays above, and choosing $B^2 = k(\E[X_{\mu_1}^2 + X_{\mu_2}^2])$, we arrive at
\begin{equation}  
\label{e.main.estim3}
\Ll|\bar F_N(t,\mu) - \bar F_N(t,\mu^{(k)})\Rr| \le \frac{6}{\sqrt{k}} \Ll( \E[X_{\mu_1}^2 + X_{\mu_2}^2] \Rr) ^\frac 1 2.
\end{equation}
Combining \eqref{e.main.estim1}, \eqref{e.main.estim2}, and \eqref{e.main.estim3}, we deduce that
\begin{equation*}  %\label{e.}
\liminf_{N \to \infty} \bar F_N(t,\mu) \ge f(t,\mu) -\frac{13t}{k}- \frac{C}{\sqrt{k}} \Ll(t + \Ll( \E[X_{\mu_1}^2 + X_{\mu_2}^2] \Rr) ^\frac 1 2\Rr).
\end{equation*}
Letting the integer $k\ge 1$ tend to infinity, we obtain the desired result.
\end{proof}

%
%
%
%%%%%%%%%%%%%%%%%%%%%%%%%%%%
%%%%%%%%%%%%%%%%%%%%%%%%%%%%
%
%
%
\section{Synchronization}
\label{s.synch}

In this section, we revisit the synchronization result of \cite{pan.multi}, see also \cite{pan.potts,pan.vec}. The structure of the reasoning presented here is similar to that in \cite{pan.multi}, and emphasizes the fundamental importance of the ultrametric structure of the Gibbs measure. There are a few differences though: one of them is that we state ``finitary'' versions of the statements; that is, the statements provide approximate criteria that the Gibbs measure may satisfy for large but finite values of $N$ and $k$; the conclusion is then that we have ``synchronization up to a small error''. A second difference between the treatment presented here and \cite{pan.multi} is in the phrasing of the synchronization property itself. In \cite{pan.multi}, this is stated as the existence of Lipschitz functions that each map the sum of the overlaps of the different species to one of the single-species overlaps. In the present section, we instead choose to phrase the synchronization of different overlaps as the statement that they are monotonically coupled.

As said above, the main powerhouse behind the synchronization result comes from the possibility to enforce the ultrametricity of the Gibbs measure. The fundamental result of~\cite{pan.aom} is that the ultrametricity property is valid as soon as the Ghirlanda-Guerra identities hold; see also the preface to \cite{pan} for a review of the series of works that preceded this final result. Moreover, as is well-known and was seen again in Section~\ref{s.serious}, these identities are valid as soon as certain random energy functions become concentrated, a property that one can ``build into the measure'' by means of a small perturbation of the energy function.

In order to emphasize that the underlying constants in the statements below do not depend on the specific Gibbs measure under consideration, we will state them for rather general measures. We start by stating a finitary version of the statement from \cite{pan.aom} that ``Ghirlanda-Guerra identities imply ultrametricity''.

\begin{theorem}[GG implies ultrametricity \cite{pan.aom}]
\label{t.gg.ultr}
For every $\ep > 0$, there exists $\de > 0$ such that the following holds. Let $G$ be a random probability measure supported on the unit ball of an arbitrary Hilbert space; denote by $\la \cdot \ra$ the expectation associated with the measure $G^{\otimes \N}$, with canonical random variables $(\si^\ell)_{\ell \ge 1}$, and define, for every $\ell,\ell', n \ge 1$,
\begin{equation*}  %\label{e.}
R^{\ell,\ell'} := \sigma^{\ell} \cdot \sigma^{\ell'} , \quad \text{ and } \quad  R^{\le n} := \Ll( R^{\ell,\ell'} \Rr) _{1 \le \ell , \ell' \le n}.
\end{equation*}
Finally, recalling that $\la \cdot \ra$ is itself random, denote by $\E$ the expectation with respect to this additional source of randomness. Assume that, for every $n, p \in \{1,\ldots, \lfloor \de^{-1} \rfloor \}$ and $f \in C(\R^{n\times n})$ satisfying $\|f\|_{L^\infty} \le 1$,  
\begin{equation}  
\label{e.gg.delta}
\Ll| \E \la f ( R^{\le n} ) (R^{1,n+1})^p \ra  - \frac 1 n \E \la f(R^{\le n})  \ra \E \la (R^{1,2})^p \ra - \frac 1 n\sum_{\ell = 2}^n \E \la f(R^{\le n}) (R^{1,\ell})^p  \ra \Rr| \le \de.
\end{equation}
Then 
\begin{equation}  
\label{e.ultr}
\E \la \1_{\{R^{1,2} \ge \min \Ll( R^{1,3}, R^{2,3} \Rr) - \ep \}}\ra \ge 1-\ep.
\end{equation}
\end{theorem}
\begin{proof}%[Proof of Theorem~\ref{t.gg.ultr}]
We argue by contradiction. Denote by $R := (R^{\ell,\ell'})_{\ell, \ell' \ge 1}$ the entire overlap array, and assume that Theorem~\ref{t.gg.ultr} is false: there exists $\ep > 0$ and, for each $\de > 0$ no matter how small, a random probability measure $G$ such that \eqref{e.gg.delta} holds but \eqref{e.ultr} is violated. Since each entry of $R$ takes values in $[-1,1]$, up to extraction of a subsequence, we can find a random array $\mathsf R = (\mathsf R^{\ell,\ell'})_{\ell , \ell' \ge 1}$ defined with respect to a certain probability measure $\mathbb{M}$ such that, for each integer $n \ge 1$, the law of the array $\mathsf R^{k} := (\msf R^{\ell,\ell'})_{1 \le \ell , \ell' \le n}$ under $\mathbb{M}$ is obtained as the limit law of a subsequence of overlap arrays, each violating \eqref{e.ultr} but satisfying \eqref{e.gg.delta} for a sequence of values of $\delta$ that tends to zero. In other words, the array $\msf R$ satisfies, for every integers $n,p \ge 1$ and $f \in C(\R^{n\times n})$,
\begin{equation*}  %\label{e.}
\M \Ll[f ( \msf R^{\le n} ) (\msf R^{1,n+1})^p \Rr]  = \frac 1 n \M\Ll[ f(\msf R^{\le n})  \Rr] \M \Ll[ (\msf R^{1,2})^p \Rr] + \frac 1 n\sum_{\ell = 2}^n \M \Ll[ f(\msf R^{\le n}) (\msf R^{1,\ell})^p  \Rr] ,
\end{equation*}
as well as
\begin{equation*}  %\label{e.}
\M \Ll[ R^{1,2} \le \min \Ll( R^{1,3}, R^{2,3} \Rr) - \ep \Rr] \ge \ep.
\end{equation*}
This was shown to be impossible in \cite{pan.aom}, see also \cite[Theorem~2.14]{pan}.
\end{proof}

In order to prepare the ground for synchronization statements, we clarify the notion of monotone coupling in the next proposition.

\begin{proposition}[Monotone coupling]
\label{p.mono.coupl}
Let $(X,Y)$ be a random vector taking values in~$\R^2$, and let $(X',Y')$ be an independent copy of this vector, defined under the probability measure $\P$. The following three statements are equivalent.

(1) We have
\begin{equation}  
\label{e.mono.coupl.ass}
\P \Ll[ X < X' \text{ and } \  Y' < Y \Rr] = 0.
\end{equation}

(2) For every $x, y \in \R$, we have
\begin{equation}  
\label{e.mono.coupl}
\P \Ll[ X \le x \text{ and } Y \le y \Rr] = \min \Ll( \P \Ll[ X \le x \Rr] , \P \Ll[ Y \le y \Rr]  \Rr) .
\end{equation}

(3) The law of $(X,Y)$ is 
\begin{equation*}  %\label{e.}
(F_X^{-1}, F_Y^{-1}) \Ll( \mathrm{Leb}_{ \Ll[ 0,1 \Rr] } \Rr) ,
\end{equation*}
that is, the law of $(X,Y)$ is the image of the Lebesgue measure over $[0,1]$ under the mapping $r \mapsto (F_X^{-1}(r), F_Y^{-1}(r))$, where, for every $r \in [0,1]$,
\begin{equation}  
\label{e.cdfinv}
F_X^{-1}(r) := \inf \Ll\{ s \in \R \ : \ \P \Ll[ X \le s \Rr] \ge r \Rr\} ,
\end{equation}
and similarly with $X$ replaced by $Y$. 
\end{proposition}
Whenever any of the conditions (1-3) appearing in Proposition~\ref{p.mono.coupl} holds, we say that the random variables $X$ and $Y$ are \emph{monotonically coupled}.
\begin{proof}
We first show that (1) implies (2). 
The statement \eqref{e.mono.coupl} with the equality sign replaced by ``$\, \le \,$'' is clear. To show the converse inequality, we argue by contradiction and assume that there exist $x,y \in \R$ such that
\begin{equation*}  %\label{e.}
\P \Ll[ X \le x \text{ and } Y \le y \Rr] < \min \Ll( \P \Ll[ X \le x \Rr] , \P \Ll[ Y \le y \Rr]  \Rr) .
\end{equation*}
It follows that 
\begin{equation*}  
%\label{e.mono.coupl.csq1}
\P \Ll[ X \le x \text{ and } Y > y \Rr] = \P[X \le x] - \P[X \le x \text{ and } Y \le y]> 0,
\end{equation*}
and similarly,
\begin{equation*}  
%\label{e.mono.coupl.csq2}
\P \Ll[ X > x \text{ and } Y \le y \Rr] = \P [Y \le y] - \P[X \le x \text{ and } Y \le y] > 0.
\end{equation*}
In particular,
\begin{equation*}  %\label{e.}
\P \Ll[ X \le x < X' \text{ and } Y' \le y < Y \Rr]  > 0.
\end{equation*}
This contradicts \eqref{e.mono.coupl.ass}.

We now show that (3) implies (1). Let $U$ and $U'$ be two independent random variables distributed uniformly over $[0,1]$. We can realize $(X,Y)$ and $(X',Y')$ by setting
\begin{equation*}  %\label{e.}
(X,Y) = (F_X^{-1}, F_Y^{-1})(U) \quad \text{and} \quad (X',Y') = (F_X^{-1}, F_Y^{-1})(U').
\end{equation*}
For definiteness, suppose that $U \le U'$. Since $F_X^{-1}$ and $F_Y^{-1}$ are increasing (in the sense of wide inequalities), it then implies that $X \le X'$ and $Y \le Y'$. This shows that property (1) holds.

Summarizing, we have shown that (1) implies (2) and (3) implies (1). In particular, (3) implies (2). Since there is at most one joint law for $(X,Y)$ that satisfies (2), we deduce that (2) and (3) are equivalent. The proof is thus complete.
%We now show that the conditions (2) and (3) are equivalent. It is clear that the property (2) characterizes the joint law of $(X,Y)$ uniquely. Indeed, one way to verify this is to observe that, for every smooth and compactly supported function $f \in C_c^\infty(\R^2)$, since
%\begin{equation*}  %\label{e.}
%f(X,Y) = \int_{\R^2} \dr^2_{xy} f(x,y) \1_{\{ X \le x \text{ and } Y \le y \}} \, \d x \, \d y,
%\end{equation*}
%we have
%\begin{equation*}  %\label{e.}
%\E \Ll[ f(X,Y) \Rr]  = \int \dr_{xy}^2 f(x,y) \P[X \le x \text{ and } Y \le y] \, \d x \, \d y.
%\end{equation*}
%It thus suffices to verify that the law proposed in (3) satisfies the property (2). Assuming that the law of $(X,Y)$ is as in (3), we have
%\begin{equation*}  %\label{e.}
%\E \Ll[ f(X,Y) \Rr] = \int_0^1 f \Ll( F_X^{-1}(r), F_Y^{-1}(r) \Rr) \, \d r.
%\end{equation*}
\end{proof}

We now turn to our variant of the main result of \cite{pan.multi}, which states that approximate Ghirlanda-Guerra identities imply approximate synchronization, in the sense of monotone couplings between overlaps.

\begin{theorem}[Synchronization]
\label{t.sync}
Let $(\lambda_{n})_{n \ge 1}$ be an enumeration of the set of rational numbers in $[0,1]$. 
For every $\ep > 0$, there exists $\de > 0$ such that the following holds. 
Let $G$ be a random probability measure supported on the Cartesian product of the unit balls of two arbitrary Hilbert spaces; denote by $\la \cdot \ra$ the expectation associated with the measure $G^{\otimes \N}$, with canonical random variables $\Ll(\si^\ell = (\si^\ell_1, \si^\ell_2)\Rr)_{\ell \ge 1}$, and define, for every $a \in \{1,2\}$ and $\ell,\ell', n \ge 1$,
\begin{equation*}  %\label{e.}
R_a^{\ell,\ell'} := \sigma_a^{\ell} \cdot \sigma_a^{\ell'} , \quad \text{ and } \quad  R^{\le n} := \Ll( R_a^{\ell,\ell'} \Rr) _{a \in \{1,2\}, 1 \le \ell , \ell' \le n}.
\end{equation*}
Finally, recalling that $\la \cdot \ra$ is itself random, denote by $\E$ the expectation with respect to this additional source of randomness. Assume that, for every $n,h_1, h_2, p \in \{1,\ldots, \lfloor \de^{-1} \rfloor \}$ and $f \in C(\R^{2\times n\times n})$ satisfying $\|f\|_{L^\infty} \le 1$,  
\begin{multline}  
\label{e.gg2.delta}
\bigg| \E \la f ( R^{\le n} ) \Ll(\lambda_{h_1} R_1^{1,n+1}+ \lambda_{h_2} R_2^{1,n+1}\Rr)^p \ra  \\
- \frac 1 n \E \la f(R^{\le n})  \ra \E \la \Ll(\lambda_{h_1} R_1^{1,2}+ \lambda_{h_2} R_2^{1,2}\Rr)^p \ra - \frac 1 n\sum_{\ell = 2}^n \E \la f(R^{\le n}) \Ll(\lambda_{h_1} R_1^{1,\ell}+ \lambda_{h_2} R_2^{1,\ell}\Rr)^p  \ra \bigg| \le \de.
\end{multline}
Then, for every $f \in C^\infty(\R^2)$,
\begin{equation}  
\label{e.sync}
\Ll| \E \la f(R^{1,2}_1, R^{1,2}_2) \ra - \E \Ll[ f\Ll(F_1^{-1}(U), F_2^{-1}(U) \Rr)\Rr]  \Rr| \le \ep \Ll( \|f\|_{L^\infty} + \|\nabla f\|_{L^\infty} \Rr) ,
\end{equation}
where $U$ stands for a uniform random variable over $[0,1]$, and, for every $a \in \{1,2\}$ and $r \in [0,1]$, we write
\begin{equation}  
\label{e.cdfinv.a}
F_a^{-1} (r) := \inf \Ll\{ s \in \R \ : \ \E \la \1_{\{ R^{1,2}_a \le s\}} \ra \ge r \Rr\}.
\end{equation}
\end{theorem}
The proof of Theorem~\ref{t.sync} makes use of the following lemma, asserting  that if two sequences of random variables converge in law separately, then their monotone coupling converges in law as well. 
\begin{lemma}[continuity of monotone coupling]
\label{l.continuity.mono.coupl}
Let $(X_n), (Y_n)$ be two sequences of random variables which converge in law to $X$ and $Y$ respectively. Then the associated monotone couplings converge: using the notation in \eqref{e.cdfinv}, and with $U$ a uniform random variable over $[0,1]$, we have
\begin{equation}  
\label{e.continuity.mono.coupl}
\Ll(F_{X_n}^{-1}(U), F_{Y_n}^{-1}(U)\Rr) \xrightarrow[n \to \infty]{\text{(law)}} \Ll(F_{X}^{-1}(U), F_{Y}^{-1}(U)\Rr).
\end{equation}
\end{lemma}
\begin{proof}
Since the law of $F_{X_n}^{-1}(U)$ is that of $X_n$, it is clear that the convergence in \eqref{e.continuity.mono.coupl} holds for each coordinate separately. Up to the extraction of a subsequence, we can assume that $\Ll(F_{X_n}^{-1}(U), F_{Y_n}^{-1}(U)\Rr)$ converges in law to some random vector $(A,B)$; we denote by $(A',B')$ an independent copy of this vector. By classical properties of convergence in law and Proposition~\ref{p.mono.coupl}, we infer that 
\begin{equation*}  %\label{e.}
\P \Ll[ A < A' \text{ and } B' < B \Rr] = 0.
\end{equation*}
We conclude using Proposition~\ref{p.mono.coupl} once more.
\end{proof}
\begin{proof}[Proof of Theorem~\ref{t.sync}]

\emph{Step 1.} For any two probability measures $\mu,\nu$ on $[-1,1]^2$, we define
\begin{equation*}  %\label{e.}
\|\mu - \nu\| := \sup \Ll\{ \int f \, \d \mu - \int f \, \d \nu \ : \ \|f\|_{L^{\infty}} + \|\nabla f\|_{L^\infty} \le 1 \Rr\} .
\end{equation*}
In this step, we show that the quantity above, as a function of $(\mu,\nu)$, is continuous for the topology of weak convergence. In other words, if a sequence of probability measures $\mu_n$ over $[-1,1]^2$ converges weakly to $\mu$, then $\|\mu_n - \nu\|$ converges to~$\|\mu - \nu\|$. For every integer $k \ge 1$ and $x \in \R^2$, define
\begin{equation*}  %\label{e.}
P_k(x) :=  c_k\Ll( 1 - \frac{|x|^2}{16} \Rr) ^k,
\end{equation*}
where the constant $c_k$ is such that $\int_{\Ll[-2,2\Rr]^2} P_k = 1$. Let $f \in C^\infty([-1,1]^2)$ be such that $\|f\|_{L^\infty} + \|\nabla f\|_{L^\infty} \le 1$. We may extend $f$ to a Lipschitz function on $\R^2$ such that $\|f\|_{L^\infty} + \|\nabla f\|_{L^\infty} \le 2$. 
Denoting the spatial convolution by $\ast$, we have, for every $x \in [-1,1]^2$,
\begin{align*}  %\label{e.}
(f - f \ast P_k)(x) 
& = \int_{\Ll[-2,2\Rr]^2} (f(x) - f(x-y)) P_k(y) \, \d y
\\
& = \int_{\Ll[-2,2\Rr]^2} \int_0^1 y \cdot \nabla f(x - ty) P_k(y) \, \d t \, \d y,
\end{align*}
so
\begin{equation*}  %\label{e.}
\|f - f \ast P_k\|_{L^\infty} \le 2 \int_{\Ll[-2,2\Rr]^2} |y| \, P_k(y) \, \d y ,
\end{equation*}
and the latter quantity tends to $0$ as $k$ tends to infinity (uniformly over $f$). On the other hand, $f \ast P_k$ is a polynomial of degree at most $2k$ and, and for each fixed $k$, the coefficients of this polynomial can be bounded in terms of $\|f\|_{L^\infty}$. In particular, for each fixed $k$, we have
\begin{equation*}  %\label{e.}
\sup \Ll\{ \int f \ast P_k \, \d \mu_n - \int f \ast P_k \, \d \mu \ : \  \|f\|_{L^\infty} + \|\nabla f\|_{L^\infty} \le 1  \Rr\} \xrightarrow[n \to \infty]{} 0.
\end{equation*}
Combining these two facts gives the announced continuity result.

\emph{Step 2.} We need to show that, provided that $\de > 0$ is chosen sufficiently small in terms of $\ep$, we have
\begin{equation}  
\label{e.dist.laws}
\Ll\|\mathrm{Law}(R_1^{1,2}, R_2^{1,2}) - (F_1^{-1}, F_2^{-1})\Ll( \mathrm{Leb}_{ \Ll[ 0,1 \Rr] } \Rr) \Rr\| \le \ep.
\end{equation}
In the expression above, we denote by $\mathrm{Law}(R_1^{1,2}, R_2^{1,2})$ the law of $(R_1^{1,2}, R_2^{1,2})$ under the measure $\E \la \cdot \ra$.
Assuming the contrary, there exist $\ep > 0$ and, for $\de > 0$ as small as desired, an overlap distribution satisfying \eqref{e.gg2.delta} but not \eqref{e.dist.laws}. Up to extraction of a subsequence, we can assume that the overlap array converges in law to a limit random overlap $\msf R$, whose law we denote by $\M$. By Lemma~\ref{l.continuity.mono.coupl} and the result of the previous step, we infer that
\begin{equation*}  %\label{e.}
\Ll\|\mathrm{Law}(\msf R_1^{1,2}, \msf R_2^{1,2}) - (F_1^{-1}, F_2^{-1})\Ll( \mathrm{Leb}_{ \Ll[ 0,1 \Rr] } \Rr) \Rr\| \ge \ep,
\end{equation*}
where in the expression above, $F_1^{-1}$ and $F_2^{-1}$ now stand for the inverse cumulative distribution functions of $\msf R_1^{1,2}$ and $\msf R_2^{1,2}$ respectively (that is, we replace $\E \la \1_{\{ R^{1,2}_a \le s\}} \ra $ by $\M \Ll[\1_{\{ \msf R^{1,2}_a \le s\}}\Rr]$ in \eqref{e.cdfinv.a}). In particular, the random variables $\msf R^{1,2}_1$ and $\msf R^{1,2}_2$ are not monotonically coupled.

\emph{Step 3.} We now show that $\msf R^{1,2}_1$ and $\msf R^{1,2}_2$ are in fact monotonically coupled, thereby reaching a contradiction. Denote by $\td {\msf R}^{1,2}$ an independent copy of $\msf R^{1,2}$. (Notice that this is with respect to the ``averaged'' measure $\M$, so $\msf R^{3,4}$ would not qualify as an independent copy of $\msf R^{1,2}$ in this sense.)
By Proposition~\ref{p.mono.coupl}, we need to show that 
\begin{equation}  
\label{e.needto.synchr}
\M \Ll[ \msf R_1^{1,2} < \td {\msf R}_1^{1,2} \text{ and } \td {\msf R}_2^{1,2} < \msf R_2^{1,2} \Rr] = 0.
\end{equation}
We first observe that, by the construction of $\msf R$, we have that for all integers $n,h_1, h_2,p \ge 1$ and $f \in C(\R^{2\times n\times n})$,
\begin{multline*}  %\label{e.}
\M \Ll[f ( \msf R^{\le n} ) \Ll(\lambda_{h_1}\msf R_1^{1,n+1}+ \lambda_{h_2} 
\msf R_2^{1,n+1}\Rr)^p \Rr]  \\
= \frac 1 n \M\Ll[ f(\msf R^{\le n})  \Rr] \M\Ll[ \Ll(\lambda_{h_1} \msf R_1^{1,2}+ \lambda_{h_2} \msf R_2^{1,2}\Rr)^p \Rr] + \frac 1 n\sum_{\ell = 2}^n \M \Ll[ f(\msf R^{\le n}) \Ll(\lambda_{h_1} \msf R_1^{1,\ell}+ \lambda_{h_2} \msf R_2^{1,\ell}\Rr)^p  \Rr] .
\end{multline*}
Since every continuous function can be uniformly approximated by a polynomial on compact sets, and using the Cram\'er-Wold theorem, we deduce that conditionally on~$\msf R^{\le n}$, the law of $\msf R^{1,n+1}$ is 
\begin{equation*}  %\label{e.}
\frac 1 n \mathrm{Law}(\msf R^{1,2}) + \frac 1 n \sum_{\ell = 2}^n \de_{\msf R^{1,\ell}},
\end{equation*}
 where $\mathrm{Law}(\msf R^{1,2})$ denotes the law of $\msf R^{1,2}$ under $\M$, and $\de_{\msf R^{1,\ell}}$ is the Dirac mass at $\msf R^{1,\ell}$. In particular,
\begin{align*}  %\label{e.}
& 2 \M \Ll[  \msf R_1^{1,2} < {\msf R}_1^{1,3} \text{ and } {\msf R}_2^{1,3} < \msf R_2^{1,2}  \Rr] 
\\
& \qquad  = \M \Ll[ \msf R_1^{1,2} < \td {\msf R}_1^{1,2} \text{ and } \td {\msf R}_2^{1,2} < \msf R_2^{1,2} \Rr] + \M \Ll[ \msf R_1^{1,2} < {\msf R}_1^{1,2} \text{ and } {\msf R}_2^{1,2} < \msf R_2^{1,2} \Rr]
\\
& \qquad =  \M \Ll[ \msf R_1^{1,2} < \td {\msf R}_1^{1,2} \text{ and } \td {\msf R}_2^{1,2} < \msf R_2^{1,2} \Rr].
\end{align*}
The statement \eqref{e.needto.synchr} we aim to show is thus equivalent to
\begin{equation}
\label{e.needto2.synchr}
\M \Ll[  \msf R_1^{1,2} < {\msf R}_1^{1,3} \text{ and } {\msf R}_2^{1,3} < \msf R_2^{1,2}  \Rr] = 0.
\end{equation}
The validity of \eqref{e.needto2.synchr} now follows from the fact that $\msf R_1$, $\msf R_2$, and $\msf R_1 + \msf R_2$ are ultrametric, which itself is a consequence of Theorem~\ref{t.gg.ultr}. Indeed, by ultrametricity, we have 
\begin{equation*}  %\label{e.}
\msf R_1^{1,2} < \msf R_1^{1,3} \implies \msf R_1^{2,3} = \msf R_1^{1,2},
\end{equation*}
and 
\begin{equation*}  %\label{e.}
\msf R_2^{1,3} < \msf R_2^{1,2} \implies \msf R_2^{2,3} = \msf R_2^{1,3},
\end{equation*}
so that
\begin{equation*}  %\label{e.}
\msf R_1^{1,2} < \msf R_1^{1,3} \text{ and } \msf R_2^{1,3} < \msf R_2^{1,2} \implies \msf R_1^{2,3} + \msf R_2^{2,3} < \min \Ll( \msf R_1^{1,2} + \msf R_2^{1,2}, \msf R_1^{1,3} + \msf R_2^{1,3} \Rr) ,
\end{equation*}
and the latter statement contradicts the ultrametricity of $\msf R_1 + \msf R_2$. This completes the proof of \eqref{e.needto2.synchr}, and therefore of Theorem~\ref{t.sync}.
\end{proof}

As was apparent in \eqref{e.approxN.hj}, what we ultimately want to use is not only that two overlaps asymptotically become monotonically coupled, but rather that one of the overlaps can essentially be inferred by observing the other. Even if the two overlaps were perfectly synchronized, this can only be true if the law of the observed overlap is sufficiently ``spread out'': in an extreme example, if the observed overlap is deterministic, then the statement of monotone coupling is uninformative, and the conditional variances in \eqref{e.approxN.hj} boil down to regular variances, which need not be small. In the next proposition, we give a precise statement to this effect. That is, we show that if the law of one of the overlaps is sufficiently spread out, then the conditional variance of the other overlap is small. The usefulness of writing finitary versions of the statements of ultrametricity and synchronization appears most clearly here.

\begin{proposition}[Control of conditional variance]
\label{p.cond.exp}
For every $\ep > 0$, there exists $\de > 0$ such that the following holds. Let $\E$, $\la \cdot \ra$, $(R^{\ell,\ell'}_a)$ be as in the statement of Theorem~\ref{t.sync}, and assume that \eqref{e.gg2.delta} holds for every $n,h_1,h_2,p \in \{1,\ldots, \lfloor \de^{-1} \rfloor\}$ and $f \in C(\R^{2\times n\times n})$ satisfying $\|f\|_{L^\infty} \le 1$. Assume furthermore that the law of $R_2^{1,2}$ is of the form
\begin{equation*}  %\label{e.}
\frac 1 {k} \sum_{\ell = 1}^k \de_{q_\ell},
\end{equation*}
for some integer $k \ge 1$ and parameters $-1 =q_0 < q_1 < \cdots < q_k \le 1$. We then have
\begin{equation}
\label{e.cond.exp}
\E \la \Ll(R_1^{1,2} - \E \la R_1^{1,2} \vb R_2^{1,2} \ra \Rr)^2 \ra \le \frac {12} k + \ep k^2  \sup_{\ell \in \{0,\ldots, k-1\}} (q_{\ell + 1} - q_\ell)^{-1} .
\end{equation}
\end{proposition}

\begin{proof}[Proof of Proposition~\ref{p.cond.exp}]
By Theorem~\ref{t.sync}, we can choose $\de > 0$ sufficiently small that~\eqref{e.sync} holds for every Lipschitz function $f :\R^2 \to \R$. 
We set
\begin{equation*}  %\label{e.}
\eta := \inf_{\ell \in \{0,\ldots, k-1\}} (q_{\ell + 1} - q_\ell),
\end{equation*}
and, for every $a \in \{1,2\}$ and $s \in \R$, 
\begin{equation*}  %\label{e.}
F_a(s) := \E \la \1_{\{ R_a^{1,2} \le s \}}\ra,
\end{equation*}
with $F_a^{-1}$ defined as in \eqref{e.cdfinv.a}. 
By assumption, the function $F_2$ is piecewise constant, with discontinuities at $q_1,\ldots, q_k$.
Let $\tilde F_2$ denote the function which coincides with $F_2$ on the set $(-\infty,q_0] \cup \{q_0,\ldots,q_k\} \cup [q_k,+\infty)$, and is affine on each interval $[q_\ell,q_{\ell+1}]$, $\ell \in \{0,\ldots,k-1\}$. The function $\tilde F_2$  satisfies $\|\tilde F_2\|_{L^\infty} \le 1$ and $\|\nabla \tilde F_2\|_{L^\infty} \le \eta^{-1}$. Notice that, for every $u \in [0,1]$,
\begin{equation}  
\label{e.obs.F2F2}
\tilde F_2(F_2^{-1}(u)) = F_2(F_2^{-1}(u)) = k^{-1} {\lceil k u \rceil}.
\end{equation}
We define, for every $x \in \R$,
\begin{equation*}  %\label{e.}
\rho_k(x) := k\max(1-|kx|, 0),
\end{equation*}
and observe that $\int \rho_k = 1$ and that the convolution $F_1^{-1} \ast \rho_k$ is a Lipschitz function, with Lipschitz constant bounded by $k^2$.
 For every $x,y \in [-1,1]$, we set
\begin{equation*}  %\label{e.}
f(x,y) = (x-(F_1^{-1} \ast \rho_k)(\td F_2(y)))^2.
\end{equation*}
The Lipschitz constant of this function is bounded by $2 \eta^{-1} k^2$, and thus
\begin{equation*}  %\label{e.}
\Ll| \E \la f(R_1^{1,2}, R_2^{1,2})\ra - \E \Ll[ f (F_1^{-1}(U), F_2^{-1}(U)) \Rr]  \Rr| \le 3 \ep \eta^{-1}k^2 .
\end{equation*}
Using \eqref{e.obs.F2F2} and the fact that $F_1^{-1}$ takes values in $[-1,1]$ and is monotone, we can estimate the second term on the left side above by
\begin{align*}  %\label{e.}
\E \Ll[ \Ll( F_1^{-1}(U) - (F_1^{-1} \ast \rho_k)(k^{-1} \lceil kU \rceil) \Rr) ^2 \Rr] 
& \le 2\E \Ll[ \Ll| F_1^{-1}(U) - (F_1^{-1} \ast \rho_k)(k^{-1} \lceil kU \rceil) \Rr|  \Rr] 
\\
& = 2\sum_{\ell = 0}^{k-1} \int_{\frac \ell k}^{\frac{\ell+1}{k}} \Ll| F_1^{-1}(u) - (F_1^{-1} \ast \rho_k)(k^{-1} \lceil ku \rceil) \Rr| \, \d u
\\
& \le \frac 2 k\sum_{\ell = 0}^{k-1} \Ll(F_1^{-1} \Ll( \frac{\ell+2}{k} \Rr) - F_1^{-1} \Ll( \frac{\ell-1}{k} \Rr) \Rr)
\\
& \le \frac {12} k.
\end{align*}
We have thus shown that
\begin{equation*}  %\label{e.}
\E \la \Ll(R_1^{1,2} - (F_1^{-1} \ast \rho_k)(\td F_2(R_2^{1,2})) \Rr)^2 \ra \le 3 \ep \eta^{-1} k^2 + 12 k^{-1},
\end{equation*}
and thus in particular, since the conditional expectation is an $L^2$ projection,
\begin{equation*}  %\label{e.}
\E \la \Ll(R_1^{1,2} - \E \la  R_1^{1,2} \vb R_2^{1,2} \ra \Rr)^2 \ra \le 3 \ep \eta^{-1} k^2 + 12 k^{-1}.
\end{equation*}
Up to a redefinition of $\ep$, this is \eqref{e.cond.exp}.
\end{proof}

%
%
%
%%%%%%%%%%%%%%%%%%%%%%%%%%%%
%%%%%%%%%%%%%%%%%%%%%%%%%%%%
%
%
%

\section{The free energy as a saddle-point problem?}
\label{s.concave}

This final section has a more speculative flavor, and concerns the possibility to rewrite the limit free energy of models such as the one investigated here in the form of a saddle-point problem, a possibility discussed for instance in \cite{obnoxious}. A strong indication in favor of this possibility comes from the study of certain models of statistical inference. The statistical-inference problem most similar to the spin-glass model studied here is probably that of estimating a non-symmetric rank-one matrix. This problem was investigated in~\cite{mio.nonsym,bmm,ree,HB1}, and it was found there that the free energy could indeed be conveniently represented in the form of a saddle-point problem.

Of course, any quantity can be written as a saddle-point problem, so the relevant question is whether there is some natural way for doing so. The point of view provided by Hamilton-Jacobi equations suggests two natural routes for finding variational formulations of the limit free energy. The first one, available only when the nonlinearity in the equation is convex, consists in writing the Hopf-Lax formula for the solution, see for instance \cite{parisi}. As was already emphasized, the main feature of the model under consideration here is that the nonlinearity in the equation is \emph{not} convex (nor concave). The second possible route is based on the fact that, irrespectively of the structure of the nonlinearity, it is also possible to write the solution of a Hamilton-Jacobi equation as a saddle-point problem, \emph{provided that the initial condition is concave (or convex)}, as was suggested also by Hopf in \cite{hopf} and then confirmed rigorously using the notion of viscosity solutions in  \cite{bareva} (see also \cite{lioroc}).

This second possibility can be applied to good effect in the context of the model of statistical inference studied in \cite{mio.nonsym,bmm,ree,HB1}: as was shown in \cite{HB1}, the relevant Hamilton-Jacobi equation is a finite-dimensional version of \eqref{e.hj}, and the initial condition is convex, thereby allowing to recover the saddle-point formulas obtained in \cite{mio.nonsym,bmm,ree}. 

However, perhaps surprisingly, this strategy does not seem to work in the context of the model under consideration in this paper, and it is the aim of this section to explore this more precisely. 

This point hides an important subtelty, which requires that we introduce more precise language to speak about concavity properties of the initial condition. Indeed, one can endow the set of probability measures with \emph{two} different geometric structures. Perhaps the more immediate one is to think of it as an affine subspace of the space of signed measures. In this point of view, the natural ``straight line'' between the measures $\mu$ and~$\nu$ is given by $t \mapsto (1-t) \mu + t \nu$. The second relevant geometric structure on the space of probability measures is that given by optimal transport. In this second point of view, the natural ``straight line'' between the measures $\mu$ and $\nu$ can be seen as the set of laws of the random variables $(1-t) X_\mu + t X_\nu$, with $t$ varying in $[0,1]$, and where the law of $(X_\mu,X_\nu)$ is an optimal coupling between the measures $\mu$ and $\nu$ (since we are only concerned with one-dimensional measures here, the coupling given by \eqref{e.def.Xnu} is optimal).

These two points of view give rise to two different notions of convexity, which we will call ``affine convexity'' and ``transport convexity'' respectively. (The notion of ``transport convexity'' is sometimes also called ``displacement convexity''.) The subtelty here is that, at least in the simpler setting of mixed $p$-spin models, the initial condition in \eqref{e.hj} is affine-concave, as was shown in \cite{aufche}; 
 but, whether for these $p$-spin models or for the bipartite model investigated here, this initial condition is \emph{not} transport-concave (nor transport-convex). And, since the derivatives in \eqref{e.hj} are transport-type derivatives, it is the notion of transport concavity (or convexity) that would have been required to guarantee saddle-point formulas by the general mechanism described above.

In the remainder of this section, we examine more precisely what  natural attempts at writing saddle-point formulas for the solution to \eqref{e.hj} may look like, and explain why these attempts fail in general (although we do not exclude the possibility that they be valid for some specific choices of the measures $\pi_1$ and $\pi_2$ in \eqref{e.ass.P}).

\subsection{Attempts based on the Hopf formula}

We start by arguing that the initial condition $\psi$ in \eqref{e.hj} is neither transport-concave nor transport-convex in general. This observation is also valid for models with a single type such as mixed $p$-spin models. The transport concavity (or convexity) of the mapping $\mu \mapsto \psi(\mu)$ would imply in particular that the mapping
\begin{equation}  
\label{e.map.psi}
\chi : \Ll\{
\begin{array}{rcl}  %\label{}
\R_+ & \to & \R \\
h & \mapsto & \psi \Ll( (\de_h,\de_0) \Rr) 
\end{array}
\Rr.
\end{equation}
is concave (or convex). Recall from \eqref{e.def.psi} that 
\begin{equation*}  %\label{e.}
\chi(h) = \psi((\de_h,\de_0)) = -\E \log \int \exp \Ll( (2h)^\frac 1 2 z_1 \si_1 - h(\si_1)^2   \Rr) \, \d \pi_1(\si_1),
\end{equation*}
where here $\si_1$ is real-valued, and $z_1$ is a standard one-dimensional Gaussian random variable. Dropping the subscript ``$1$'' on $z_1$ and $\si_1$ to lighten the notation, and denoting by  $\la \cdot \ra$ the corresponding Gibbs measure, we have
\begin{equation*}  %\label{e.}
\dr_h \chi = \E \la \si  \si' \ra,
\end{equation*}
and
\begin{align*}  %\label{e.}
\dr_h^2 \chi & = \E \la \si \si' \Ll( (2h)^{-\frac 1 2} z(\si+\si') - \si^2 - (\si')^2 \Rr) \ra 
- 2\E \la \si \si' \Ll( (2h)^{-\frac 1 2} z\si'' - (\si'')^2 \Rr) \ra
\\
& = \E \la \si \si' \Ll( (\si + \si')(\si + \si' - 2\si'') - \si^2 - (\si')^2 \Rr) \ra 
\\
& \qquad - 2 \E \la \si \si' \Ll( \si'' (\si + \si' + \si'' - 3 \si''') - (\si'')^2 \Rr) \ra
\\
& = 2 \E \Ll[ \la \si^2 \ra^2 - 4 \la \si^2 \ra \la \si \ra^2  + 3 \la \si \ra^4 \Rr] 
\\
& = 2 \E \Ll[ \Ll( \la \si^2 \ra - \la \si \ra^2 \Rr) \Ll( \la \si^2 \ra - 3 \la \si \ra^2 \Rr)  \Rr] .
\end{align*}
Recall also that when $h = 0$, the Gibbs measure simplifies into being the measure $\pi_1$. It is therefore clear that we can choose the measure $\pi_1$ in such a way that $\dr_h^2 \chi (h = 0)$ has any desired sign: for instance, if $\pi_1$ is the uniform measure on $\{-1,1\}$, then $\la \si \ra = 0$ at $h = 0$, so $\dr_h^2 \chi > 0$; but if we choose $\pi_1$ to be the probability measure on $\{-1,1\}$ such that $\la \si \ra^2 = \frac 1 2$ at $h = 0$, then we have $\la \si^2 \ra = 1 < 3 \la \si \ra^2 = \frac 3 2$, and thus $\dr_h^2 \chi < 0$ at $h = 0$. In both examples, we also have that $\dr_h \chi$ tends to $1$ as $h$ tends to infinity. In the case with $\la \si \ra^2 = \frac 1 2$ at $h = 0$, the derivative $\dr_h \chi$ at $h = 0$ is $\frac 1 2$ and then decreases, but must then tend to $1$. In particular, the function $\dr_h \chi$ is not monotone: that is, the function $\chi$ is neither concave nor convex. 

In a possibly confusing twist, for the most studied case in which $\pi_1$ is the uniform measure on $\{-1,1\}$, one can show that the function $\chi$ is in fact convex. This implies that, at least for the model with a single type, the replica-symmetric solution for this specific choice of measure can in fact be written as a saddle-point problem. But, as is argued here, this is an accident rather than the rule. My understanding is that the solution proposed in \cite{korshe, fyo1, fyo2} is based on this coincidence. As a side note, it is also worth mentioning that the de Ameida-Thouless-type stability criterion employed there is known to be invalid in general, even for models with a single type \cite{pan.gen}. 

As was recalled in \eqref{e.hj.xi} (see also \cite{parisi}), the limit free energy of mixed $p$-spin models can be expressed in terms of the solution $f_1 = f_1(t,\mu) : \R_+ \times \mcl P_2(\R_+) \to \R$ of the equation
\begin{equation}  
\label{e.pde.mixed}
\Ll\{
\begin{aligned}
& \dr_t f_1 - \int \xi(\dr_\mu f_1) \, \d \mu = 0  & \quad \text{on } \ \R_+ \times \mcl P_2(\R_+),
\\
& f_1(0,\cdot ) = \psi_1 & \quad \text{on } \ \mcl P_2(\R_+),
\end{aligned}
\Rr.
\end{equation}
and this solution can be written in variational form using the Hopf-Lax formula: we have
\begin{equation}  
\label{e.hopf-lax}
f_1(t,\mu) = \sup_{\nu \in \mcl P_2(\R_+)} \Ll( \psi_1(\nu) - t \E \Ll[ \xi^* \Ll( \frac{X_\nu - X_\mu}{t} \Rr)  \Rr]   \Rr) ,
\end{equation}
where $X_\mu$, $X_\nu$ are defined according to \eqref{e.def.Xnu}, and
\begin{equation*}  %\label{e.}
\xi^*(s) := \sup_{r \ge 0} \Ll( r s - \xi(r) \Rr) .
\end{equation*}
We can rewrite this formula as
\begin{align}  %\label{e.}
f_1(t,\mu) &  =  \sup_{\nu \in \mcl P_2(\R_+)} \Ll( \psi_1(\nu) - \E \Ll[ \sup_{r \ge 0} \Ll\{ r (X_\nu - X_\mu) - t \xi(r) \Rr\}\Rr]   \Rr) 
\\
\label{e.nuf}
& =  \sup_{\nu \in \mcl P_2(\R_+)} \inf_{f \in L^2([0,1];\R_+)} \Ll( \psi_1(\nu) - \E \Ll[f(U) (X_\nu - X_\mu) - t \xi(f(U)) \Rr]   \Rr) .
\end{align}
One may wonder whether supremum and infimum can be interchanged in the expression above. If this were the case, it would imply in particular that 
\begin{equation*}  %\label{e.}
\psi_1(\mu) = \inf_{f \in L^2([0,1];\R_+)}\sup_{\nu \in \mcl P_2(\R_+)}  \Ll( \psi_1(\nu) - \E \Ll[f(U) (X_\nu - X_\mu) \Rr]   \Rr) .
\end{equation*}
But notice that the supremum over $\nu$ above is an affine function of $X_\mu$; taking the infimum, we find that this would imply the transport concavity of $\psi_1$. But we have argued above that this is not so in general. Similarly, replacing $\inf_f \sup_{\nu}$ by $\sup_f \inf_{\nu}$ in the expression above would lead to the conclusion that $\psi_1$ is transport-convex, which has also been excluded in general. Conversely, if $\psi_1$ were actually transport-concave, then interchanging the supremum and the infimum in \eqref{e.nuf} would be valid; and in general, what we find after the interchange is the solution to the same Hamilton-Jacobi equation, but with the initial condition replaced by its transport-concave envelope.

We now come back to the bipartite model investigated in the present paper. The considerations above raise the question of whether the solution to \eqref{e.hj} can be written as a saddle-point, with respect to the variables $f = (f_1, f_2) \in (L^2([0,1];\R_+))^2$ and $\nu = (\nu_1,\nu_2) \in (\mcl P_2(\R_+))^2$, of the functional
\begin{equation*}  %\label{e.}
\psi(\nu) - \E \Ll[\sum_{a = 1}^2 f_a(U)(X_{\nu_a} - X_{\mu_a}) - t f_1(U) f_2(U)\Rr] .
\end{equation*}
But any possible arrangement of $\inf$'s and $\sup$'s leads to a contradiction. If we aim for optimizing first over $f$ and then over $\nu$, in analogy with \eqref{e.nuf}, then this amounts to trying to write down a Hopf-Lax formula although the nonlinearity in the equation is neither convex nor concave. The convex (or concave) dual of the mapping $(x,y) \mapsto xy$ is so degenerate that it is easy to rule out this possibility. On the other hand, if we try to optimize first over $\nu$ and then over $f$, then we face the same situation as above: each possibity would imply either that $\psi$ is transport-convex, or that it is transport-concave, and both have been ruled out in general.

\subsection{A related attempt}
A related attempt at generating a candidate variational formula for the limit free energy of the bipartite model is as follows. In the papers \cite{barcon, pan.multi}, the authors investigate a large class of models covering in particular the situation in which the definition of $H_N(\si)$ in \eqref{e.def.HN} is replaced by
\begin{equation*}  %\label{e.}
N^{-\frac 1 2} \sum_{i,j = 1}^N J_{ij} \,  \sigma_{1,i} \, \si_{2,j} + N^{-\frac 1 2} \sum_{a \in \{1,2\}} \sum_{i,j = 1}^N  J^{(a)}_{ij} \,  \sigma_{a,i} \, \si_{a,j},
\end{equation*}
where $(J_{ij}^{(a)})_{a \in \{1,2\}, 1\le i,j \le N}$ are independent centered Gaussian random variables with a fixed variance, independent of $(J_{ij})$. Assuming that the matrix
\begin{equation*}  
%\label{e.matrix.delta}
A := \begin{pmatrix}  %\label{}
2\E[(J_{11}^{(1)})^2] & \E[(J_{11})^2] \\
\E[(J_{11})^2] & 2 \E[(J^{(2)}_{11})^2] \\
\end{pmatrix}
\end{equation*}
is positive definite, they derive a variational formula for the free energy of the model. 

One may wonder whether the formula obtained by ignoring the assumption of positive definiteness of the matrix $A$ necessary for their proofs actually matches the prediction given by the Hamilton-Jacobi equation. We will argue here that this is not so. 
We have already seen in the previous subsection that writing up a naive Hopf-Lax formula for \eqref{e.hj} would clearly lead to an invalid prediction. However, the formula given in \cite{barcon,pan.multi}, while equivalent to the Hopf-Lax formula in the case when the matrix $A$ is positive definite, is actually different in outlook, and extends to a different expression in the setting when the matrix $A$ is taken to be the matrix of interest to us here, namely 
\begin{equation}  
\label{e.choice.A}
\begin{pmatrix}  %\label{}
0 & 1 \\
1 & 0
\end{pmatrix}.
\end{equation}
Assuming that the matrix $A$ is positive definite, we first explain the derivation of the formula in \cite{barcon, pan.multi} starting from the point of view provided by Hamilton-Jacobi equations. One can check (at least formally, and probably rigorously by combining the arguments of \cite{barcon, pan.multi} with those of \cite{HJsoft}) that the relevant Hamilton-Jacobi equation for this model is given by
\begin{equation*}  %\label{e.}
\dr_t f - \frac 1 2 \int \dr_\mu f \cdot A \dr_\mu f \, \d \hat \mu,
\end{equation*}
with the same initial condition $\psi$ defined in \eqref{e.def.psi} (that is, $\psi$ does not depend on $A$), and where we used the vector notation $\dr_\mu f := (\dr_{\mu_1} f, \dr_{\mu_2} f)$. 
Since we assume that $A$ is positive definite, we can write down the Hopf-Lax formula
\begin{equation*}  %\label{e.}
f(t,\mu) = \sup_{\nu \in (\mcl P_2(\R_+))^2} \Ll( \psi(\nu) - \frac 1 {2t} \E \Ll[ (X_\nu - X_\mu) \cdot A^{-1} (X_\nu - X_\mu) \Rr] \Rr) ,
\end{equation*}
with the notation $X_\nu = (X_{\nu_1}, X_{\nu_2})$. From here, we could replace $A$ by the matrix in \eqref{e.choice.A}, but it is easy to see that with this choice the supremum is infinite. However the formula of \cite{barcon, pan.multi} (for positive definite $A$) has an additional restriction on the support of the pair of measures $\nu$. Under the assumption that $A$ is positive definite, we can indeed write
\begin{equation} 
\label{e.sup.de0}
f(t,(\de_0,\de_0)) = \sup_{\nu} \Ll( \psi(\nu) - \frac 1{2t} \E \Ll[ X_\nu \cdot A^{-1} X_\nu \Rr]  \Rr) ,
\end{equation}
where the supremum is taken over every pair of measures $\nu = (\nu_a)_{a \in \{1,2\}} \in (\mcl P(\R_+))^2$ with the restriction that, denoting by $q_a$ the top of the support of $\nu_a$, 
\begin{equation*}  %\label{e.}
A^{-1} \begin{pmatrix}  %\label{}
q_1 \\ q_2
\end{pmatrix} \le \begin{pmatrix}  %\label{}
2t \\ 2t
\end{pmatrix},
\end{equation*}
in the sense that the inequality holds component by component. That this additional restriction does not change the value of the supremum in \eqref{e.sup.de0} in the case when $A$ is positive definite can be derived from the Lipschitz estimate on $\psi$ guaranteed by Proposition~\ref{p.continuity}, and arguing as in \cite[Step~4]{parisi}. Blindly replacing the matrix $A$ by that in~\eqref{e.choice.A} thus leads to the formula
\begin{equation}
\label{e.pierro}
\sup_{\nu \in \mcl P([0.2t])^2} \Ll( \psi(\nu)- \frac 1 t \E \Ll[ X_{\nu_1} X_{\nu_2} \Rr]  \Rr) .
\end{equation}
One can verify that this formula does not match the solution to the equation \eqref{e.hj} evaluated at $\mu = (\de_0,\de_0)$. 
For instance, recalling the notation in \eqref{e.ass.P}, we may take $\pi_1$ to be the uniform measure on $\{-1,1\}$, $\pi_2$ to be a non-uniform measure on $\{-1,1\}$, and verify that in this case the solution to the equation satisfies $\dr_t f(0,(\de_0,\de_0)) = 0$, but that this property is not satisfied by the expression in \eqref{e.pierro}. Assuming that the overlaps are concentrated for small $t$, the limit free energy should be described by the equation \eqref{e.naive} in this region, and this would imply that indeed $\dr_t f(0,(\de_0,\de_0)) = 0$.

\appendix

\section{Gaussian integrals}

\subsection{Gaussian integration by parts}

Let $\mu$ be the law of a $d$-dimensional centered Gaussian vector, with covariance matrix $\msf C \in \R^{d \times d}$. We assume (temporarily) that $\msf C$ is invertible. In this case, the measure $\mu$ has a density with respect to the Lebesgue measure on $\R^d$, which is proportional to 
\begin{equation*}  %\label{e.}
\exp \Ll( - \frac 1 2 x \cdot \msf C^{-1} x \Rr) .
\end{equation*}
For every bounded and smooth function $F \in C^\infty(\R^d;\R^d)$, we thus have, by integration by parts,
\begin{equation*}  %\label{e.}
\int \msf C^{-1} x \cdot F(x) \, \d \mu(x) = \int \nabla \cdot F(x) \, \d \mu(x),
\end{equation*}
or equivalently,
\begin{equation*}  
%\label{e.gibp}
\int x \cdot F(x) \, \d \mu(x) = \int \nabla \cdot (\msf C F)(x) \, \d \mu(x).
\end{equation*}
This last identity remains valid when $\msf C$ is not invertible, by approximation. In particular, for every bounded and smooth $f \in C^\infty(\R^d;\R)$,
\begin{equation}  
\label{e.gibp}
\int x_1 f(x) \, \d \mu(x) = \sum_{\ell = 1}^d \E \Ll[ x_1 x_\ell \Rr] \int \dr_{x_\ell} f(x) \, \d \mu(x).
\end{equation}
One consequence of this observation is the following result. (At least the first part of it is very classical; the last part is certainly also well-known, but I could not find a precise reference). 
\begin{lemma}
\label{l.gibp}
Let $\Sigma$ be a finite set, let $(x_1(\si), x_2(\si))_{\si \in \Sigma}$ be a centered Gaussian random field with respect to the probability measure $\P$ (with expectation $\E$), and let $P$ be a probability measure on $\Sigma$. For every $a,b \in \{1,2\}$ and $\si,\si' \in \Sigma$, we write
\begin{equation*}  %\label{e.}
\msf C_{ab}(\si,\si') := \E \Ll[ x_a(\si) x_b(\si') \Rr] .
\end{equation*}
We denote by $\la \cdot \ra$ the Gibbs measure built from $(x_2(\si))$, so that for every $f : \Sigma \to \R$,
\begin{equation*}  %\label{e.}
\la f(\si) \ra := \frac{ \int f(\si) \exp \Ll( x_2(\si) \Rr) \, \d P(\si) } { \int \exp \Ll( x_2(\si) \Rr) \, \d P(\si) },
\end{equation*}
and write $\si',\si''$ for independent copies of the random variable $\si$ under $\la \cdot \ra$. We have
\begin{equation}  
\label{e.gibp1}
\E \la x_1(\si) \ra = \E \la \msf C_{12} (\si,\si) - \msf C_{12} (\si,\si') \ra,
\end{equation}
and
\begin{multline}  
\label{e.gibp2}
\E \la x_1^2(\si) \ra = \E \la \msf C_{11}(\si,\si) \ra + \E \big\langle \Ll(\msf C_{12}(\si,\si) - \msf C_{12}(\si,\si') \Rr)
\\
\Ll(\msf C_{12}(\si,\si) + \msf C_{12}(\si,\si') - 2\msf C_{12}(\si,\si'')  \Rr) \big\rangle.
\end{multline}
%\begin{multline}  
%\label{e.gibp3}
%\E \la x_0(\si)x_1(\si) \ra = \E \la \msf C_{01}(\si,\si) \ra + \E \big\langle \Ll(\msf C_{02}(\si,\si) - \msf C_{02}(\si,\si') \Rr)
%\\
%\Ll(\msf C_{12}(\si,\si) + \msf C_{12}(\si,\si') - 2\msf C_{12}(\si,\si'')  \Rr) \big\rangle.
%\end{multline}
More generally, we write $(\si^\ell)_{\ell \ge 1}$ for a sequence of independent copies of the random variable $\si$ under $\la \cdot \ra$.  For every $p \ge 1$, there exists a polynomial $P_p$ (which does not depend on any parameter in the problem) taking as inputs the variables $(\msf C_{ab}(\si^k,\si^{\ell}))_{a,b \in \{1,2\}, k,\ell \ge 1}$ such that 
\begin{equation}  
\label{e.gibpp}
\E \la x_1^p(\si) \ra = \E \la P_p \Ll( \Ll( \msf C_{ab}(\si^k,\si^{\ell})\Rr)_{a,b \in \{1,2\}, k,\ell \ge 1}  \Rr) \ra.
\end{equation}
Moreover, the polynomial only depends on $(\msf C_{11}(\si^k,\si^\ell))_{k,\ell \ge 1}$ and $(\msf C_{12}(\si^k,\si^\ell))_{k,\ell \ge 1}$, and is homogeneous of degree $p$ provided that we count each occurrence of a variable $\msf C_{11}(\si^k,\si^\ell)$ as having degree $2$. 
\end{lemma}
\begin{proof}
We start by writing
\begin{equation*}  %\label{e.}
\E \la x_1(\si) \ra = \int \E \Ll[ x_1(\si) \frac{\exp(x_2(\si))}{\int \exp(x_2(\si')) \, \d P(\si')}  \Rr] \, \d P(\si).
\end{equation*}
We then apply \eqref{e.gibp} to rewrite the inner expectation as
\begin{equation*}  %\label{e.}
\msf C_{12}(\si,\si) \E \Ll[\frac{\exp \Ll( x_2(\si) \Rr)}{\int \exp(x_2(\si')) \, \d P(\si')} \Rr] - \int \msf C_{12}(\si,\si') \E\Ll[\frac{\exp(x_2(\si) + x_2(\si'))}{\Ll(\int \exp(x_2(\si'')) \, \d P(\si'')\Rr)^2}\Rr] \, \d P(\si').
\end{equation*}
Combining the two previous displays leads to \eqref{e.gibp1}. The argument for \eqref{e.gibp2} is similar, except that we now need to compute 
\begin{equation}  
\label{e.exp.square}
\E \Ll[ (x_1(\si))^2 \frac{\exp(x_2(\si))}{\int \exp(x_2(\si')) \, \d P(\si')}  \Rr].
\end{equation}
In order to apply \eqref{e.gibp} in this case, we split the square of $x_1(\si)$ into two parts, one of them being incorporated into the function ``$f$'' in \eqref{e.gibp}. We thus find that the quantity in \eqref{e.exp.square} equals
\begin{multline*}  %\label{e.}
\msf C_{11}(\si,\si) \E \Ll[\frac{\exp \Ll( x_2(\si) \Rr)}{\int \exp(x_2(\si')) \, \d P(\si')} \Rr] 
+ \msf C_{12}(\si,\si)\E\Ll[\frac{x_1(\si)  \exp \Ll( x_2(\si) \Rr)}{\int \exp(x_2(\si')) \, \d P(\si')} \Rr]
\\
 - \int \msf C_{12}(\si,\si') \E\Ll[\frac{x_1(\si) \exp(x_2(\si) + x_2(\si'))}{\Ll(\int \exp(x_2(\si'')) \, \d P(\si'')\Rr)^2} \Rr]\, \d P(\si').
\end{multline*}
This shows that
\begin{equation*}  %\label{e.}
\E \la (x_1(\si))^2 \ra = \E \la \msf C_{11}(\si,\si) \ra + \E \la  x_1(\si) \Ll( \msf C_{12}(\si,\si) - \msf C_{12}(\si,\si')\Rr)  \ra.
\end{equation*}
For every $\si, \si' \in \Sigma$, we define
\begin{equation*}  %\label{e.}
\td x_1(\si,\si') := x_1(\si) \Ll( \msf C_{12}(\si,\si) - \msf C_{12}(\si,\si')\Rr).
\end{equation*}
The variables $(\td x_1(\si,\si'), x_2(\si) + x_2(\si'))_{\si,\si' \in \Sigma}$ form a centered Gaussian field, with
\begin{align*}  %\label{e.}
\E \Ll[ \td x_1(\si,\si') (x_2(\si'') + x_2(\si''')) \Rr] 
&  = \Ll( \msf C_{12}(\si,\si) - \msf C_{12}(\si,\si')\Rr)\Ll( \msf C_{12}(\si,\si'') + \msf C_{12}(\si,\si''')\Rr).
\end{align*}
Applying \eqref{e.gibp1}, we deduce that 
\begin{multline*}  %\label{e.}
\E \la  x_1(\si) \Ll( \msf C_{12}(\si,\si) - \msf C_{12}(\si,\si')\Rr)  \ra = \E \big\langle \Ll(\msf C_{12}(\si,\si) - \msf C_{12}(\si,\si') \Rr)
\\
\Ll(\msf C_{12}(\si,\si) + \msf C_{12}(\si,\si') - \msf C_{12}(\si,\si'') - \msf C_{12}(\si,\si''') \Rr) \big\rangle,
\end{multline*}
and replacing $\si'''$ by $\si''$ in the expression above does not change its value. This completes the proof of \eqref{e.gibp2}. For~\eqref{e.gibpp}, we apply \eqref{e.gibp} again to rewrite
\begin{equation*}  %\label{e.}
\E \Ll[ x_1^p(\si) \frac{\exp(x_2(\si))}{\int \exp(x_2(\si')) \, \d P(\si')}  \Rr]
\end{equation*}
as
\begin{multline*}  %\label{e.}
\msf C_{11}(\si,\si) (p-1) \E \Ll[\frac{x_1^{p-2}(\si) \exp \Ll( x_2(\si) \Rr)}{\int \exp(x_2(\si')) \, \d P(\si')} \Rr] 
+ \msf C_{12}(\si,\si)\E\Ll[\frac{x_1^{p-1}(\si)  \exp \Ll( x_2(\si) \Rr)}{\int \exp(x_2(\si')) \, \d P(\si')} \Rr]
\\
 - \int \msf C_{12}(\si,\si') \E\Ll[\frac{x_1^{p-1}(\si) \exp(x_2(\si) + x_2(\si'))}{\Ll(\int \exp(x_2(\si'')) \, \d P(\si'')\Rr)^2} \Rr]\, \d P(\si').
\end{multline*}
We can then obtain \eqref{e.gibpp} by induction on $p$. 
\end{proof}

\subsection{Existence of Gaussian process}

The next lemma serves to guarantee that the Gaussian random field introduced in \eqref{e.cov.HNah} indeed exists.
\begin{lemma}
\label{l.exist.HNah}
Let $p ,k,N \ge 1$ be integers, and $\lambda_1,\lambda_2 \ge 0$. There exists a centered Gaussian field $(X(\si,\al))_{\si \in \R^N, \al \in \N^k}$ such that, for every $\si,\si' \in \R^N$ and $\al,\al' \in \N^k$,
\begin{equation}  
\label{e.exist.HNah}
\E \Ll[ X(\si,\al) X(\si',\al') \Rr] = \Ll( \lambda_1 \, \si \cdot \si' + \lambda_2 \, \al \wedge \al' \Rr)^p,
\end{equation}
where we recall that the notation $\al \wedge \al'$ was introduced in \eqref{e.def.wedge}. 
\end{lemma}
\begin{proof}
Recall the definition of the tree $\mcl A$ in \eqref{e.def.mclA}. For each $n \in \N$, we define the finite approximation
\begin{equation*}  %\label{e.}
\mcl A_n := \{0,\ldots,n\}^{0} \cup \cdots \cup \{0,\ldots,n\}^{k} ,
\end{equation*}
again with the understanding that $\{0,\ldots,n\}^{0} = \{\emptyset\}$, and we denote the set of leaves by $\mcl L_n := \{0,\ldots,n\}^{k}$.
By Kolmogorov's extension theorem, it suffices to construct a Gaussian process $(X_n(\si,\al))_{\si \in \R^N, \al \in \mcl L_n}$ such that \eqref{e.exist.HNah} holds for every $\si,\si' \in \R^N$ and $\al,\al' \in \mcl L_n$. Let $(f_\al)_{\al \in \mcl A_n}$ be an orthonormal basis of $\R^{|\mcl A_n|}$, and for each $\al \in \mcl L_n$, let
\begin{equation*}  %\label{e.}
g_\al := \sum_{\ell = 1}^k f_{\al_{|\ell}},
\end{equation*}
so that for every $\al,\al' \in \mcl L_n$,
\begin{equation}  
\label{e.product.g}
g_\al \cdot g_{\al'} = \al \wedge \al'.
\end{equation}
Viewing $(\sqrt{\lambda_1} \si, \sqrt{\lambda_2} g_\al)$ as a vector in $\R^N \times \R^{|\mcl A_n|}$, we consider the $p$-fold tensor product
\begin{equation*}  %\label{e.}
\Ll(\sqrt{\lambda_1} \si, \sqrt{\lambda_2} g_\al\Rr) ^{\otimes p} \in \Ll(\R^N \times \R^{|\mcl A_n|}\Rr)^{\otimes p}.
\end{equation*}
Recall that, if we denote by $(e_i)_{i \in \{1,\ldots,N\}}$ an orthonormal basis of $\R^N$, then an orthonormal basis of the tensor product $(\R^N \times \R^{|\mcl A_n|})^{\otimes p}$ is given by
\begin{equation*}  %\label{e.}
\mcl B := \Ll\{ v_1 \otimes \cdots \otimes v_p \ : \ v_1, \ldots v_p \in \{e_i, i \in \{1,\ldots, N\} \} \cup \{f_\al, \al \in \mcl A_n \} \Rr\}.
\end{equation*}
We now give ourselves a standard Gaussian vector $W$ taking values in $\R^{|\mcl B|}$, and define
\begin{equation*}  %\label{e.}
X(\si,\al) := W \cdot \Ll(\sqrt{\lambda_1} \si, \sqrt{\lambda_2} e'_\al\Rr) ^{\otimes p},
\end{equation*}
so that for every $\si,\si' \in \R^N$ and $\al,\al' \in \mcl L_n$,
\begin{align*}  %\label{e.}
\E \Ll[ X(\si,\al) X(\si',\al') \Rr] 
& = \Ll(\sqrt{\lambda_1} \si, \sqrt{\lambda_2} g_\al\Rr) ^{\otimes p} \cdot \Ll(\sqrt{\lambda_1} \si', \sqrt{\lambda_2} g_{\al'}\Rr) ^{\otimes p}
\\
& = \Ll(\Ll(\sqrt{\lambda_1} \si, \sqrt{\lambda_2} g_\al\Rr) \cdot \Ll(\sqrt{\lambda_1} \si', \sqrt{\lambda_2} g_{\al'}\Rr) \Rr)^{p}
\\
& = \Ll(\lambda_1 \, \si \cdot \si' + \lambda_2 \, \al \wedge \al'\Rr)^{p},
\end{align*}
where we used \eqref{e.product.g} in the last step. 
\end{proof}

\noindent \textbf{Acknowledgements.} I was partially supported by the NSF grant DMS-1954357.

\small
\bibliographystyle{abbrv}
\bibliography{bipartite}

\end{document}